\documentclass[11pt,twoside]{article}
\newcommand{\ifims}[2]{#1} 
\newcommand{\ifAMS}[2]{#1}   
\newcommand{\ifau}[3]{#1}  
\newcommand{\ifbook}[2]{#1}   

\ifbook{

  }
  {

  }

\def\thetitle{
Simultaneous likelihood-based bootstrap confidence sets\\
 for a large number of models}
\def\theruntitle{Simultaneous bootstrap confidence sets}

\def\theabstract{
The paper studies a problem of constructing simultaneous likelihood-based confidence sets.
We consider a simultaneous multiplier bootstrap procedure for estimating the quantiles of the joint distribution of the likelihood ratio statistics, and for adjusting the confidence level for multiplicity.
Theoretical results state the bootstrap validity in the following setting: the sample size \(n\) is fixed, the maximal parameter dimension \(p_{\textrm{max}}\) and the number of considered parametric models \(K\) are s.t. \((\log K)^{12}p_{\max}^{3}/n\) is small.
We also consider the situation when the parametric models are misspecified. If the models' misspecification is significant, then the bootstrap critical values exceed the true ones and the simultaneous bootstrap confidence set becomes conservative. Numerical experiments for local constant and local quadratic regressions illustrate the theoretical results.
}

\def\kwdp{62F25}
\def\kwds{62F40, 62E17}

\def\thekeywords{simultaneous inference, correction for multiplicity, family-wise error, misspecified model, multiplier/weighted bootstrap
}

\def\authora{Mayya Zhilova}
\def\thanksa
{I am very grateful to Vladimir Spokoiny for many helpful discussions and comments.}
\def\thanksaa{
Financial support by the German Research Foundation (DFG) through the Collaborative
Research Center 649 ``Economic Risk'' is gratefully acknowledged. The research was partly supported by the Russian Science Foundation grant (project 14-50-00150).
}

\def\runauthora{zhilova, m.}
\def\addressa{
    Weierstrass-Institute,
    \\
    Mohrenstr. 39,\\ 10117 Berlin, Germany,
    }
\def\emaila{zhilova@wias-berlin.de}
\def\affiliationa{Weierstrass-Institute}
\def\authorb{}

\usepackage{multicol}
\usepackage{mathtools}
\usepackage{multirow}
\usepackage{hyperref}
\hypersetup{
            colorlinks=true,
            linkcolor=hookersgreen,
            linktoc=all,
            citecolor=ultramarine,
            urlcolor=black,
            filecolor=black
            }

\usepackage{nameref}
\usepackage{bigstrut}
\usepackage{etoolbox}
\usepackage{xcolor}
\usepackage{fancybox}
\usepackage{pseudocode}
\usepackage{enumitem}

\usepackage{amsmath,amssymb,amsthm}
\usepackage{natbib}
\usepackage{epsfig,graphicx}
\usepackage{comment}
\usepackage{color}
\usepackage{srcltx}
\usepackage[mathscr]{eucal}
\usepackage[math]{easyeqn}
\usepackage{etoolbox}

\ifims{
\textheight=23cm
\textwidth=14.8cm
\topmargin=0pt
\oddsidemargin=1.0cm
\evensidemargin=1.0cm
\linespread{1.3}
\renewenvironment{abstract}
    {\centerline{\textbf{Abstract}}\bigskip
      \begin{center}
       \begin{minipage}{12cm}
        \begin{small}
    }
    {   \end{small}
       \end{minipage}
      \end{center}
     \bigskip
    }

}{ 
}

\numberwithin{equation}{section}
\numberwithin{figure}{section}
\newcounter{example}[section]
\numberwithin{example}{section}
\newcounter{remark}[section]
\numberwithin{remark}{section}
\newtheorem{theorem}{Theorem}[section]
\newtheorem{proposition}[theorem]{Proposition}
\newtheorem{lemma}[theorem]{Lemma}
\newtheorem{corollary}[theorem]{Corollary}

\newtheorem{exmp}[example]{Example}
\newtheorem{rmrk}[remark]{Remark}
\newenvironment{example}{\begin{exmp}\rm}{\end{exmp}}
\newenvironment{remark}{\begin{rmrk}\rm}{\end{rmrk}}

\begin{document}
\thispagestyle{empty}
\ifims{

\title{\thetitle}
\ifau{ 
  \author{
    \authora
    \ifdef{\thanksa}{\thanks{\thanksa}\ $^{,}$\thanks{\thanksaa}}{}
    \\[5.pt]
    \addressa \\
    \texttt{ \emaila}
  }
}
{  
  \author{
    \authora
    \ifdef{\thanksa}{\thanks{\thanksa}}{}
    \\[5.pt]
    \addressa \\
    \texttt{ \emaila}
    \and
    \authorb
    \ifdef{\thanksb}{\thanks{\thanksb}}{}
    \\[5.pt]
    \addressb \\
    \texttt{ \emailb}
  }
}
{   
  \author{
    \authora
    \ifdef{\thanksa}{\thanks{\thanksa}}{}
    \\[5.pt]
    \addressa \\
    \texttt{ \emaila}
    \and
    \authorb
    \ifdef{\thanksb}{\thanks{\thanksb}}{}
    \\[5.pt]
    \addressb \\
    \texttt{ \emailb}
    \and
    \authorc
    \ifdef{\thanksc}{\thanks{\thanksc}}{}
    \\[5.pt]
    \addressc \\
    \texttt{ \emailc}
  }
}

\maketitle
\pagestyle{myheadings}
\markboth
 {\hfill \textsc{ \small \theruntitle} \hfill}
 {\hfill
 \textsc{ \small
 \ifau{\runauthora}
      {\runauthora\ and \runauthorb}
      {\runauthora, \runauthorb, and \runauthorc}
 }
 \hfill}
\begin{abstract}
\theabstract
\end{abstract}

\ifAMS
    {\par\noindent\emph{AMS 2000 Subject Classification:} Primary \kwdp. Secondary \kwds}
    {\par\noindent\emph{JEL codes}: \kwdp}

} 
{ 
\begin{frontmatter}
\title{\thetitle\protect\thanksref{T1}}
\thankstext{T1}{\thankstitle}


\runtitle{\theruntitle}

\begin{aug}
\author{\authora\ead[label=e1]{\emaila}}
\address{\addressa \\
 \printead{e1}}
 \end{aug}

 \runauthor{\runauthora}
\affiliation{\affiliationa}

\begin{abstract}
\theabstract
\end{abstract}



\end{frontmatter}
} 

\renewcommand{\(}{$\,}
\renewcommand{\)}{\,$}

\def\nquad{\hspace{-1cm}}
\def\eqdef{\stackrel{\operatorname{def}}{=}}
\def\tod{\stackrel{d}{\longrightarrow}}
\def\tow{\stackrel{w}{\longrightarrow}}
\def\toP{\stackrel{\P}{\longrightarrow}}

\newcommand{\cc}[1]{\mathscr{#1}}
\newcommand{\bb}[1]{\boldsymbol{#1}}

\renewcommand{\bar}[1]{\overline{#1}}
\renewcommand{\hat}[1]{\widehat{#1}}
\renewcommand{\tilde}[1]{\widetilde{#1}}

\renewcommand{\Gamma}{\varGamma}
\renewcommand{\Pi}{\varPi}
\renewcommand{\Sigma}{\varSigma}
\renewcommand{\Delta}{\varDelta}
\renewcommand{\Lambda}{\varLambda}
\renewcommand{\Psi}{\varPsi}
\renewcommand{\Phi}{\varPhi}
\renewcommand{\Theta}{\varTheta}
\renewcommand{\Omega}{\varOmega}
\renewcommand{\Xi}{\varXi}
\renewcommand{\Upsilon}{\varUpsilon}
\def\nn{\nonumber \\}

\def\suml{\sum\limits}
\def\supl{\sup\limits}
\def\maxl{\max\limits}
\def\infl{\inf\limits}
\def\intl{\int\limits}
\def\liml{\lim\limits}
\def\Cov{\operatorname{Cov}}
\def\Var{\operatorname{Var}}
\def\arginf{\operatornamewithlimits{arginf}}
\def\argsup{\operatornamewithlimits{argsup}}
\def\argmax{\operatornamewithlimits{argmax}}
\def\argmin{\operatornamewithlimits{argmin}}
\def\val{\operatorname{val}}

\def\D{\boldsymbol{D}}
\def\dd{\operatorname{d}}
\def\tr{\operatorname{tr}}
\def\I{I\!\!I}
\def\R{I\!\!R}
\def\E{I\!\!E}
\def\P{I\!\!P}
\def\X{\mathfrak{X}}
\def\kappa{\varkappa}
\def\Const{\mathrm{Const.} \,}
\def\cdt{\boldsymbol{\cdot}}
\def\tm{\!\times\!}
\def\T{\top}
\def\diag{\operatorname{diag}}
\def\diam{\operatorname{diam}}
\def\rank{\operatorname{rank}}
\def\loc{\operatorname{loc}}

\def\av{\bb{a}}
\def\bv{\bb{b}}
\def\cv{\bb{c}}
\def\dv{\bb{d}}
\def\ev{\bb{e}}
\def\fv{\bb{f}}
\def\gv{\bb{g}}
\def\hv{\bb{h}}
\def\iv{\bb{i}}
\def\jv{\bb{j}}
\def\kv{\bb{k}}
\def\lv{\bb{l}}
\def\mv{\bb{m}}
\def\nv{\bb{n}}
\def\ov{\bb{o}}
\def\pv{\bb{p}}
\def\qv{\bb{q}}
\def\rv{\bb{r}}
\def\sv{\bb{s}}
\def\tv{\bb{t}}
\def\uv{\bb{u}}
\def\vv{\bb{v}}
\def\wv{\bb{w}}
\def\xv{\bb{x}}
\def\yv{\bb{y}}
\def\zv{\bb{z}}

\def\Cv{\bb{C}}
\def\Gv{\bb{G}}
\def\Mv{\bb{M}}
\def\Sv{\bb{S}}
\def\Uv{\bb{U}}
\def\Xv{\bb{X}}
\def\Yv{\bb{Y}}
\def\Zv{\bb{Z}}

\def\alphav{\bb{\alpha}}
\def\epsv{\bb{\varepsilon}}
\def\etav{\bb{\eta}}
\def\gammav{\bb{\gamma}}
\def\varepsilonv{\bb{\varepsilon}}
\def\phiv{\bb{\phi}}
\def\psiv{\bb{\psi}}
\def\tauv{\bb{\tau}}
\def\upsilonv{\bb{\upsilon}}
\def\xiv{\bb{\xi}}
\def\zetav{\bb{\zeta}}

\def\Psiv{\bb{\Psi}}
\def\CONST{\mathtt{C}}

\def\itemv{\vfill\item}
\newenvironment{myslide}[1]
    {\begin{frame}\frametitle{#1}\vfill}
    {\vfill\end{frame}}

\def\vsp{\vspace{0.05\textheight} \vfill}
\def\summarysign{\resizebox{0.08\textwidth}{0.08\textheight}{\includegraphics{summary}}\,}
\def\nix{}
\def\wpu{$\bullet$}

\def\btri{\vfill{\( \blacktriangleright \) }}
\def\btrir{\vfill{\( \blacktriangleright \) }}

\newcommand{\mygraphics}[3]{\begin{center}
    \resizebox{#1\textwidth}{#2\textheight}{\includegraphics{#3}}
    \end{center}
}

\newcommand{\mybox}[3]{\begin{center}
    \resizebox{#1\textwidth}{#2\textheight}{#3}
    \end{center}
}

\newenvironment{eqnh}
{
    \setbeamercolor{postit}{fg=black,bg=hellgelb} 
    \begin{beamercolorbox}[center,wd=\textwidth]{postit} 
    \begin{eqnarray*}}
    {\end{eqnarray*}\end{beamercolorbox}
}

\def\ND{\cc{N}}
\def\Bernoulli{\mathrm{Bernoulli}}
\def\Vola{\mathrm{Vola}}
\def\Poisson{\mathrm{Poisson}}
\def\ag{\mathrm{ag}}
\def\glob{\operatorname{glob}}
\def\blk{\operatorname{block}}
\def\lin{\operatorname{lin}}
\def\cond{\, \big| \,}

\def\rdl{\epsilon}
\def\rd{\bb{\rdl}}
\def\rddelta{\delta}
\def\rdomega{\varrho}
\def\rddeltab{\rddelta^{*}}
\def\rhorb{\rhor^{*}}

\def\wv{\bb{w}}
\def\varthetav{\bb{\vartheta}}
\def\Lr{\breve{L}}
\def\zetavr{\breve{\zetav}}
\def\etavr{\breve{\etav}}
\def\xivr{\breve{\xiv}}

\def\rdb{\rd}
\def\rdm{\underline{\rdb}}

\def\taub{\tau_{\rdb}}
\def\taum{\tau_{\rdm}}
\def\kappab{\kappa_{\rd}}
\def\deltab{\delta_{\rd}}

\def\taubGP{\tau_{\rdb,\GP}}
\def\taumGP{\tau_{\rdm,\GP}}
\def\kappabGP{\kappa_{\rd,\GP}}
\def\deltabGP{\delta_{\rd,\GP}}
\def\nubm{\nu_{\rd}}
\def\uub{u_{\rd}}
\def\uubGP{u_{\rd,\GP}}
\def\nubmGP{\nu_{\rd, G}}

\def\rG{\rd,\GP}

\def\LinSp{\mathrm{L}}
\def\Id{I\!\!\!I}
\def\Ind{\operatorname{1}\hspace{-4.3pt}\operatorname{I}}

\def\BG{\mathcal{R}}
\def\bg{r}
\def\fmup{\phi}
\def\rg{r}
\def\uc{u_{c}}
\def\muc{\mu_{c}}
\def\mud{\mu_{0}}
\def\xxd{\xx_{0}}
\def\yyd{\yy_{0}}
\def\gmd{\gm_{0}}

\def\ms{m^{*}}
\def\Inv{A}
\def\InvT{\Inv^{\T}}
\def\Invt{\tilde{\Inv}}

\def\ssize{N}
\def\nsize{{n}}

\def\rhor{\omega}

\def\LT{L}
\def\LGP{\LT_{\GP}}
\def\La{\mathbb{L}}
\def\Lab{\La_{\rdb}}
\def\Lam{\La_{\rdm}}

\def\DP{D}
\def\DPc{\DP_{0}}
\def\DPb{\DP_{\rdb}}
\def\DPm{\DP_{\rdm}}

\def\LabGP{\La_{\rdb,\GP}}
\def\LamGP{\La_{\rdm,\GP}}

\def\DPbGP{\DP_{\rdb,\GP}}
\def\DPmGP{\DP_{\rdm,\GP}}
\def\riskbGP{\riskt_{\rdb,\GP}}

\def\gmi{\mathtt{b}}
\def\gmiid{\mathtt{g}_{1}}
\def\kullbi{\Bbbk}
\def\Thetasi{\Theta_{\loc}}
\def\rri{\mathtt{u}}
\def\rris{\rri_{0}}

\def\Ipc{\bb{\mathrm{f}}}
\def\IF{\Bbb{F}}
\def\IFc{\IF_{0}}
\def\IFb{\IF_{\rdb}}
\def\IFm{\IF_{\rdm}}

\def\DF{\cc{D}}
\def\DFc{\DF_{0}}
\def\DFb{\DF_{\rdb}}
\def\DFm{\breve{\DF}_{\rd}}
\def\DFm{\DF_{\rdm}}

\def\DPr{\breve{\DP}}
\def\VF{\cc{V}}
\def\VFc{\VF_{0}}

\def\HHc{\HH_{0}}
\def\HHb{\HH_{\rd}}
\def\HHm{\HH_{\rdm}}

\def\xib{\xi^{*}}
\def\xivb{\xiv_{\rdb}}
\def\xivm{\xiv_{\rdm}}
\def\CAm{\underline{\CA}}
\def\CAb{\CA}

\def\penr{\operatorname{pen}}
\def\pen{\mathfrak{t}}
\def\PEN{\operatorname{PEN}}
\def\RSS{\operatorname{RSS}}
\def\med{\operatorname{med}}

\def\ex{\mathrm{e}}
\def\entrl{\mathbb{Q}}
\def\entrlb{\entrl}
\def\entr{\entrl}

\def\kullb{\cc{K}} 
\def\kullbc{\kullb^{c}}

\def\gm{\mathtt{g}}
\def\gmc{\gm_{c}}
\def\gmb{\gm}
\def\gmbm{\gmb_{1}}

\def\yy{\mathtt{y}}
\def\yyc{\yy_{c}}
\def\xx{\mathtt{x}}
\def\xxc{\xx_{c}}
\def\tc{t_{c}}

\def\alp{\alpha}
\def\alpn{\rho}
\def\gmu{\mathfrak{a}}

\def\losst{\varrho}
\def\loss{\wp}
\def\lossp{u}
\def\closs{g}

\def\riskt{\cc{R}}
\def\emprisk{\ell}
\def\bias{b}
\def\bern{q}

\def\TT{\nsize}

\def\Pone{P}
\def\Pf{\P_{f(\cdot)}}
\def\Ef{\E_{f(\cdot)}}
\def\Ps{\P_{\thetas}}
\def\Es{\E_{\thetas}}
\def\Pu{\P_{\upsilons}}
\def\Eu{\E_{\upsilons}}

\def\Pvs{\P_{\thetavs}}
\def\Evs{\E_{\thetavs}}

\def\UPd{w}
\def\nunup{\nu_{1}}
\def\rru{\rr_{1}}
\def\rups{\rr_{0}}
\def\rupsb{\rups^{*}}
\def\rrf{\rr^{\flat}}

\def\smooths{\mathbb{S}}
\def\smooth{\smooths_{1}}

\def\elli{\bar{\ell}}

\def\K{K}

\def\Psir{\breve{\Psi}}

\def\af{a}
\def\afs{\af^{*}}

\def\kapla{\varkappa}

\newcommand{\mlew}[1]{\tilde{\thetav}_{#1}}
\newcommand{\mlea}[1]{\hat{\thetav}_{#1}}
\newcommand{\mluw}[1]{\tilde{\theta}_{#1}}
\newcommand{\mlua}[1]{\hat{\theta}_{#1}}
\newcommand{\penm}[1]{\boldsymbol{m}_{#1}}

\def\Pdom{\mu_{0}}
\def\PDOM{\bb{\mu}_{0}}
\def\EDOM{\E_{0}}

\def\mk{m}
\def\Mk{\cc{M}}
\def\SV{\cc{S}}

\def\Cs{E}
\def\Csd{\Cs^{\circ}}
\def\Ca{A}
\def\CS{\cc{E}}
\def\CA{\cc{A}}
\def\CAb{\CA_{\rd}}
\def\CAC{\CA_{\CoFu}}

\def\Ccb{m_{\rdb}}
\def\Ccm{m_{\rdm}}
\def\CcbGP{m_{\rdb,\GP}}
\def\CcmGP{m_{\rdm,\GP}}

\def\etas{\eta^{*}}

\def\omegav{\bb{\phi}}
\def\omegavs{\omegav^{*}}
\def\omegavc{\omegav'}

\def\nuvs{\nuv^{*}}
\def\nuvc{\nuv'}

\def\nunu{\nu_{0}}
\def\numu{\nu_{1}}
\def\nupi{\nu^{+}}
\def\nubu{\beta}

\def\nus{\nu}
\def\nusb{\nus}
\def\nusr{\nus^{\bracketing}}
\def\Nusb{\mathbb{N}}
\def\Nusr{\mathbb{N}^{\diamond}}

\def\dist{d}
\def\distd{\mathfrak{a}}

\def\hatk{\kappa}
\def\ko{k^{\circ}}

\def\qqq{\mathfrak{q}}
\def\ppp{{s}}
\def\Cqq{C(\qqq)}
\def\Cqqb{C^{\diamond}(\qqq)}
\def\Crho{C(\mrho)}
\def\Cqqm{\log(4)}
\def\Cqpr{(\qqq \rrp + \dimp / 2)}

\def\Cdima{\mathfrak{e}_{0}}
\def\Cdimb{\mathfrak{e}_{1}}
\def\cdima{\mathfrak{c}_{0}}
\def\cdimb{\mathfrak{c}_{1}}
\def\cdim{\mathfrak{c}}

\def\rdomega{\varrho}
\def\deltaD{\delta}
\def\alphai{\alpha_{1}}
\def\alphaii{\alpha_{2}}
\def\alphaiii{\alpha_{3}}
\def\alphaiv{\alpha_{4}}

\def\err{\diamondsuit}
\def\errbm{\bar{\err}_{\rdomega}}
\def\errm{\err_{\rdm}}
\def\errb{\err_{\rdb}}

\def\errbGP{\err_{\rdomega,\GP}}
\def\errmGP{\err_{\rdm,\GP}}
\def\errbmGP{\bar{\err}_{\rd,\GP}}

\def\errs{\err_{\rdomega}^{*}}
\def\deltas{\alpha}

\def\xivbGP{\xiv_{\rdb,\GP}}
\def\xivmGP{\xiv_{\rdm,\GP}}

\def\SP{S}
\def\GP{G}
\def\GPt{\GP_{0}}
\def\GPn{\GP_{1}}
\def\gp{g}
\def\gs{s}

\def\errbGP{\err_{\rdb,\GP}}
\def\errmGP{\err_{\rdm,\GP}}
\def\errpmGP{\err_{\GP}^{\pm}}

\def\LCS{\cc{C}}

\def\DPGP{\DP_{\GP}}
\def\thetavsGP{\thetavs_{\GP}}

\def\LL{\cc{L}}
\def\LLb{\LL^{*}}
\def\LLh{\cc{L}}

\def\YY{\cc{Y}}
\def\LP{L^{\circ}}

\def\modcnrd{\mathfrak{A}}

\def\pens{\pi}
\def\pnn{\mathfrak{g}}
\def\pnnd{\mathfrak{u}}
\def\pnndGP{\pnnd_{\GP}}

\def\confpr{\mathfrak{c}}
\def\confprb{\confpr^{*}}

\def\pn{\pens^{*}}
\def\penInt{\mathfrak{D}}
\def\penH{\mathbb{H}}
\def\pmu{\mathfrak{u}}
\def\Closs{\cc{R}}

\def\dimp{p}
\def\riskb{\riskt_{\rdb}}
\def\dimpp{\dimp+1}
\def\BB{I\!\!B}
\def\vA{\mathtt{v}}

\def\deficiency{\Delta}
\def\spread{\Delta}
\def\dimtotal{\dimp^{*}}

\def\thetav{\bb{\theta}}
\def\thetavs{\thetav^{*}}
\def\thetavc{\thetav'}
\def\thetavd{\thetav^{\circ}}
\def\thetavdc{\thetav^{\sharp}}
\def\dthetavs{\thetav,\thetavs}

\def\thetas{\theta^{*}}
\def\thetac{\theta'}
\def\thetad{\theta^{\circ}}
\def\thetab{\theta^{\dag}}
\def\thetavb{\thetav^{\dag}}

\def\vtheta{\vartheta}
\def\vthetav{\bb{\vtheta}}
\def\prior{\Pi}

\def\Gam{\Xi}
\def\Gam{\mathcal{S}}
\def\RG{R}
\def\Psu{\Upsilon}
\def\Phim{\breve{\Phi}}

\def\Proj{P}

\def\gammavs{\gammav^{*}}
\def\gammavd{\gammav^{\circ}}
\def\etavs{\etav^{*}}
\def\etavd{\etav^{\circ}}
\def\etavc{\etav'}

\def\taus{\tau_{0}}
\def\taup{\lceil \tau \rceil}

\def\sigmas{{\sigma^{*}}}
\def\Sigmas{\Sigma_{0}}

\def\upsilonc{\upsilon'}
\def\upsilond{\upsilon^{\circ}}
\def\upsilonp{{\upsilon}^{*}}
\def\upsilonm{{\upsilon}_{*}}
\def\upsilonvs{\upsilonv^{*}}
\def\upsilons{\upsilon^{*}}
\def\upsilonb{\bar{\upsilon}}
\def\upsilonvd{\upsilonv^{\circ}}

\def\ups{\bb{\upsilon}}
\def\upss{\ups_{0}}
\def\upsc{\ups^{\prime}}
\def\upsd{\ups^{\circ}}
\def\upsdc{\ups^{\sharp}}
\def\upsdu{\ups^{\flat}}

\def\Ups{\varUpsilon}
\def\Upsd{\Ups^{\circ}}
\def\Upss{\Ups_{\circ}}
\def\UpsP{\Ups^{c}}

\def\Thetas{\Theta_{0}}
\def\ThetasGP{\Theta_{0,\GP}}
\def\varthetav{\bb{\vartheta}}

\def\glink{g}

\def\fvs{\fv}
\def\fs{f}
\def\fb{\fv^{\dag}}

\def\uc{\uv'}
\def\ud{\uv^{\circ}}
\def\uvs{\uv^{*}}
\def\us{u^{*}}
\def\vs{v^{*}}

\def\reps{\epsilon}
\def\eps{\epsilon}

\def\repsc{\reps_{0}}
\def\repsb{\reps^{*}}
\def\repsg{g}

\def\lu{\delta}
\def\lub{\bar{\lu}}

\def\Uu{U}
\def\UU{\cc{Y}}
\def\UUM{\cc{M}}
\def\UP{\cc{U}}
\def\up{\mathfrak{u}}

\def\VP{V}
\def\VPc{\VP_{0}}
\def\VPV{\cc{U}}
\def\VPVc{\cc{\VPV}_{0}}
\def\VPGP{\VP_{\GP}}
\def\VPSP{\VP_{\SP}}

\def\VV{H}
\def\GV{\cc{G}}
\def\GVS{S}

\def\VVb{\VV^{*}}
\def\VVc{\VV_{0}}
\def\vv{\bb{h}}
\def\vva{g}
\def\vp{\mathbf{v}}
\def\vpc{\vp_{0}}
\def\VVca{\VV}
\def\Vtt{H}

\def\DG{D}

\def\Vn{V_{0}}
\def\vn{v_{0}}

\def\norm{\mathfrak{c}}
\def\normc{\delta}
\def\norma{c}

\def\egridd{\cc{E}_{\delta}}
\def\penb{\varkappa}

\def\dotzeta{\dot{\zeta}}
\def\mes{\pi}
\def\mesl{\Lambda}
\def\cprr{F}

\def\lambdam{\gm_{1}}
\def\lambdaB{{\lambda}^{*}}
\def\lambdac{{\lambda'}}

\def\cla{{b}}
\def\fis{\mathfrak{a}}
\def\fiss{\fis_{1}}

\def\Vd{{V}}
\def\vd{\bar{v}}

\def\klim{k^{\circ}}
\def\midm{\mid \!}

\def\Ldrift{M}
\def\ldrift{m}
\def\mY{b}
\def\Lvar{D}
\def\lvar{\sigma}

\def\Mubcu{\Upsilon}
\def\Dthetav{\bb{u}}

\def\B{\cc{B}}
\def\BD{\B^{\circ}}
\def\BU{B}
\def\BI{\B^{*}}

\def\mub{\mu^{*}}
\def\mubc{\mu}
\def\mubcb{\mubc^{*}}
\def\Mubc{\mathbb{M}}
\def\Mubcb{\mathrm{M}}

\def\zzc{\zz_{c}}
\def\ww{w}
\def\wwc{\ww_{c}}

\def\norms{\circ} 
\def\rs{\rr_{\norms}}
\def\yys{\yy_{\norms}}
\def\xxs{\xx_{\norms}}
\def\zzs{\zz_{\norms}}
\def\uu{\mathtt{u}}
\def\uus{\uu_{\norms}}
\def\mus{\mu_{\norms}}
\def\gms{\gm_{\norms}}
\def\wws{\ww_{\circ}}

\def\srho{s}
\def\mrho{\varrho}

\def\Lmgf{\mathfrak{M}}
\def\Lmgfb{\Lmgf^{*}}

\def\lmgf{\mathfrak{m}}
\def\lmgfb{\lmgf^{*}}

\def\Expzeta{\mathfrak{N}}
\def\expzeta{\mathfrak{s}}

\def\rr{\mathtt{r}}
\def\rrb{\rr^{*}}
\def\rru{\rr_{\circ}}
\def\rrc{\rr'}
\def\rs{r_{*}}

\def\zz{\mathfrak{z}}
\def\zzb{\tilde{\zz}}
\def\tt{\mathfrak{t}}
\def\zb{z_{\rd}}
\def\zzg{\zz_{1}}
\def\zzQ{\zz_{0}}
\def\zzq{\zz}

\def\Cr{\mathfrak{c}}
\def\Crp{\mathfrak{C}}
\def\Crl{\mathfrak{r}}
\def\Crlp{\mathfrak{R}}
\def\Crlq{\cc{T}}
\def\Crlmu{\cc{M}}

\def\zetah{\zeta_{h}}
\def\GG{G}
\def\HH{H}
\def\pG{p}
\def\pH{q}
\def\hh{H^{*}}

\def\mubch{\mubc_{1}}
\def\rhoh{\rho_{1}}
\def\CoFuh{\CoFu_{1}}
\def\dimh{p_{1}}
\def\VPh{\VP_{1}}
\def\VPt{\VP_{0}}

\def\LLh{L_{1}}
\def\pnndh{\pnnd_{1}}

\def\LCS{C}
\def\Ac{A_{0}}
\def\Ab{A_{\rd}}
\def\DPrb{\DPr_{\rdb}}
\def\DPrm{\DPr_{\rdm}}
\def\Cb{\cc{C}_{\rdb}}
\def\Ub{\cc{U}_{\rdb}}
\def\zetavrb{\zetavr_{\rd}}
\def\xivrb{\breve{\xiv}_{\rd}}
\def\VPrb{\breve{\VP}_{\rdb}}
\def\Larb{\breve{\La}_{\rdb}}
\def\Larm{\breve{\La}_{\rdm}}

\def\deltav{\bb{\delta}}

\def\score{\nabla}
\def\scorer{\breve{\nabla}}

\def\LCS{C}
\def\Ac{A_{0}}
\def\Bc{B_{0}}
\def\AF{A}
\def\Ab{A_{\rdb}}
\def\Am{A_{\rdm}}
\def\DPrc{\DPr_{0}}
\def\DPrb{\DPr_{\rdb}}
\def\DPrm{\DPr_{\rdm}}
\def\Cb{\cc{C}_{\rdb}}
\def\Cm{\cc{C}_{\rdm}}
\def\Ub{\cc{U}_{\rdb}}
\def\deltav{\bb{\delta}}
\def\nuv{\bb{\nu}}
\def\xivrb{\breve{\xiv}_{\rd}}
\def\VPrb{\breve{\VP}_{\rdb}}
\def\Larb{\breve{\La}_{\rdb}}
\def\Lar{\breve{\La}}
\def\Larm{\breve{\La}_{\rdm}}
\def\VH{Q}
\def\VHc{\VH_{0}}
\def\zetavrm{\zetavr_{\rdm}}
\def\N{\mathbb{N}}

\def\Span{\operatorname{span}}
\def\Exc{{\square}}
\def\UUs{U_{\circ}}
\def\errbm{\errb^{*}}
\def\corrDF{\nu}
\def\BBr{\breve{\BB}}
\def\taua{\tau}
\def\AssId{\mathcal{I}}
\def\assId{\iota}
\def\AFD{\cc{A}}

\def\BanX{\cc{X}}
\def\basX{\ev}
\def\apprX{\alpha}
\def\fvs{\fv^{*}}
\def\lkh{\ell}
\def\Bc{B_{0}}
\def\dimn{\dimp_{\nsize}}
\def\betan{\beta_{\nsize}}


\def\xivGP{\xiv_{\GP}}
\def\dimA{\mathtt{p}}
\def\dimAGP{\dimA}
\def\dime{\dimA_{e}}
\def\dimG{\dimA_{\GP}}
\def\dimS{\dimA_{s}}
\def\nubm{\nu_{\rd}}
\def\uub{u_{\rd}}
\def\uubGP{u_{\rd,\GP}}

\def\priorden{\pi}
\def\xivGP{\xiv_{\GP}}
\def\dimAGP{\dimA}
\def\nubm{\nu_{\rd}}
\def\uub{u_{\rd}}
\def\uubGP{u_{\rd,\GP}}

\def\CR{\mathcal{C}}
\def\CRb{\CR_{\rdb}}
\def\vthetavb{\bar{\vthetav}}
\def\Covpost{\mathfrak{S}}

\def\Db{\DP_{+}}
\def\Dm{\DP_{-}}
\def\uvb{\uv_{+}}
\def\uvm{\uv_{-}}
\def\uud{\omega}
\def\taub{\delta}
\def\Lip{L}
\def\Xb{X_{+}}
\def\Xm{X_{-}}
\def\deltam{\delta_{-}}
\def\betauv{\delta}
\def\betab{\betauv_{1}}
\def\betaf{\betauv_{2}}
\def\upsv{\bb{\varkappa}}
\def\upsvb{\bar{\upsv}}
\def\rhob{\varrho}
\def\alpb{\alp_{1}}
\def\betap{\betauv_{3}}
\def\Ec{\E^{\circ}}
\def\ff{f}
\def\fpos{g}
\def\fneg{h}
\def\alpb{\alp_{+}}
\def\alpm{\alp_{-}}

\def\kappak{\kappa}
\def\kappas{\kappak^{*}}
\def\Kappak{\cc{K}}
\def\DPk{\DP_{\kappak}}
\def\VPk{\VP_{\kappak}}

\def\ts{s}
\def\tsv{\bb{\ts}}
\def\mm{\kappa}
\def\mmc{\mm'}
\def\mmd{\mm^{\circ}}
\def\mmo{\mm^{*}}
\def\mmmmo{\mm,\mmo}
\def\mmt{\tilde{\mm}}
\def\mma{\hat{\mm}}
\def\pp{z}

\def\LLL{L_{1}}
\def\LLr{L_{0}}
\def\muL{\mu_{1}}
\def\mur{\mu_{0}}

\def\LmgfL{\Lmgf_{1}}
\def\Lmgfr{\Lmgf_{0}}
\def\Lmgfm{\Lmgf_{1}}

\def\Kappa{\cc{K}}
\def\CoFu{\cc{C}}
\def\CoFuc{\CoFu_{0}}
\def\CoFub{\CoFu^{*}}
\def\CoFuL{\CoFu_{1}}
\def\CoFur{\CoFu_{0}}
\def\CAL{\CA_{1}}
\def\CAr{\CA_{0}}
\def\CAzz{\cc{A}}

\def\pnnL{\pnn_{1}}
\def\pnnr{\pnn_{0}}
\def\ttd{\delta}
\def\alphaL{\alpha_{1}}
\def\alphar{\alpha_{0}}
\def\alpharL{\alpha}
\def\rat{\mathfrak{t}}
\def\mquad{\nquad}
\def\zzL{\zz_{1}}
\def\zzr{\zz_{0}}

\def\mmset{\mathcal{I}}
\def\xex{u}
\def\dcm{q}
\def\dc{g}
\def\dcL{\dc_{1}}
\def\dcr{\dc_{0}}
\def\kk{k}

\def\cpen{\tau}

\def\dens{f}
\def\jj{j}
\def\JJ{\cc{J}}
\def\Zphi{Z}
\def\Zphiv{\bb{\Zphi}}

\def\nuu{\mathfrak{u}}
\def\nud{\mathfrak{u}_{0}}
\def\nun{c_{\nuu}}
\def\rhork{\kullb}
\def\GH{\mbox{GH}}
\def\HYP{\mbox{HYP}}
\def\NIG{\mbox{NIG}}
\def\IR{{\rm I\!R}}
\def\taggr{b}
\def\penm{\boldsymbol{m}}
\def\Crlp{\cc{R}}

\def\Mh{M}
\def\Mht{\Mh^{c}}

\def\Mhh{\Mh^{-}}
\def\Mhc{G}
\def\Lh{L_{1}}
\def\Uh{\cc{U}}
\def\wloc{w}
\def\Bias{B}
\def\bias{b}
\def\ExpzetaU{\Expzeta_{1}}
\def\vpci{\vp_{i,0}}
\def\IFci{\IF_{i,0}}

\def\erqb{\Circle_{\rdb}}
\def\erqm{\Circle_{\rdm}}
\def\errqm{\errm^{*}}
\def\errqb{\errb^{*}}
\def\Nsize{N}
\def\VVD{\VV_{1}}
\def\AA{A}
\def\Wloc{W}

\newcommand*{\Scale}[2][4]{\scalebox{#1}{$#2$}}

\def\xs{x^{*}}
\def\xn{x_{0}}
\def\xm{x_{m}}
\def\rx{\rr_{x}}
\def\etavn{\etav_{0}}
\def\etavm{\etav_{m}}
\def\etavpr{\etav_{pr}}
\def\re{\rr_{\eta}}
\def\res{\rr_{\eta}^{*}}
\def\rth{\rr_{\theta}}
\def\rths{\rr_{\theta}^{*}}
\def\rx{\rr_{x}}
\def\xpr{x_{0}}
\def\xm{x_{m}}
\def \indsqrt{\hspace{-0.05cm}\vspace{-0.1cm}\(\text{\scalebox{0.75}{$\sqrt{\ \,}$}}\)}
\def\upsilonvx{\upsilonv_{x}}
\def\upsilonvsx{\upsilonvs_{x}}
\def\upsilonvsc{\upsilonvs_{0}}
\def \sigmau{\underline{\sigma}}
\def \sigmao{\bar{\sigma}}

\def\LR {\mathrm{LR}}
\def\LRb {\mathrm{LR}^{\sbt}}

 \def \Ctau {\CONST_{\Scale[0.65]{P^{\ast}}}}
 \def \Cptau {\CONST_{f^{\Scale[0.45]{\ast}}}}

\def\KL{\operatorname{KL}}
\def\smb{\operatorname{smb}}
\def\ac{\operatorname{ac}}

\def \dv{h}
\def \NV{N_{\sigma}}
\def \ND{N_{\dv}}
\def \Ndens{N_{f^{\Scale[0.45]{\ast}}}}

\def \zzbarr{\hat{\zz}}
\def \bbV{\barbar{V}}
\def \bbVb{\barbar{\Vr}}
\def \bbD{\barbar{D}}
\def \bbB{\barbar{B}}
\def \bbH{\barbar{H}}
\def \deltabbVb{\delta_{{\Scale[0.6]{\bbVb}}}}
\def \bbgmu{\barbar{\gmu}}

\newcommand{\varg}[2]{\left[{\begin{smallmatrix} #1 \\#2 \end{smallmatrix}}\right]}

\newcommand{\gradv}[2]{\left({\begin{smallmatrix} #1 \\#2 \end{smallmatrix}}\right)}

\newcommand{\supp}{\operatorname{supp}}

\def\qqq{\mathfrak{c}}
\def\qqqb{\qqq^{\sbt}}

\def\mcorr{\Scale[1.15]{\mathfrak{c}}}
\def\mcorrb{{\Scale[1.15]{\mathfrak{c}}}^{\sbt}}
\def\mcorrq{\Scale[1.15]{\mathfrak{q}}}
\def\mcorrqb{{\Scale[1.15]{\mathfrak{q}}}^{\sbt}}

\def\qqs{\Scale[0.85]{\mathfrak{c}}}
\def\qqb{\Scale[0.85]{\mathfrak{b}}}
\def \alphaq {\alpha_{\qqs}}
\def \alphaqb {\alpha_{\qqb}}
\def\st{\tilde{s}}

\def\Pf{\mathcal{P}}
\def\Ff{\mathcal{F}}
\def\Uf{\mathcal{U}}
\def\Tf{\mathcal{T}}

\def \Pe{\P_{\etav}}
\def \Pes{\P_{\etav}^{*}}
\def \Ps{\P^{*}}

\def\Sety{\Omega(\rups,\yy)}
\def\Setyb{\Omega^{\sbt}(\rupsb,\yy)}

\def\penx{\mathfrak{q}}
\def\Jm{\mathtt{J}}

\def\DP{D}
\def\DPc{\DP_{0}}
\def\Db{\cc{D}}
\def\Bb{\mathcal{B}}
\def\Dbr{\mathbb{D}}
\def\Dbmm{\Db_{m}}

\def \zv{\bb{z}}

\def\Mb{M^{\sbt}}

\def \Exc{\cc{X}}
\def \Excb{\Exc^{\sbt}}

\def \Mb{M^{\sbt}}
\def \Cac{C_{ac}}
\def \CONSTac{\CONST_{\ac}}
\def \Deltaac{\Delta_{\ac}}
\def \DeltaacAT{\DeltappLR}
\def \DeltaLR{\Delta_{\Scale[0.6]{\operatorname{LR}}}}
\def \DeltaLRBE{\Delta_{\Scale[0.6]{\textrm{B.E.},\operatorname{LR}}}}

\def\gmub{{\gmu}^{\sbt}}
\def\rupsb{\rups^{\sbt}}

\def \Fe{F_{\varepsilon}}
\def \Feb{F_{\varepsilon}^{\sbt}}

\def \pb{p^{\sbt}}
\def \Fb{F^{\sbt}}
\def \Gb{G^{\sbt}}
\def \Hb{H^{\sbt}}

\def \Grid{\mathrm{G}}
\def \Gridx{\Grid_{\X}}
\def \Gridcar{\Grid_{\#}}
\def \Gridgamma{\Grid_{\gamma}}

\def \Br{\mathtt{B}}
\def \Brb{\Br^{\sbt}}
\def \crr{\mathtt{c}}

\def\Labe{\La_{\rdb,\etav}}
\def\Lame{\La_{\rdm,\etav}}
\def\Labth{\La_{\rdb,\thetav}}
\def\Lamth{\La_{\rdm,\thetav}}
\def\xivbth{\xiv_{\rdb,\thetav}}
\def\xivmth{\xiv_{\rdm,\thetav}}
\def\xivt{\tilde{\xiv}}

\def\gammavg{\gammav_{\gm}}
\def\etavs{\etav^{*}}
\def\thetavh{\hat{\thetav}}
\def\card{\mathrm{card}}

\newcommand{\sbt}{\,\begin{picture}(0,1)
                    \put(1,3){\circle{2}}
                    \put(1,3){\circle{3}}
                    \end{picture}\ }
\newcommand{\sbtw}{\,\begin{picture}(0,1)
					\put(1.2,8){\circle{3}}
                    \put(1.2,8){\circle{4}}
                    \put(1.2,8){\circle{5}}
                    \put(1.2,8){\circle{6}}
                    \end{picture}\ }

\newcommand{\sbtt}{\,\begin{picture}(0,1)
                    \put(1,2){\circle{2}}
                    \put(1,2){\circle{3}}
                    \end{picture}\ }
\def\Lt{\widetilde{L}}
\def\Lb{L^{\sbt}}
\def\Eb{\E^{\sbt}}
\def\Pb{\P^{\sbt}}
\def\Varb{\Var^{\sbt}}
\def\Covb{\Cov^{\sbt}}
\def\Indb{\Ind^{\sbt}}
\def\zetavb{\zetav^{\sbt}}
\def\zetab{\zeta^{\sbt}}
\def\xivb{\xiv^{\sbt}}
\def\Xib{\Xi^{\sbt}}

\def\etavb{\etav^{\sbt}}
\def\Thetasb{\Thetas^{\sbt}}
\def\Labb{\Lab^{\sbt}}
\def\Lamb{\Lam^{\sbt}}
\def\omegab{\omega_{1}}
\def\omegabk{\omega_{1,k}}
\def\deltab{\delta^{\sbt}}
\def\zzb{\zz^{\sbt}}
\def\rhob{\rho^{\sbt}}
\def\phib{\phi^{\sbt}}
\def \vphi{\varphi}
\def \vphib{\varphi^{\sbt}}

\def \deltasmb{\delta_{\smb}}
\def \DeltaSmB{\barbar{\delta}_{\Scale[0.6]{{\operatorname{smb}}}}}

\def\deltabias{\delta_{\textit{bias}}}
\def\Sigmabias{\Sigma_{\textit{bias}}}

\def\Xib{\Xi^{\sbt}}
\def\Xic{\Xi_{0}}
\def\Xibc{\Xib_{0}}

\def\xib{\xi^{\sbt}}

\def\Phivbar{\bar{\Phi}}
\def\Psivbar{\bar{\Psi}}

\def \Idv{\bb{1}}

\def\xivbr{\bar{\xiv}}
\def\xivbrn{{\bar{\xiv}}_{0}}
\def\xivn{\xiv_{0}}
\def\phivbr{\bar{\phiv}}
\def\psivbr{\bar{\psiv}}
\def\zbr{\bar{z}}
\def\phivbar{\bar{\phiv}}
\def\psivbar{\bar{\psiv}}
\def\phivt{\tilde{\phiv}}

\def\varrhom{\varrho_{m}}
\def\arho{a_{\varrho}}
\def\brho{b_{\varrho}}
\def\mmn{m_{0}}

\def\phivm{\phiv_{\bar{m}}}
\def\phivtm{\tilde{\phiv}_{\bar{m}}}
\def\phivbarm{\bar{\phiv}_{\bar{m}}}
\def \Thr{\mathcal{T}}

\def\alphavbar{\bar{\bb{\alpha}}}
\def\ybr{\breve{y}}
\def\xibr{\breve{\xi}}
\def\Phibar{\bar{\Phi}}
\def\Phibarj{\bar{\Phi}^{\,j}}
\def\Psibar{\bar{\Psi}}
\def\Psibarj{\bar{\Psi}^{\,j}}
\def\xivbbr{{\xivbr}^{\sbt}}
\def\phivbbar{{\phivbar}^{\sbt}}
\def\Sigmabar{\bar{\Sigma}}
\def\Sigmabr{\breve{\Sigma}}
\def\deltabr{\breve{\delta}}
\def\zbar{\bar{z}}
\def\zvbar{\bar{\zv}}
\def\abar{\bar{a}}
\def\Bsmb{B_{0}}
\def\Bt{\tilde{B}}
\def\Vt{\tilde{V}}
\def\Vh{\hat{V}}
\def\Bh{\hat{B}}

\def \Deltapluspp{\Delta_{\phiv\psiv}^{+}}
\def \Deltaminuspp{\Delta_{\phiv\psiv}^{-}}
\def \Deltapp{\Delta_{\ell_{2}}\hspace{-0.1cm} }
\def \DeltappLR{\Delta_{\ac,\Scale[0.58]{\operatorname{LR}}}}

\def\deltabrBE{\breve{\delta}_{n,\Scale[0.55]{\operatorname{B.E.}}}}
\def\deltaBE{{\delta}_{n,\Scale[0.55]{\operatorname{B.E.}}}}
\def\DeltaBE{{\Delta}_{\Scale[0.55]{\operatorname{B.E.}}}}
\def\DeltaBEfull{{\Delta}_{\Scale[0.55]{\operatorname{B.E.}},\,\full}}
\def\DeltaBEfullzz{{\Delta}_{\Scale[0.55]{\operatorname{B.E.}}\,\zz,\,\full}}
\def\Deltabfull{{\Delta}_{\mb_{0},\,\full}}
\def\DeltabfullI{{\Delta}_{\mb,\,\full}}
\def\DeltabfullIk{{\Delta}_{\mb,\,\full,\,k}}
\def\DeltabfullII{{\Delta}_{\full_{2}}}
\def\DeltabfullO{{\Delta}_{\full_{0}}}
\def\Deltafull{\Delta_{\full}}

\def \tv{\bb{t}}
\def \yv{\bb{y}}
\def \xv{\bb{x}}
\def \xvn{\xv_{0}}
\def \xvg{ \xv_{{\tiny \textit{\textrm{G}}}}}

\def \Ret{\operatorname{Re}}
\def \Imt{\operatorname{Im}}

\def \CONST{\mathtt{C}}
\def \errbb{\err^{\sbt}}
\def \errmb{\errm^{\sbt}}
\def \Deltab{\Delta^{\sbt}}

\def \DeltaW{\Delta_{\operatorname{W}}}
\def \DeltaWsq{\Delta_{{\operatorname{W}}^{2}}}

\def \DeltakW{\Delta_{k,\operatorname{W}}}
\def \DeltakWsq{\Delta_{k,{\operatorname{W}}^{2}}}
\def \DeltabkW{\Deltab_{k,\operatorname{W}}}
\def \DeltabkWsq{\Deltab_{k,{\operatorname{W}}^{2}}}

\def \DeltabW{\Deltab_{\textrm{W}}}
\def \Deltabw{\Deltab_{\text{\tiny{W}}}}
\def \DeltabWsq{\Deltab_{{\textrm{W}}^{\,2}}}
\def \DeltaG{\Delta_{\textrm{G}}}
\def \DeltaGp{\Delta_{\textrm{G},0}}
\def \DeltaGI{\Delta_{\textrm{G},1}}
\def \DeltaGII{\Delta_{\textrm{G},2}}
\def \DeltaGIII{\Delta_{\textrm{G},3}}
\def \DeltaGs{\Delta_{\textrm{Gs}}}
\def\wind{\text{\tiny{W}}}

\def\ZZ{\mbox{\small{\(\mathfrak{Z}\hspace{-0.05cm}\)}}}
\def\ZZqf{\mbox{\small{\(\mathfrak{Z}\)}}_{\hspace{-0.06cm}\mathtt{qf}}}
\def\ZZb{\ZZ}
\def\ZZqfb{\ZZqf^{\sbt}}

\def\Deltaqqf{\Delta_{\mathtt{qf},1}}
\def\Deltaqlf{\Delta_{\mathtt{qf},2}}
\def\Deltaqqfkk{\Delta_{\mathtt{qf},1,k}}
\def\Deltaqlfkk{\Delta_{\mathtt{qf},2,k}}
\def\Deltaqqfk{\Delta_{\mathtt{qf}_{1},k}}
\def\Deltaqlfk{\Delta_{\mathtt{qf}_{2},k}}

\def\bbv{\barbar{v}}
\def\bbg{\bb{g}}
\def\vt{\mathrm{v}}
\def\Vb{\cc{V}}
\def\Vbr{\mathbb{V}}
\def\wb{\mathbf{w}}
{
\def\DPb{D_{\rdb}}
\def\DPm{D_{\rdm}}
\def\Dbb{\Db_{\rdb}}
\def\Dbm{\Db_{\rdm}}
\def\Dbrb{\Dbr_{\rdb}}
\def\Dbrm{\Dbr_{\rdm}}

\def\U{\mathcal{U}}
\def\gmb{\gm^{\sbt}}
\def\gmdb{\gmd^{\sbt}}
\def\Gn{\Gamma_{0}}
\def\Gbn{\Gamma_{0}^{\sbt}}
\def\Gq{\Gamma_{\dimq}}
\def\nunub{\nunu^{\sbt}}
\def\Ub{\U^{\sbt}}
\def\A{\mathcal{A}}
\def\Ab{\A^{\sbt}}

\def \Dr{\mathfrak{D}}
\def \dr{\mathfrak{d}}
\def \Drz{\bar{\Dr}}
\def \Drc{\Dr_{0}}
\def \Vr{\mathcal{V}}
\def \Vrc{\Vr_{0}}
\def \Vrb{\bar{\Vr}}
\def \Zr{\mathcal{Z}}
\def\Hr{\mathcal{H}}
\def \Kr{\mathcal{K}}
\def \Kc{K_{0}}
\def \Sigmab{\Sigma^{\sbt}}

\def\full{\operatorname{full}}
\def\total{\operatorname{total}}
\def\mb{\operatorname{b}}
\def\sm{\operatorname{sm}}
\def\sumop{\operatorname{sum}}

\def\dimtotal{\dimp_{\Scale[0.6]{\sumop}}}

\def\BBPb{\BB_{\rdb}}

\def \UPc{U_{0}}
\def \rhoru{\rhor_{u}}
\def \ru{\rr_{u}}
\def \rus{\rr_{u}^{*}}
\def \uc{u^{\circ}}

\def\Uscr{\mathscr{U}}
\def\Yscr{\mathscr{Y}}
\def\Yscrb{\Yscr^{\sbt}}
\def\X{\mathcal{X}}
\def\Xs{\mathcal{X}_{0}}
\def\Xm{\mathcal{X}_{M}}
\def\Ys{\Upsilon_{0}}
\def\Us{U_{0}}
\def\H{\mathcal{H}}
\def \Hs{\H_{0}}
\def \Hm{\H_{M}}
\def\Hg{\textEta}
\def\B{\mathcal{B}}
\def\Int{\mathcal{I}}
\def \Id {\boldsymbol{I}}

\def \Hc{H_{0}}

\def\cind{\circ}
\def\bind{(b)}

\def\thetat{\tilde{\theta}}
\def\thetabt{\tilde{\theta}^{\sbtt}}
\def\thetavt{\tilde{\thetav}}
\def\thetavb{\tilde{\thetav}}
\def\thetavbt{\tilde{\thetav}^{\sbtt}}
\def\thetabt{\tilde{\theta}^{\sbtt}}
\def\varepsilons{\varepsilon^{*}}
\def\varepsilonvs{\varepsilon^{*}}
\def\varepsilonh{\hat{\varepsilon}}
\def\varepsilonvh{\hat{\varepsilonv}}

\def\Ym{\mathbb{Y}}

\def \betav{\bb{\beta}}

\def\dimd{d}
\def\dimq{q}

\def\numK{K}
\def\numKp{K^{\prime}}
\def\dimQ{Q}
\def\numk{k}
\def\numS{S}

\def\ak{k}
\def\kn{\textrm{u}_{0}}

\def\dn{\textrm{d}_{0}}
\def\dm{\textrm{d}_{m}}
\def\drb{\textrm{d}_{\rdb}}
\def\drm{\textrm{d}_{\rdm}}
\def\mn{{\mu}_{0}}

\def\dpc{d_{0}}
\def\dpi{\mathbf{d}}
\def\lvec{\vec{\varphi}}
\def\lbp{\mathbf{l}}

\def \cf{\mathfrak{c}}
\def \cfb{\mathfrak{c}_{\rdb}}
\def \cfm{\mathfrak{c}_{\rdm}}
\def \kappab{\kappa_{\rdb}}
\def \kappam{\kappa_{\rdm}}

\def \gammavp{\gammav_{\dimp}}
\def \gammavq{\gammav_{\dimq}}

\def \tauw{\tau_{w}}
\def \taue{\tau_{\rdb}}
\def \tauew{\tau_{\rdb,w}}

\def \sk{\mathbf{s}}
\def \kk{\mathbf{k}}

\def\xblue {{\color{blue}x}}
\def\rspace {{\color{red} \textbf{(...) }}}
\def \lqouter {{\color{red} \(\langle\langle\)}}
\def \rqouter {{\color{red} \(\rangle\rangle\) } }
\def \ants{\textbf{{\color{red} (***) }}}
\def \bants{\textbf{{\color{blue} (***) }}}
\def \gants{\textbf{{\color{green} (***) }}}
\def \ant{\textbf{{\color{red} (*) }}} 
\def \bant{\textbf{{\color{blue} (*) }}}

\makeatletter
\newcommand{\labitem}[2]{%
\def\@itemlabel{\textbf{#1}}
\item
\def\@currentlabel{#1}\label{#2}}
\newcommand{\highw}{\raisebox{.45\height}{\small{\textit{w}}}}
\makeatother

\makeatletter
\def\th@newremark{\th@remark\thm@headfont{\bfseries}}{\thm@bodyfont{\normalfont}}
\makeatother
\theoremstyle{newremark}
\newtheorem{newremark}{Remark}
\numberwithin{newremark}{section}

\newcommand{\scale}[2][4]{\scalebox{#1}{$#2$}}
\newcommand{\inbr}[1]{\scale[0.6]{[#1]}}

\newcommand*\oline[1]{%
  \vbox{%
    \hrule height 0.4pt
    \kern0.25ex
    \hbox{%
      \kern 0em 
      \ifmmode#1\else\ensuremath{#1}\fi
      \kern 0em
    }
  }
}

\newcommand*{\barbar}[1]{\hat{#1}}

\def\IF{\mathbf{f}}
\definecolor{blue(pigment)}{rgb}{0.2, 0.2, 0.6}
\definecolor{ultramarine}{rgb}{0.07, 0.04, 0.56}
\definecolor{ao(english)}{rgb}{0.0, 0.5, 0.0}
\definecolor{darkspringgreen}{rgb}{0.09, 0.45, 0.27}
\definecolor{hookersgreen}{rgb}{0.0, 0.44, 0.0}
\definecolor{darkgray}{rgb}{0.66, 0.66, 0.66}
\definecolor{dimgray}{rgb}{0.41, 0.41, 0.41}

\newcommand\crule[3][dimgray]{\textcolor{#1}{\rule{#2}{#3}}}
\newcommand\rrule[3][red]{\textcolor{#1}{\rule{#2}{#3}}}

\newcommand*\rfrac[2]{{}^{#1}\!/_{#2}}

\def \yy{\xx}
\def \qq {\mbox{\large{$\mathfrak{q}$}}}
\def\Ym{\bb{Y}}
\def\Yv{Y}

\newpage
\hypersetup{linkcolor=ultramarine}
{\small{\tableofcontents}}
\newpage
\hypersetup{linkcolor=hookersgreen}

\section{Introduction}
\label{sect:introduction}
The problem of simultaneous confidence estimation appears in numerous practical applications when a confidence statement has to be made  simultaneously for a collection of  objects, e.g. in safety analysis in clinical trials, gene expression analysis, population biology, functional magnetic resonance imaging and many others. See e.g.  \cite{Miller1981simultaneous,Westfall1993resampling,Manly2006Biol,Benjamini2010simultaneous,Dickhaus2014simult},
  and references therein. This problem is also closely related to construction of simultaneous confidence bands in curve estimation, which goes back to \cite{Working1929applications}. For an extensive literature review about constructing the simultaneous confidence bands we refer to \cite{HallHorow2013simple}, \cite{Liu2010simultaneous}, and \cite{Wasserman2006all}.

  A simultaneous confidence set requires a probability bound to be constructed jointly for several possibly dependent statistics. Therefore, the critical values of the corresponding statistics should be chosen in such a way that the joint probability distribution  achieves a required family-wise confidence level. This choice  can be made by multiplicity correction of the marginal confidence levels. The Bonferroni correction method (\cite{Bonferroni1936teoria}) uses a  probability union bound, the corrected marginal significance levels are taken equal to the total level divided by the number of models. This procedure can be very conservative if the considered statistics are positively correlated and if their number is large. The {\v{S}}id{\'a}k correction method (\cite{Sidak1967rectangular}) is more powerful than Bonferroni correction, however, it also becomes conservative in the case of large number of dependent statistics.
 
Most of the existing results about simultaneous bootstrap confidence sets and resampling-based multiple testing are asymptotic (with sample size tending to infinity), see e.g. \cite{Beran1988balanced,Beran1990refining,Hall1Pittelkow990simultaneous,Haerdle1991bootstrap,Shao1995jackknife,HallHorow2013simple}, and  \cite{Westfall1993resampling,Dickhaus2014simult}. 
 The results based on asymptotic distribution of maximum of an approximating Gaussian process (see \cite{BickelRosenblatt1973density,Johnston1982probabilities,Hardle1989asymptotic}) require a huge sample size \(n\), since they yield a coverage probability error of order \(\left(\log(n)\right)^{-1}\) (see \cite{Hall1991convratesup}). 
 Some papers considered an alternative approach in context of confidence band estimation based on the approximation of the underlying empirical processes by
its bootstrap counterpart. In particular, 
 \cite{Hall1993Edgeworthbands} 
showed that such an approach leads to a significant improvement
of the error rate (see also \cite{Neumann1998simultaneous,Claeskens2003bootstrap}). 
\cite{Chernozhukov2014Honest} constructed honest confidence bands for nonparametric density estimators without requiring the existence of limit distribution of the supremum of the studentized empirical process: instead, they used  an approximation between sup-norms of an  empirical and Gaussian processes, and anti-concentration property of suprema of Gaussian processes.

 In many modern applications the sample size cannot be large, and/or can be smaller than a parameter dimension, for example, in genomics, brain imaging, spatial epidemiology and microarray data analysis, see \cite{Leek2008general,Kim2008effects,ArlotBlanchTesting2010,CaoKosorok2011simultaneous}, and references therein.

For the recent results on resampling-based simultaneous confidence sets in high-dimensional finite sample set-up we refer to the papers by \cite{ArlotBlanchTesting2010} and \cite{ChernoMultBoot,Chernozhukov2014Honest,Chernozhukov.et.al.(2014a)}. 
\cite{ArlotBlanchTesting2010} considered i.i.d. observations of a Gaussian vector with a dimension possibly much larger than the sample size, and with unknown covariance matrix. They examined
multiple testing problems for the mean values of its coordinates and
provided non-asymptotic control for the family-wise error rate using resampling-type procedures.
\cite{ChernoMultBoot} presented a number of non-asymptotic results on Gaussian approximation and multiplier bootstrap for maxima of sums of high-dimensional vectors (with a dimension possibly much larger than a sample size) in a very general set-up. As an application the authors considered the problem of multiple hypothesis testing in the framework of approximate means. They derived non-asymptotic results for the general stepdown procedure by \cite{RomanoWolf2005exact} with improved error rates and in high-dimensional setting. \cite{Chernozhukov2014Honest} showed how this technique applies to the problem of constructing an honest confidence set in nonparametric density estimation. \cite{Chernozhukov.et.al.(2014a)} extended the results from maxima to the class of sparsely convex sets.

The present paper studies simultaneous likelihood-based bootstrap confidence sets in the following setting:
\begin{enumerate}
\labitem{1}{sett1}\hspace{-0.21cm}. the sample size \(n\) is fixed;
\labitem{2}{sett2}\hspace{-0.21cm}. the parametric models can be misspecified;
\labitem{3}{sett3}\hspace{-0.21cm}. the number \(\numK\) of the parametric models can be exponentially large w.r.t. \(n\);
\labitem{4}{sett4}\hspace{-0.21cm}. the maximal dimension \(\dimp_{\max}\) of the considered parametric models can be dependent on the sample size \(n\).
\end{enumerate}
This set-up, in contrast with the paper by \cite{Chernozhukov.et.al.(2014a)}, does not require the sparsity condition , in particular the dimension \(\dimp_{1},\dots,\dimp_{\numK}\) of each parametric family may grow with the sample size. Moreover, the simultaneous likelihood-based confidence sets are not necessarily convex, and the parametric assumption can be violated. 

The considered simultaneous multiplier bootstrap procedure involves two main steps:  estimation of the quantile functions  of the likelihood ratio statistics, and multiplicity correction of the marginal confidence level. 
Theoretical results of the paper state the bootstrap validity in the setting \ref{sett1}-\ref{sett4} taking in account the multiplicity correction. The resulting approximation bound requires the quantity \((\log K)^{12}p_{\max}^{3}/n\) to be small. The log-factor here is suboptimal and can probably be improved. The paper particularly focuses on the impact of the model misspecification.
We distinguish between slight and strong  misspecifications. 
Under the so called small modeling bias condition  \ref{itm:SmBHk} given in Section \ref{sect:ConditAddBoot} the bootstrap approximation is accurate. This condition roughly means that all the parametric models are close to the true distribution. If the \ref{itm:SmBHk} condition is not fulfilled, then the simultaneous bootstrap confidence set is still applicable, however, it becomes conservative. This property is nicely confirmed by the numerical experiments in Section \ref{sect:numer_examples}.

Let the random data 
\begin{EQA}
\label{def:Ymatr}
\Ym&\eqdef& \left(\Yv_{1},\dots,\Yv_{n}\right)^{\T}
\end{EQA}
consist of \emph{independent} observations \(Y_{i}\),
and belong to the probability space \(\left(\Omega, \mathcal{F},\P\right)\).
The sample size \(n\) is \emph{fixed}. \(\P\) is an \emph{unknown probability distribution} of the sample \(\Ym\). Consider \(\numK\) regular parametric families of probability distributions:
\begin{EQA}[c]
\label{def:Pk}
\left\{\P_{k}(\thetav)\right\}\eqdef \left\{\P_{k}(\thetav)\ll\mu_{0}, \thetav\in \Theta_{k}\subset \R^{\dimp_{k}}\right\},\quad
k=1,\dots,\numK.
\end{EQA}
 Each parametric family induces the quasi log-likelihood function for \(\thetav\in\Theta_{k}\subset \R^{\dimp_{k}}\)
\begin{align}
\label{def:Lk}
\begin{split}
L_{k}(\Ym,\thetav)&\eqdef
\log\left(\frac{d\P_{k}(\thetav)}{d\mu_{0}}(\Ym)\right)
\\&=
\sum\nolimits_{i=1}^{n}
\log\left(
\frac{d\P_{k}(\thetav)}{d\mu_{0}}(\Yv_{i})
\right).
\end{split}
\end{align}
It is important that we \emph{do not require} that \(\P\) belongs to any of the known parametric families \(\left\{\P_{k}(\thetav)\right\}\), that is why the term \emph{quasi} log-likelihood is used here. Below in this section we consider two popular examples of simultaneous confidence sets in terms of the quasi log-likelihood functions \eqref{def:Lk}. Namely, the simultaneous confidence band for local constant regression, and multiple quantiles regression.

The target of estimation for the misspecified log-likelihood \(L_{k}(\thetav)\) is such a parameter \(\thetavs_{k}\), that minimises the Kullback-Leibler distance between the unknown true measure \(\P\) and the parametric family \(\left\{\P_{k}(\thetav)\right\}\):
\begin{EQA}
\label{def:thetavs_k}
\thetavs_{k}&\eqdef&\argmax_{\thetav\in\Theta_{k}}\E L_{k}(\thetav).
\end{EQA}
The maximum likelihood estimator is defined as:
\begin{EQA}
\label{def:thetavt_k}
\thetavt_{k}&\eqdef& \argmax_{\thetav\in\Theta_{k}}L_{k}(\thetav).
\end{EQA}
The parametric sets \(\Theta_{k}\) have dimensions \(\dimp_{k}\), therefore, \(\thetavt_{k}, \thetavs_{k}\in \R^{\dimp_{k}}\).
For \(1\leq k,j\leq \numK\) and \(k\neq j\) the numbers \(\dimp_{k}\) and \(\dimp_{j}\) can be unequal.

The likelihood-based confidence set for the target parameter \(\thetavs_{k}\) is
\begin{EQA}
\label{def:CS_k}
\CS_{k}(\zz)&\eqdef&\left\{\thetav\in\Theta_{k}: L_{k}(\thetavt_{k})-L_{k}(\thetav)\leq \zz^{2}/2\right\}\subset\R^{\dimp_{k}}.
\end{EQA}
Let \(\zz_{k}(\alpha)\) denote the \((1-\alpha)\)-quantile of the corresponding square-root likelihood ratio statistic:
\begin{EQA}
\label{def:zzalpha_k}
\zz_{k}(\alpha)&\eqdef&
\inf\left\{\zz\geq 0: \P\left(L_{k}(\thetavt_{k})-L_{k}(\thetavs_{k})> \zz^{2}/2\right)\leq \alpha\right\}.
\end{EQA}
Together with \eqref{def:CS_k} this implies for each \(k=1,\dots,\numK\):
\begin{EQA}
\label{bound:pointwise}
\P\Bigl(\thetavs_{k}\in \CS_{k}\left(\zz_{k}(\alpha)\right)\Bigr)&\geq& 1-\alpha.
\end{EQA}
Thus \(\CS_{k}(\zz)\) and the quantile function \(\zz_{k}(\alpha)\) fully determine the marginal \((1-\alpha)\)-confidence set. The simultaneous confidence set requires a correction for multiplicity . Let \(\qqq(\alpha)\) denote a maximal number \(c \in (0,\alpha]\) s.t.
\begin{EQA}
\label{bound:simult_goal}
\P\left(\bigcup\nolimits_{k=1}^{\numK}
\Bigl\{
\sqrt{2L_{k}(\thetavt_{k})-2L_{k}(\thetavs_{k})}> \zz_{k}(c)
\Bigr\}
\right)
&\leq&
\alpha.
\end{EQA}
This is equivalent to
\begin{EQA}[c]
\label{bound:simult}
\hspace{-0.7cm}
\qqq(\alpha)\eqdef
\sup\Biggl\{c\in(0,\alpha]:
\P\left(\max_{1\leq k \leq \numK}
\left\{
\sqrt{2L_{k}(\thetavt_{k})-2L_{k}(\thetavs_{k})}- \zz_{k}(c)
\right\}
>0
\right)
\leq
\alpha
\Biggr\}.
\end{EQA}
Therefore, taking the marginal confidence sets with the same  confidence levels \(1-\qqq(\alpha)\) yields the simultaneous confidence bound of the total level \(1-\alpha\). The value \(\qqq(\alpha)\in(0,\alpha]\) is the correction for multiplicity. In order to construct the simultaneous confidence set using this correction, one has to estimate the values \(\zz_{k}(\qqq(\alpha))\) for all \(k=1,\dots,\numK\). By its definition this problem splits into two subproblems:
\begin{enumerate}
\labitem{\(1 \)}{itm:sub1}\hspace{-0.27cm}.\,\,\,\textbf{Marginal step.}
 Estimation of the marginal quantile functions \(\zz_{1}(\alpha)\), \dots, \(\zz_{\numK}(\alpha)\) given in \eqref{def:zzalpha_k}.
\labitem{\(2 \)}{itm:sub2}\hspace{-0.27cm}.
\textbf{Correction for multiplicity.}
Estimation of the correction for multiplicity \(\qqq(\alpha)\) given in \eqref{bound:simult}.
\end{enumerate}
If the  \ref{itm:sub1}-st problem is solved for any \(\alpha\in(0,1)\), the \ref{itm:sub2}-nd problem can be treated by calibrating the value \(\alpha\) s.t. \eqref{bound:simult} holds. It is important to take into account the correlation between the likelihood ratio statistics \(L_{k}(\thetavt_{k})-L_{k}(\thetavs_{k})\), \(k=1,\dots,\numK\), otherwise the estimate of the correction \(\qqq(\alpha)\) can be too conservative. For instance, 
the Bonferroni correction would lead to the marginal confidence level \(1-\alpha/\numK\), which may be very conservative if \(\numK\) is large and the statistics \(L_{k}(\thetavt_{k})-L_{k}(\thetavs_{k})\) are highly correlated.

In Section \ref{sect:bootstr_intro} we suggest a multiplier bootstrap procedure, which performs the steps \ref{itm:sub1} and \ref{itm:sub2} described above.
  Theoretical justification of the procedure is given in Section \ref{sect:theorres}. The proofs are based on several approximation bounds: non-asymptotic square-root Wilks theorem, simultaneous Gaussian approximation for \(\ell_{2}\)-norms, Gaussian comparison, and simultaneous Gaussian anti-concentration inequality.

\cite{SpZh2014PMB} considered the \ref{itm:sub1}-st subproblem for the case of a single parametric model (\(\numK=1\)): a multiplier bootstrap procedure was applied for construction of a likelihood-based confidence set, and justified theoretically for a fixed sample size and for possibly misspecified parametric model. In the present paper we extend that approach for the case of simultaneously many parametric models.

Below we illustrate the definitions \eqref{def:Lk}-\eqref{bound:simult} of the simultaneous likelihood-based confidence sets  with two popular examples.
\par
\textbf{Example 1 (Simultaneous confidence band for local constant regression):}
Let \(Y_{1},\dots,Y_{n}\) be independent random scalar observations and \(X_{1},\dots,X_{n}\) some deterministic design points. Consider the following quadratic likelihood function reweighted with the kernel functions \(K(\cdot)\):
\begin{EQA}[c]
L(\thetav,x,h)\eqdef -\frac{1}{2}\sum\nolimits_{i=1}^{n} (Y_{i}-\thetav)^{2}w_{i}(x,h),\\
w_{i}(x,h)\eqdef K(\{x-X_{i}\}/h),\\
K(x)\in[0,1],\ \int_{\R}K(x)dx =1,\ K(x)=K(-x).
\end{EQA}
Here \(h>0\) denotes bandwidth, the local smoothing parameter. The target point and the local MLE read as:
\begin{EQA}[c]
\thetavs(x,h)\eqdef  \frac{ \sum\nolimits_{i=1}^{n}w_{i}(x,h)\E Y_{i} }{\sum\nolimits_{i=1}^{n}w_{i}(x,h)},~~~~
\thetavt(x,h)\eqdef \frac{ \sum\nolimits_{i=1}^{n}w_{i}(x,h) Y_{i} }{\sum\nolimits_{i=1}^{n}w_{i}(x,h)}.
\end{EQA}
\(\thetavt(x,h)\) is also known as Nadaraya-Watson estimate. Fix a bandwidth \(h\) and consider the range of points  \(x_{1},\dots,x_{\numK}\). They yield \(\numK\) local constant models with the target parameters \(\thetavs_{k}\eqdef\thetavs(x_{k},h)\) and the likelihood functions \(L_{k}(\thetav)\eqdef L(\thetav,x_{k},h)\) for \(k=1,\dots,\numK\).
The confidence intervals for each model are defined as
\begin{EQA}
\CS_{k}(\zz,h)&\eqdef&\left\{\thetav\in\Theta: L(\thetavt(x_{k},h),x_{k},h)-L(\thetav,x_{k},h)\leq \zz^{2}/2\right\},
\end{EQA}
for the quintile functions \(\zz_{k}(\alpha)\) and for the multiplicity correction \(\qqq(\alpha)\) from \eqref{def:zzalpha_k} and \eqref{bound:simult} they form the following simultaneous confidence band:
\begin{EQA}
\P\left(\bigcap\nolimits_{k=1}^{\numK}\Bigl\{\thetavs_{k}\in \CS_{k}\bigl(\zz_{k}\left(\qqq(\alpha)\right)\bigl)\Bigr\}\right)&\geq& 1-\alpha.
\end{EQA}
In Section \ref{sect:numer_examples} we provide results of numerical experiments for this model.

\par\textbf{Example 2 (Multiple quantiles regression):}
Quantile regression is an important method of statistical analysis, widely used in various applications. It aims at estimating conditional quantile functions of a response variable, see \cite{Koenker2005quantile}. Multiple quantiles regression model considers simultaneously several quantile regression functions based on a range of quantile indices, see e.g. 
\cite{LiuWu2011simultaneousmultiplequant,Qu2008testingregquant,He1997quantilecross}. 
Let \(Y_{1},\dots,Y_{n}\) be independent random scalar observations and \(X_{1},\dots,X_{n}\in\R^{\dimd}\) some deterministic design points, as in Example 1. 
 Consider the following quantile regression models for \(k=1,\dots,\numK\):
\begin{EQA}
Y_{i}&=& g_{k}(X_{i})+\varepsilon_{k,i},\quad i=1,\dots, n,
\end{EQA}
where \(g_{k}(\xv):\R^{\dimd}\mapsto\R\) are unknown functions, the random values \(\varepsilon_{k,1},\dots,\varepsilon_{k,n}\) are independent for each fixed \(k\), and
\begin{EQA}[c]
\P(\varepsilon_{k,i}<0)=\tau_{k} \quad\text{ for all } i=1,\dots, n.
\end{EQA}
The range of quantile indices \(\tau_{1},\dots,\tau_{\numK}\in (0,1)\) is known and fixed. We are interested in simultaneous parametric confidence sets for the functions \(g_{1}(\cdot),\dots,g_{\numK}(\cdot)\). 
Let \(f_{k}(\xv,\thetav):\R^{\dimd}\times\R^{\dimp_{k}}\mapsto \R\) be known regression functions. Using the quantile regression approach by \cite{Koenker1978regression}, this problem can be treated with the quasi maximum likelihood method and the following  log-likelihood functions:
\begin{EQA}
L_{k}(\thetav)&=& -\sum\nolimits_{i=1}^{n}\rho_{\tau_{k}}\left(Y_{i}-f_{k}(X_{i},\thetav)\right),\\
\rho_{\tau_{k}}(x)&\eqdef& x\left(\tau_{k}-\Ind\left\{x<0\right\}\right).
\end{EQA}
 for \(k=1,\dots,\numK\).
This quasi log-likelihood function corresponds to the Asymmetric Laplace distribution with the density \(\tau_{k}(1-\tau_{k})\ex^{-\rho_{\tau_{k}}(x-a)}\). If \(\tau=1/2\),  then  \(\rho_{1/2}(x)= |x|/2\) and \(L(\thetav)=-\sum_{i=1}^{n}\left|Y_{i}-f_{k}(X_{i},\thetav)\right|/2\), which corresponds to the median regression.

The paper is organised as follows: Section \ref{sect:bootstr_intro} describes the multiplier bootstrap procedure, Section \ref{sect:theorres} explains the ideas of the theoretical approach and provides main results in Sections \ref{sect:scheme} and \ref{sect:mainres} correspondingly. All the necessary conditions are given in Section \ref{sect:conditions}. In Section \ref{typical_local} and in statements of the main theoretical results we  provide information about dependence of the involved terms on the sample size and parametric dimensions in the case of i.i.d. observations. Proofs of the main results are given in Section \ref{sect:proofs_all}. Statements from Sections \ref{sect:maxgar} and  \ref{sect:sqrootWilks} are used for the proofs in Section \ref{sect:proofs_all}.
Numerical experiments are described in Section \ref{sect:numer_examples}:
 we construct simultaneous confidence corridors for local constant and local quadratic regressions using both bootstrap and Monte Carlo procedures. The quality of the bootstrap procedure is checked by computing the effective simultaneous coverage probabilities of the bootstrap confidence sets.  We also compare the widths of the confidence bands and the values of multiplicity correction obtained with bootstrap and with Monte Carlo procedures. The experiments confirm that the multiplier bootstrap and the bootstrap multiplicity correction become conservative if the local parametric model is considerably misspecified.

 The results given here are valid on a random set of probability \(1-C\ex^{-\yy}\) for some explicit constant \(C>0\). The number \(\yy>0\) determines this dominating probability level. For the case of the i.i.d. observations (see Secion \ref{typical_local}) we take \(\yy= \CONST \log{n}\).
 Throughout the text \(\|\cdot\|\) denotes the Euclidean norm for a vector and spectral norm for a matrix.
\(\|\cdot\|_{\max}\) is the maximal absolute value of elements of a vector (or a matrix), \(\dimtotal\eqdef \dimp_{1}+\dots+\dimp_{\numK}\), \(\dimp_{\max}\eqdef\max\limits_{1\leq k\leq \numK}\dimp_{k}\).

\section{The multiplier bootstrap procedure}
\label{sect:bootstr_intro}
Let \(\ell_{i,k}(\thetav)\) denote the log-density from the \(k\)-th parametric distribution family evaluated at the \(i\)-th observation:
\begin{EQA}[c]
\label{def:ell_ik}
	\ell_{i,k}(\thetav)\eqdef  \log\left(
\frac{d\P_{k}(\thetav)}{d\mu_{0}}(\Yv_{i})
\right),
\end{EQA}
then due to independence of \(\Yv_{1},\dots,\Yv_{n}\)
\begin{EQA}
L_{k}(\thetav)&=&\sum\nolimits_{i=1}^{n} \ell_{i,k}(\thetav) \quad \forall\, k=1,\dots,\numK.
\end{EQA}
Consider i.i.d. scalar random variables \(u_{i}\) independent of the data \(\Ym\), s.t. \(\E u_{i}=1\), \(\Var u_{i}=1\), \(\E\exp(u_{i})<\infty\) (e.g. \(u_{i}\sim \mathcal{N}(1,1)\) or \(u_{i}\sim exp(1)\) or \(u_{i}\sim 2Bernoulli(0.5)\)).
Multiply the summands of the likelihood function \(L_{k}(\thetav)\) with the new random variables:
\begin{EQA}[c]
\label{def:Lbk}
	\Lb_{k}(\thetav)\eqdef\sum\nolimits_{i=1}^{n} \ell_{i,k}(\thetav)u_{i},
\end{EQA}
then it holds
	\(\Eb\Lb_{k}(\thetav)=L_{k}(\thetav)\),
where \( \Eb \) stands for the conditional expectation given \(\Ym\).

Therefore, the quasi MLE for the \(\Ym\)-world is a target parameter for the bootstrap world for each \(k=1,\dots,\numK\):
\begin{EQA}[c]
	\argmax\nolimits_{\thetav\in\Theta_{k}}\Eb \Lb_{k}(\thetav)
	=
	\argmax\nolimits_{\thetav\in\Theta_{k}}L_{k}(\thetav)
	=
	\tilde{\thetav}_{k}.
\end{EQA}
The corresponding bootstrap MLE is:
\begin{EQA}[c]
\thetavbt_{k}\eqdef\argmax\nolimits_{\thetav\in\Theta_{k}}\Lb_{k}(\thetav).
\end{EQA}

The \(k\)-th likelihood ratio statistic in the bootstrap world equals to \(\Lb_{k}(\thetavbt_{k})-\Lb_{k}(\thetavt_{k})\), where all the elements: the function \(\Lb_{k}(\thetav)\) and the arguments \(\thetavbt_{k}\), \(\thetavt_{k}\) are known and available for  computation. This means, that given the data \(\Ym\), one can estimate the distribution or quantiles of the statistic \(\Lb_{k}(\thetavbt_{k})-\Lb_{k}(\thetavt_{k})\) by generating many independent samples of the bootstrap weights \(u_{1},\dots,u_{n}\) and computing with them the bootstrap likelihood ratio.

Let us introduce similarly to \eqref{def:zzalpha_k} the \((1-\alpha)\)-quantile for the bootstrap square-root likelihood ratio statistic:
\begin{EQA}
\label{def:zzbalpha_k}
\zzb_{k}(\alpha)&\eqdef&
\inf\left\{\zz\geq 0: \Pb\left(\Lb_{k}(\thetavbt_{k})-\Lb_{k}(\thetavt_{k})> \zz^{2}/2\right)\leq \alpha\right\},
\end{EQA}
here \(\Pb\) denotes probability measure conditional on the data \(\Ym\), therefore, \(\zzb_{k}(\alpha)\) is a random value dependent on \(\Ym\).

\cite{SpZh2014PMB} considered the case of a single parametric model (\(\numK=1\)), and showed that the bootstrap quantile \(\zzb_{k}(\alpha)\) is close to the true one \(\zz_{k}(\alpha)\) under a so called ``Small Modeling Bias'' (SmB) condition, which is fulfilled when the true distribution is close to the parametric family or when the observations are i.i.d. When the SmB condition does not hold, the bootstrap quantile is still valid, however, it becomes  conservative. Therefore, for each fixed \(k=1,\dots,\numK\) the bootstrap quantiles \(\zzb_{k}(\alpha)\) are rather good estimates for the true unknown ones \(\zz_{k}(\alpha)\), however, they are still ``pointwise'' in \(k\), i.e. the confidence bounds \eqref{bound:pointwise} hold for each \(k\) separately.
Our goal here is to estimate \(\zz_{1}(\alpha),\dots,\zz_{\numK}(\alpha)\) and \(\qqq(\alpha)\) according to \eqref{bound:simult_goal} and \eqref{bound:simult}. Let us introduce the bootstrap correction for multiplicity:
\begin{EQA}[c]
\label{def:mcorrqb}
\hspace{-0.8cm}
\qqqb(\alpha)\eqdef
\sup\left\{ c \in(0,\alpha]:
\Pb\left(\bigcup\nolimits_{k=1}^{\numK}
\Bigl\{
\sqrt{2\Lb_{k}(\thetavbt_{k})-2\Lb_{k}(\thetavt_{k})}> \zzb_{k}\left(c\right)
\Bigr\}
\right)
\leq
\alpha
\right\}.
\end{EQA}
By its definition \(\qqqb(\alpha)\) depends on the random sample \(\Ym\).

The  multiplier bootstrap procedure below explains how to estimate the bootstrap quantile functions \(\zzb_{k}\left(\qqqb(\alpha)\right)\) corrected for multiplicity.\\
\rule[-0.05cm]{\textwidth}{0.15ex}\\
\makebox[\textwidth][l]{\textbf{The simultaneous bootstrap procedure:}}\\
\rule[0.25cm]{\textwidth}{0.15ex}
\vspace{-1.2cm}
\begin{itemize}
\item[\textbf{Input:}\hspace{-1cm}]\hspace{1cm}
The data \(\Ym\) (as in \eqref{def:Ymatr}) and a fixed confidence level \((1-\alpha)\in (0,1).\)
\item[\textbf{Step 1:}\hspace{-1cm}]\hspace{1cm}
\makebox[\textwidth][l]{Generate \(B\) independent samples of i.i.d. bootstrap weights \(\{u_{1}^{(b)},\dots,u_{n}^{(b)}\}\),}\\
\makebox[\textwidth][l]{\hspace{1.1cm}\(b=1,\dots,B\). For the bootstrap likelihood processes}
\begin{EQA}[c]
L^{\sbt (b)}_{k}(\thetav)\eqdef\sum\nolimits_{i=1}^{n} \ell_{i,k}(\thetav)u_{i}^{(b)}.
\label{algo:Lbb}
 \end{EQA}
\makebox[\textwidth][l]{\hspace{1cm} compute the bootstrap likelihood ratios \(L^{\sbt(b)}_{k}({\thetav}^{\sbt(b)}_{k})-L^{\sbt(b)}_{k}(\thetavt_{k})\). For each}\\
\makebox[\textwidth][l]{\hspace{1cm}
fixed \(b\) the bootstrap likelihoods \(L^{\sbt(b)}_{1}(\thetav),\dots,L^{\sbt(b)}_{\numK}(\thetav)\) are computed using}\\
 \makebox[\textwidth][l]{\hspace{1cm}
 the same  bootstrap sample \(\{u_{i}^{(b)}\}\), s.t. the \(i\)-th
  summand \(\ell_{i,k}(\thetav)\) is always
 }\\
  \makebox[\textwidth][l]{\hspace{1cm}
 multiplied with the \(i\)-th weight \(u_{i}^{(b)}\) as in \eqref{algo:Lbb}.}
\item[\textbf{Step 2:}\hspace{-1cm}]\hspace{0.9cm}
\makebox[\textwidth][l]{
Estimate the marginal quantile functions \(\zzb_{k}(\alpha)\) defined in \eqref{def:zzbalpha_k} separately } \\
\makebox[\textwidth][l]{\hspace{1cm}
for each \(k=1,\dots,\numK\), using \(B\) bootstrap realisations of \(\Lb_{k}(\thetavbt_{k})-\Lb_{k}(\thetavt_{k})\)}\\
\makebox[\textwidth][l]{\hspace{1cm}
from \textbf{Step 1}.
}
\item[\textbf{Step 3:}\hspace{-1cm}]
\hspace{1cm} Find by an iterative procedure the maximum value
\(c\in (0,\alpha]\) s.t. 
\begin{EQA}[c]
\Pb\left(\bigcup\nolimits_{k=1}^{\numK}
\Bigl\{
\sqrt{2\Lb_{k}(\thetavbt_{k})-2\Lb_{k}(\thetavt_{k})}\geq \zzb_{k}\left(c\right)
\Bigr\}
\right)
\leq
\alpha.
\end{EQA}
\item[\textbf{Otput:}\hspace{-1cm}]\hspace{1cm}
The resulting critical values are \(\zzb_{k}\left(c\right)\), \(k=1,\dots,\numK\).
\end{itemize}
\rule[0.3cm]{\textwidth}{0.15ex}
\begin{newremark}
The requirement in Step 1  to use the same bootstrap sample \(\{u_{i}^{(b)}\}\) for generation of the bootstrap likelihood ratios \(L^{\sbt(b)}_{k}({\thetav}^{\sbt(b)}_{k})-L^{\sbt(b)}_{k}(\thetavt_{k})\), \(k=1,\dots,\numK\) allows to preserve the correlation structure between the ratios and, therefore, to make a sharper simultaneous adjustment in Step 3.
\end{newremark}

This procedure is justified theoretically in the next section. 

\section{Theoretical justification of the bootstrap procedure}
\label{sect:theorres}
Before stating the main results in Section \ref{sect:mainres} we introduce in Section \ref{sect:scheme} the basic ingredients of the proofs. The general scheme of the theoretical approach here is taken from  \cite{SpZh2014PMB}. In the present work we extend that approach for the case of simultaneously many parametric models.


\subsection{Overview of the theoretical approach}
\label{sect:scheme}
For justification of the described multiplier bootstrap procedure for simultaneous inference it has to be checked that the joint distributions of the sets of likelihood ratio statistics \(\left\{ L_{k}(\thetavt_{k})-L_{k}(\thetavs_{k}): k=1,\dots,\numK\right\}\) and \(\left\{\Lb_{k}(\thetavbt_{k})-\Lb_{k}(\thetavt_{k}): k=1,\dots,\numK\right\}\) are close to each other. These joint distributions are approximated using several non-asymptotic  steps given in the following scheme:
\begin{EQA}[lcccccr]
	&& \substack{\text{uniform}\vspace{0.09cm}\\\text{sq-Wilks}\\
	\text{theorem}}&&\substack{\text{joint Gauss.}\\ \text{approx. \& }\\\text{anti-concentr.}^{\ast}}&&
	\\
	\nquad\Ym\text{-\scriptsize{world: }}&\sqrt{ 2L_{k}(\thetavt_{k})-2L_{k}(\thetavs_{k})}
	 &~\underset{\substack{\vspace{0.01cm}\\\frac{\dimp_{k}+\log\numK}{\sqrt{n}}}}{\approx}~&
	\|\xiv_{k}\|
	&
	{\approx}&\|\xivbr_{k}\|&
	\\
\label{rectangle_simult}
&
&
&
{\Scale[1.2]{\bigcap\limits_{1\leq k\leq \numK}}}
&
& ~~~\text{\rotatebox{90}{$\approx$}}~ \highw
	&\substack{\text{simultaneous}\vspace{0.1cm}\\\text{ Gauss.\,compar.}^{\ast\ast}\vspace{0.05cm}
	}\\%
\vspace{-0.11cm}
	\nquad\substack{\text{Bootstrap}\\
	\text{world: }}&\sqrt{2\Lb_{k}(\thetavbt_{k})-2\Lb_{k}(\thetavt_{k}) }
	 &\underset{\substack{\vspace{0.01cm}\\\frac{\dimp_{k}+\log\numK}{\sqrt{n}}}}{\approx }&
	\|\xivb_{k}\|
	&
	{\approx}&\|{\xivbr}^{\sbt}_{k}\|,
\end{EQA}
\begin{itemize}
\item[\(^{\ast}\)] the accuracy of these approximating steps is \(\CONST\left\{\frac{\dimp_{\max}^{3}}{n}\log^{9}(K)\log^{3}(n\dimtotal)\right\}^{1/8}\);
\item[\(^{\ast\ast}\)] Gaussian comparison step yields an approximation error proportional to\\
\(\DeltaSmB^{2}\left(\frac{\dimp_{\max}^{3}}{n}\right)^{1/4}\dimp_{\max}\log^{2}(K)\log^{3/4}(n\dimtotal)\), where \(\DeltaSmB^{2}\) comes from condition \ref{itm:SmBHk},  see also \eqref{ineq:DeltaSmB_intro} below.
\end{itemize}
Here \(\xiv_{k}\) and \(\xivb_{k}\) denote normalized score vectors for the  \(\Ym\) and bootstrap likelihood processes:
\begin{EQA}[c]
\label{def:xivkxivbr}
\xiv_{k}\eqdef D^{-1}_{k}\nabla_{\thetav}L_{k}(\thetavs_{k}), \quad\quad \xivb_{k}\eqdef \xivb_{k}(\thetavs_{k})\eqdef D^{-1}_{k}\nabla_{\thetav}L_{k}(\thetavs_{k}),
\end{EQA}
\(D^{2}_{k}\) is the full Fisher information matrix for the corresponding \(k\)-th likelihood:
\begin{EQA}[c]
\label{def:Dk}
D^{2}_{k}\eqdef -\nabla_{\thetav}^{2}\E L_{k}(\thetavs_{k}).
\end{EQA}
\(\xivbr_{k}\sim \mathcal{N}(0,\Var\xiv_{k})\) and \(\xivbbr_{k}\sim \mathcal{N}(0,\Varb\xivb_{k})\) denote approximating Gaussian vectors, which have the same covariance matrices as \(\xiv\) and \(\xivb\). Moreover the vectors \(\left(\xivbr_{1}^{\T},\dots,\xivbr_{\numK}^{\T}\right)^{\T}\) and  \(\left({\xivbbr}_{1}^{\T},\dots,{\xivbbr}_{\numK}^{\T}\right)^{\T}\) are normally distributed and have the same covariance matrices as the vectors \(\left(\xiv_{1}^{\T},\dots,\xiv_{\numK}^{\T}\right)^{\T}\) and \(\left({\xivb}_{1}^{\T},\dots,{\xivb}_{\numK}^{\T}\right)^{\T}\) correspondingly. \(\Varb\) and \(\Covb\) denote variance and covariance operators w.r.t. the probability measure \(\Pb\) conditional on \(\Ym\).

The first two approximating steps: square root Wilks and Gaussian approximations are performed in parallel for both \(\Ym\) and bootstrap  worlds, which is shown in the corresponding lines of the scheme \eqref{rectangle_simult}. The two worlds are connected in the last step: Gaussian comparison for \(\ell_{2}\)-norms of Gaussian vectors. All the approximations are performed simultaneously for \(\numK\) parametric models. 

Let us consider each step in more details. Non-asymptotic square-root Wilks approximation result had been obtained recently by \cite{Sp2012Pa,Spokoiny2013Bernstein}. It says that for a fixed sample size and misspecified parametric assumption: \(\P\notin\{\P_{k}\}\), it holds with exponentially high probablity:
\begin{EQA}
\label{Wilks_intro}
	\left|
		 \sqrt{2\bigl\{L_{k}(\thetavt_{k})-L_{k}(\thetavs_{k})\bigr\}}-\|\xiv_{k}\|
	\right|
	&\leq&
	\DeltakW \simeq \frac{\dimp_{k}}{\sqrt{n}},
\end{EQA}
here the index \(k\) is fixed, i.e. this statement is for one parametric model. The precise statement of this result is given in Section \ref{sect:FS}, and its simultaneous version -- in Section \ref{sect:unif_Wilks}. The approximating value \(\|\xiv_{k}\|\) is \(\ell_{2}\)-norm of the score vector \(\xiv_{k}\) given in \eqref{def:xivkxivbr}.
The next approximating step is between the joint distributions of \(\|\xiv_{1}\|, \dots, \|\xiv_{\numK}\|\) and \(\|\xivbr_{1}\|, \dots, \|\xivbr_{\numK}\|\). This is done in Section \ref{sect:gar} for general centered random vectors under bounded exponential moments assumptions. The main tools for the simultaneous Gaussian approximation are: Lindeberg's telescopic sum, smooth maximum function and three times differentiable approximation of the indicator function \(\Ind\{x\in\R:x> 0\}\). The simultaneous anti-concentration inequality for the \(\ell_{2}\)-norms of Gaussian vectors is obtained in Section \ref{sect:ac}. The result is based on approximation of the \(\ell_{2}\)-norm with a maximum over a finite grid on a hypersphere, and on the  anti-concentration inequality for maxima of a Gaussian random vector by \cite{chernozhukov2012comparison}.
The same approximating steps are performed for the bootstrap world, the square-root bootstrap Wilks approximation is given in Sections \ref{sect:FStheory}, \ref{sect:unif_Wilks}.
The last step in the scheme \eqref{rectangle_simult} is comparison of the joint distributions of the sets of \(\ell_{2}\)-norms of Gaussian vectors: \(\|\xivbr_{1}\|, \dots, \|\xivbr_{\numK}\|\) and \(\|\xivbbr_{1}\|, \dots, \|\xivbbr_{\numK}\|\) by Slepian interpolation (see Section \ref{sect:Gausscompar} for the result in a general setting). The error of approximation is proportional to
\begin{EQA}[c]
\label{def:termcompar}
\max_{1\leq k_{1},k_{2}\leq \numK}
\left\|
\Cov(\xiv_{k_{1}},\xiv_{k_{2}})
-
\Covb(\xivb_{k_{1}},\xivb_{k_{2}})
\right\|_{\max}.
\end{EQA}
It is shown, using Bernstein matrix inequality (Sections \ref{sect:NCBI_k} and \ref{sect:proofsofmainres}), that the value \eqref{def:termcompar} is bounded from above (up to a constant) on a random set of dominating probability with
\begin{EQA}
\label{ineq:DeltaSmB_intro}
\max_{1\leq k\leq\numK}
\left\|
H^{-1}_{k}B_{k}^{2}
H^{-1}_{k}
\right\|
&\leq&
\DeltaSmB^{2}
\end{EQA}
for
\begin{align}
\label{def:Bsmbk}
\begin{split}
B_{k}^{2}&\eqdef\sum\nolimits_{i=1}
^{n}
\E\left\{\nabla_{\thetav}\ell_{i,k}(\thetavs_{k})\right\}
\E\left\{\nabla_{\thetav}\ell_{i,k}(\thetavs_{k})\right\}^{\T},\\
H_{k}^{2}&\eqdef\sum\nolimits_{i=1}
^{n}
\E\left\{\nabla_{\thetav}\ell_{i,k}(\thetavs_{k})\nabla_{\thetav}\ell_{i,k}(\thetavs_{k})^{\T}\right\}.
\end{split}
\end{align}
The value \(\left\|
H^{-1}_{k}B_{k}^{2}
H^{-1}_{k}
\right\|\) is responsible for the modelling bias of the \(k\)-th model. If the parametric family \(\left\{\P_{k}(\thetav)\right\}\) contains the true distribution \(\P\) or if the observations \(\Yv_{i}\) are i.i.d., then \(B_{k}^{2}\) equals to zero. Condition \ref{itm:SmBHk} assumes that all the values \(\left\|
H^{-1}_{k}B_{k}^{2}
H^{-1}_{k}
\right\|\) are rather small.
\subsection{Main results}
\label{sect:mainres}
The following theorem shows the closeness of the joint cumulative distribution functions (c.d.f-s.) of \(\Bigl\{\sqrt{2L_{k}(\thetavt_{k})-2L_{k}(\thetavs_{k})}, k=1,\dots,\numK\Bigr\}\) and \(\Bigl\{\sqrt{2\Lb_{k}(\thetavbt_{k})-2\Lb_{k}(\thetavt_{k})}, k=1,\dots,\numK\Bigr\}\). The approximating error term \(\Delta_{\total}\) equals to a sum of the errors from all the  steps in the scheme \eqref{rectangle_simult}.
\begin{theorem}
\label{thm:mainres1}
Under the conditions of Section \ref{sect:conditions} it holds with probability \(\geq1-12\ex^{-\yy}\) for \(z_{k}\geq C\sqrt{\dimp_{k}},\ 1\leq C<2\)
\begin{EQA}
&&\hspace{-1cm}
\Biggl|
\P\left(\bigcup\nolimits_{k=1}^{\numK}
\Bigl\{
\sqrt{2L_{k}(\thetavt_{k})-2L_{k}(\thetavs_{k})}> z_{k}
\Bigr\}
\right)
\\
&&\hspace{2cm}
-\,
\Pb\left(\bigcup\nolimits_{k=1}^{\numK}
\Bigl\{
\sqrt{2\Lb_{k}(\thetavbt_{k})-2\Lb_{k}(\thetavt_{k})}> z_{k}
\Bigr\}
\right)
\Biggr|
\leq \Delta_{\total}.
\end{EQA}
The approximating total error \(\Delta_{\total}\geq 0\) is deterministic and in the case of i.i.d. observations (see Section  \ref{typical_local}) it holds:
\begin{EQA}
\label{ineq:Delta_total_statement}
\Delta_{\total}
&\leq&
\CONST\left(\frac{\dimp_{\max}^{3}}{n}\right)^{1/8}\log^{9/8}(K)\log^{3/8}(n\dimtotal)
\left\{\left(\bbgmu^{2}+\bbgmu_{B}^{2}\right)
\left(1+\deltabbVb^{2}(\yy)\right)
\right\}^{3/8},
\end{EQA}
where the deterministic terms \(\bbgmu^{2},\bbgmu_{B}^{2}\) and \(\deltabbVb^{2}(\yy)\) come from the conditions \ref{itm:Ik}, \ref{itm:IBk} and \ref{itm:SD0k}. \(\Delta_{\total}\) is defined in \eqref{def:Delta_total_1}.
\end{theorem}
\begin{newremark}
The obtained approximation bound is mainly of theoretical interest, although it shows the impact of \(\dimp_{\max}\), \(\numK\) and \(n\) on the quality of the bootstrap procedure. For more details on the error term see Remark \ref{rem:Deltapp}.
\end{newremark}

The next theorem justifies the bootstrap procedure under the \ref{itm:SmBHk} condition. The theorem says that the bootstrap quantile functions \(\zzb_{k}(\cdot)\) with the bootstrap-corrected for multiplicity confidence levels \(1-\qqqb(\alpha)\) can be used for construction of the simultaneous confidence set in the \(\Ym\)-world.
\begin{theorem}[Bootstrap validity for a small modeling bias]
\label{thm:mainres1_2}
Assume the conditions of Theorem \ref{thm:mainres1}, and \(\qqq(\alpha),0.5\qqqb(\alpha) \geq \Delta_{\full,\,\max}\),
 then for \(\alpha\leq 1-8\ex^{-\xx}\) it holds with probability \(1-12\ex^{-\xx}\)
\begin{EQA}
\P\left(\bigcup\nolimits_{k=1}^{\numK}
\Bigl\{
\sqrt{2L_{k}(\thetavt_{k})-2L_{k}(\thetavs_{k})}\geq \zzb_{k}\left({\qqqb}(\alpha)- 2\Delta_{\full,\,\max}\right)
\Bigr\}
\right)-\alpha &\leq& \Delta_{\zz,\,\total} ,\\
\P\left(\bigcup\nolimits_{k=1}^{\numK}
\Bigl\{
\sqrt{2L_{k}(\thetavt_{k})-2L_{k}(\thetavs_{k})}\geq \zzb_{k}\left({\qqqb}(\alpha)+ 2\Delta_{\full,\,\max}\right)
\Bigr\}
\right)-\alpha &\geq& - \Delta_{\zz,\,\total} ,
\end{EQA}
where \(\Delta_{\full,\,\max}\leq \CONST\{(\dimp_{\max}+\yy)^{3}/n\}^{1/8}\)  in the case of i.i.d. observations (see Section \ref{typical_local}), and \(\Delta_{\zz,\,\total}\leq 3\Delta_{\total}\); their explicit definitions are given in \eqref{def:Delta_full_max} and \eqref{def:Delta_z_total}.
Moreover
\begin{EQA}
{\qqqb}(\alpha)
 &\leq&
 {\qqq}\left(\alpha+\Delta_{\qqq}\right)
 +\Delta_{\full,\,\max},
 \\
{\qqqb}(\alpha)&\geq& {\qqq}\left(\alpha-\Delta_{\qqq}
\right)  -\Delta_{\full,\,\max},
\end{EQA}
for \(0\leq\Delta_{\qqq}\leq 2 \Delta_{\total}\), defined   in \eqref{def:Delta_qqq}.
\end{theorem}

The following theorem does not assume the \ref{itm:SmBHk} condition to be fulfilled. It turns out that in this case the bootstrap procedure becomes conservative, and the bootstrap critical values corrected for the multiplicity \(\zzb_{k}\left(\qqqb(\alpha)\right)\)  are increased with the modelling bias \(\sqrt{\tr\{D_{k}^{-1}H_{k}^{2}D_{k}^{-1}\}}-\sqrt{\tr\{D_{k}^{-1}(H_{k}^{2}-B_{k}^{2})D_{k}^{-1}\}}\), therefore, the confidence set based on the bootstrap estimates can be conservative. 

\begin{theorem}[Bootstrap conservativeness for a large modeling bias]
\label{thm:mainres2_largesmb}
Under the conditions of Section \ref{sect:conditions} except for \ref{itm:SmBHk}  it holds with probability \(\geq1-14\ex^{-\yy}\) for \(z_{k}\geq C\sqrt{\dimp_{k}},\ 1\leq C<2\)
\begin{EQA}
&&\hspace{-0.5cm}
\P\left(\bigcup\nolimits_{k=1}^{\numK}
\Bigl\{
\sqrt{2L_{k}(\thetavt_{k})-2L_{k}(\thetavs_{k})}> z_{k}
\Bigr\}
\right)
\\
&&\hspace{2.5cm}
\leq
\Pb\left(\bigcup\nolimits_{k=1}^{\numK}
\Bigl\{
\sqrt{2\Lb_{k}(\thetavbt_{k})-2\Lb_{k}(\thetavt_{k})}> z_{k}
\Bigr\}
\right)+
\Delta_{\operatorname{b},\,\total}.
\end{EQA}
The deterministic value \(\Delta_{\operatorname{b},\,\total}\in[0,\Delta_{\total}]\) (see \eqref{ineq:Delta_total_statement} in the case \ref{typical_local}).
Moreover, the bootstrap-corrected for multiplicity confidence level \(1-\qqqb(\alpha)\) is conservative in comparison with the true corrected confidence level:
\begin{EQA}
\label{ineq:th2qqq}
1-\qqqb(\alpha) &\geq&
1-{\qqq}\left(\alpha+\Delta_{\operatorname{b},\qqq}\right)
 -\Delta_{\full,\,\max}
,
\end{EQA}
and it holds for all \(k=1,\dots,\numK\) and \(\alpha \leq 1-8\ex^{-\xx}\)  
\begin{EQA}
\zzb_{k}\left(\qqqb(\alpha)\right)&\geq&
\zz_{k}\left({\qqq}\left(\alpha+\Delta_{\operatorname{b},\qqq}\right)
 +\Delta_{\full,\,\max}\right)
\\&+&\sqrt{\tr\{D_{k}^{-1}H_{k}^{2}D_{k}^{-1}\}}-\sqrt{\tr\{D_{k}^{-1}(H_{k}^{2}-B_{k}^{2})D_{k}^{-1}\}}
 - \Deltaqqfkk,
\end{EQA}
for \(0\leq\Delta_{\operatorname{b},\qqq}\leq 2 \Delta_{\total}\), defined   in \eqref{def:Delta_b_qqq}, and
the positive value \(\Deltaqqfkk\) is
bounded from above with \((\gmu_{k}^{2}+\gmu_{B,k}^{2})(\sqrt{8\yy\dimp_{k}}+6\yy)\) for the constants \(\gmu_{k}^{2}>0,\, \gmu_{B,k}^{2}\geq0\) from  conditions \ref{itm:Ik}, \ref{itm:IBk}.
\end{theorem}
The \ref{itm:SmBHk} condition is automatically fulfilled if all the parametric models are correct or in the case of i.i.d. observations. This condition is checked for generalised linear model and linear quantile regression in \cite{SpZh2014PMB} (the version of 2015).

\section{Numerical experiments}
\label{sect:numer_examples}

Here we check the performance of the bootstrap procedure by constructing  
 simultaneous confidence sets based on the local constant and local quadratic estimates, the former one is also known as Nadaraya-Watson estimate \cite{Nadaraya1964estimating,Watson1964smooth}.
Let \(Y_{1},\dots,Y_{n}\) be independent random scalar observations and \(X_{1},\dots,X_{n}\) some deterministic design points. In Sections \ref{sect:numer_localconst}-\ref{sect:simuldata} below we introduce the models and the data, Sections \ref{sect:numereffect}-\ref{sect:numercorrection} present the results of the experiments.
\subsection{Local constant regression}
\label{sect:numer_localconst}
Consider the following quadratic likelihood function reweighted with the kernel functions \(K(\cdot)\):
\begin{EQA}[c]
L(\thetav,x,h)\eqdef -\frac{1}{2}\sum\nolimits_{i=1}^{n} (Y_{i}-\thetav)^{2}w_{i}(x,h),\\
w_{i}(x,h)\eqdef K(\{x-X_{i}\}/h),\\
K(x)\in[0,1],\ \int_{\R}K(x)dx =1,\ K(x)=K(-x).
\end{EQA}
Here \(h>0\) denotes bandwidth, the local smoothing parameter. The target point and the local MLE read as:
\begin{EQA}[c]
\thetavs(x,h)\eqdef  \frac{ \sum\nolimits_{i=1}^{n}w_{i}(x,h)\E Y_{i} }{\sum\nolimits_{i=1}^{n}w_{i}(x,h)},~~~~
\thetavt(x,h)\eqdef \frac{ \sum\nolimits_{i=1}^{n}w_{i}(x,h) Y_{i} }{\sum\nolimits_{i=1}^{n}w_{i}(x,h)}.
\end{EQA}
Let us fix a bandwidth \(h\) and consider the range of points  \(x_{1},\dots,x_{\numK}\). They yield \(\numK\) local constant models with the target parameters \(\thetavs_{k}\eqdef\thetavs(x_{k},h)\) and the likelihood functions \(L_{k}(\thetav)\eqdef L(\thetav,x_{k},h)\) for \(k=1,\dots,\numK\). 

The bootstrap local likelihood function is defined similarly to the global one \eqref{def:Lbk}, by reweighting \(L(\thetav,x,h)\) with the bootstrap multipliers \(u_{1},\dots,u_{n}\):
\begin{EQA}[c]
\Lb_{k}(\thetav)\eqdef
\Lb(\thetav,x_{k},h)\eqdef -\frac{1}{2}\sum\nolimits_{i=1}^{n} (Y_{i}-\thetav)^{2}w_{i}(x_{k},h)u_{i},\\
\thetavbt_{k}\eqdef
\thetavbt(x_{k},h)\eqdef
\frac{ \sum\nolimits_{i=1}^{n}w_{i}(x_{k},h)u_{i} Y_{i} }{\sum\nolimits_{i=1}^{n}w_{i}(x_{k},h)u_{i}}.
\end{EQA}
\subsection{Local quadratic regression}
Here the local likelihood function reads as
\begin{EQA}[c]
L(\thetav,x,h)\eqdef -\frac{1}{2}\sum\nolimits_{i=1}^{n} (Y_{i}-\Psi_{i}^{\T}\thetav)^{2}w_{i}(x,h),\\
\thetav,\Psi_{i}\in\R^{3}, \quad \Psi_{i}\eqdef\left(1,X_{i},X_{i}^{2}\right)^{\T},
\end{EQA}
and
\begin{EQA}
\thetavs(x,h)&\eqdef&\left(\Psi W(x,h)\Psi^{\T}\right)^{-1}\Psi W(x,h) \E \Ym,\\
\thetavt(x,h)&\eqdef&\left(\Psi W(x,h)\Psi^{\T}\right)^{-1}\Psi W(x,h)\Ym,
\end{EQA}
where
\begin{EQA}[c]
\Ym\eqdef \left(Y_{1},\dots,Y_{n}\right)^{\T},\quad
\Psi\eqdef\left(\Psi_{1},\dots,\Psi_{n}\right)\in \R^{3\times n},\\
W(x,h)\eqdef\diag\left\{w_{1}(x,h),\dots,w_{n}(x,h)\right\}.
\end{EQA}
And similarly for the bootstrap objects
\begin{EQA}
\Lb(\thetav,x,h)&\eqdef& -\frac{1}{2}\sum\nolimits_{i=1}^{n} (Y_{i}-\Psi_{i}^{\T}\thetav)^{2}w_{i}(x,h)u_{i},\\
\thetavbt(x,h)&\eqdef& \left(\Psi U W(x,h)\Psi^{\T}\right)^{-1}\Psi U W(x,h)\Ym,
\end{EQA}
for \(U\eqdef \diag\left\{u_{1},\dots,u_{n}\right\}\).

\subsection{Simulated data}
\label{sect:simuldata}
In the numerical experiments we constructed two \(90\%\) simultaneous confidence bands: using Monte Carlo (MC) samples and bootstrap procedure with Gaussian weights (\(u_{i}\sim \mathcal{N}(1,1)\)), in each case we used \(10^{4}\) \(\{Y_{i}\}\) and \(10^{4}\) \(\{u_{i}\}\) independent samples. The sample size \(n=400\). \(K(x)\) is Epanechnikov's kernel function. The independent random observations \(Y_{i}\) are generated as follows:
\begin{EQA}[c]
\label{def:numer_Y}
Y_{i}= f(X_{i})+ \mathcal{N}(0,1),\quad
X_{i} \text{ are equidistant on } [0,1],\\
\label{def:numer_f}
f(x)=
\begin{cases}
5,& x\in [0,0.25]\cup[0.65,1];\\
5+3.8\{1-100(x-0.35)^2\},& x\in [0.25,0.45];\\
5-3.8\{1-100(x-0.55)^2\},& x\in [0.45,0.65].
\end{cases}
\end{EQA}
The number of local models \(\numK=71\), the points \(x_{1},\dots,x_{71}\) are equidistant on \([0,1]\). For the bandwidth we considered two cases: \(h=0.12\) and \(h=0.3\).

\subsection{Effect of the modeling bias on a width of a bootstrap confidence band}
\label{sect:numereffect}
The function \(f(x)\) defined in \eqref{def:numer_f} should yield a considerable modeling bias for both mean constant and mean quadratic estimators.
Figures \ref{fig:localMLE},  \ref{fig:localMLE_quadr} demonstrate that the bootstrap confidence bands become conservative (i.e. wider than the MC confidence band) when the local model is misspecified. The top graphs on Figures \ref{fig:localMLE}, \ref{fig:localMLE_quadr} show the \(90\%\) confidence bands, the middle graphs show their width, and the bottom graphs show the value of the modelling bias  for \(K=71\) local models (see formulas \eqref{eq:numer_smb} and \eqref{eq:numer_smb_quadr} below). For the local constant estimate (Figure \ref{fig:localMLE}) the width of the bootstrap confidence sets is considerably increased by the modeling bias when \(x\in [0.25,0.65]\). In this case case the expression for the modeling bias term for the \(k\)-th model (see also\ref{itm:SmBHk} condition) reads as:
\begin{align}
\label{eq:numer_smb}
\begin{split}
\left|
H^{-1}_{k}B_{k}^{2}
H^{-1}_{k}
\right|&=
\frac{\sum\nolimits_{i=1}^{n}\left\{\E Y_{i}-\thetavs(x_{k})\right\}^{2}w_{i}^{2}(x_{k},h) }
{\sum\nolimits_{i=1}^{n}\E \left\{Y_{i}-\thetavs(x_{k})\right\}^{2}w_{i}^{2}(x_{k},h) }
\\&=
1-\left(1+\frac{\sum\nolimits_{i=1}^{n}w_{i}^{2}(x_{k},h)\left\{f(X_{i})-\thetavs(x_{k})\right\}^{2}}{\sum\nolimits_{i=1}^{n}w_{i}^{2}(x_{k},h)}
\right)^{-1}.
\end{split}
\end{align}
And for the local quadratic estimate it holds:
\begin{align}
\label{eq:numer_smb_quadr}
\left\|
H^{-1}_{k}B_{k}^{2}
H^{-1}_{k}
\right\|&=\left\|\Id_{\dimp}-H^{-1}_{k}\left\{\sum\nolimits_{i=1}^{n}\Psi_{i}\Psi_{i}^{\T}w_{i}^{2}(x_{k},h)\right\}H^{-1}_{k}\right\|,
\end{align}
where \(\Id_{\dimp}\) is the identity matrix of dimension \(\dimp\times\dimp\) (here \(\dimp=3\)), and
\begin{align}
\begin{split}
H^{2}_{k}&=\sum\nolimits_{i=1}^{n}\Psi_{i}\Psi_{i}^{\T}w_{i}^{2}(x_{k},h)\E \left\{Y_{i}-\thetavs(x_{k})\right\}^{2}
\\
&=
\sum\nolimits_{i=1}^{n}\Psi_{i}\Psi_{i}^{\T}w_{i}^{2}(x_{k},h)\left\{f(X_{i})-\thetavs(x_{k})\right\}^{2}
+\sum\nolimits_{i=1}^{n}\Psi_{i}\Psi_{i}^{\T}w_{i}^{2}(x_{k},h).
\end{split}
\end{align}
Therefore, if \(\max_{1\leq k \leq \numK}\left\{f(X_{i})-\thetavs(x_{k})\right\}^{2}= 0\), then \(\left\|H^{-1}_{k}B_{k}^{2}H^{-1}_{k}\right\|= 0\).
On the Figure \ref{fig:localMLE} both the modelling bias and the difference between the widths of the bootstrap and MC confidence bands are close to zero in the regions where the true function \(f(x)\) is constant. On Figure \ref{fig:localMLE_quadr} the modelling bias for \(h=0.12\) is overall smaller than the corresponding value on Figure \ref{fig:localMLE}. For the bigger bandwidth \(h=0.3\) the modelling biases on Figures \ref{fig:localMLE} and \ref{fig:localMLE_quadr} are comparable with each other.

Thus the numerical experiment is consistent with the theoretical results from Section \ref{sect:mainres}, and confirm that in the case when a (local) parametric model is close to the true distribution the simultaneous bootstrap confidence set is valid. Otherwise the bootstrap procedure is conservative: the modelling bias widens the simultaneous bootstrap confidence set.

\begin{figure}[!htbp]
\caption{\textbf{Local constant regression:}
\newline \hspace*{4.85em}
 Confidence bands,
their widths, and the modeling bias}
\label{fig:localMLE}
\begin{center}
  \begin{tabular}{ c c }
\hspace{-1.9cm}   {\Scale[0.85]{\text{bandwidth} \(=\,0.12\)}}
&
\hspace{-1.4cm}   {\Scale[0.85]{\text{bandwidth} \(=\,0.3\)}}\\
\hspace{-1.9cm}  \includegraphics[width=0.522\textwidth]{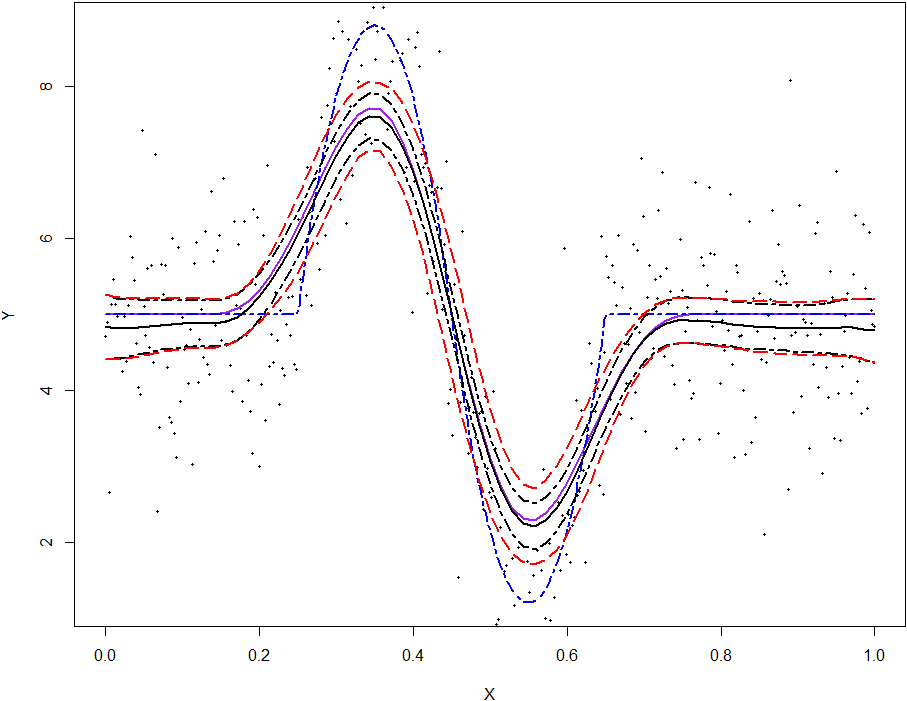}
&
\hspace{-1.4cm}
\includegraphics[width=0.5\textwidth]{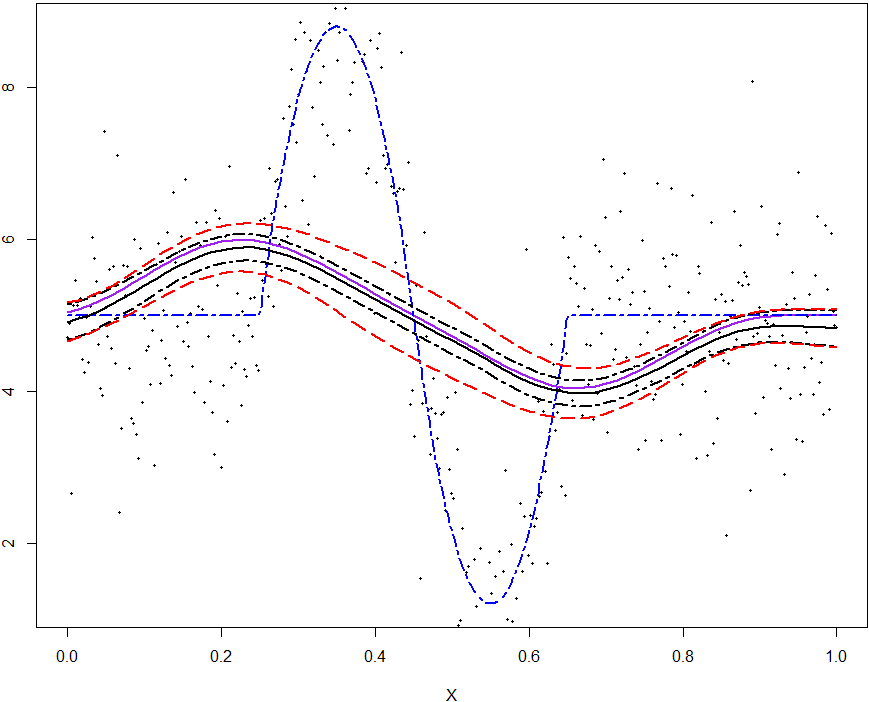}
  \\
\hspace{-1.9cm}
  \includegraphics[width=0.523\textwidth]{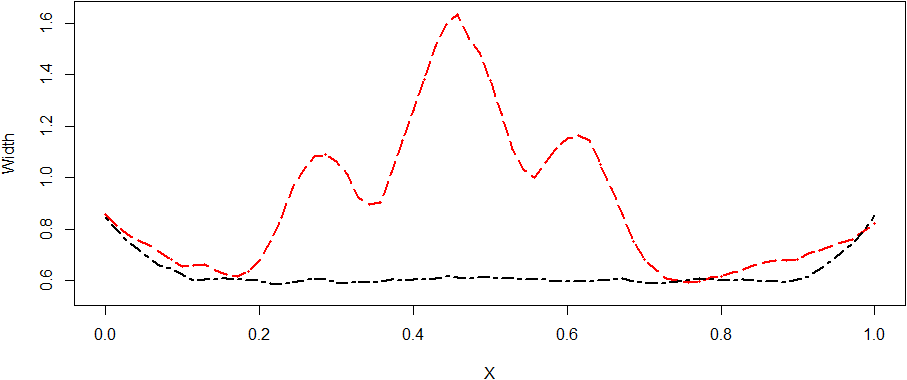}
&
\hspace{-1.4cm}
\includegraphics[width=0.5\textwidth]{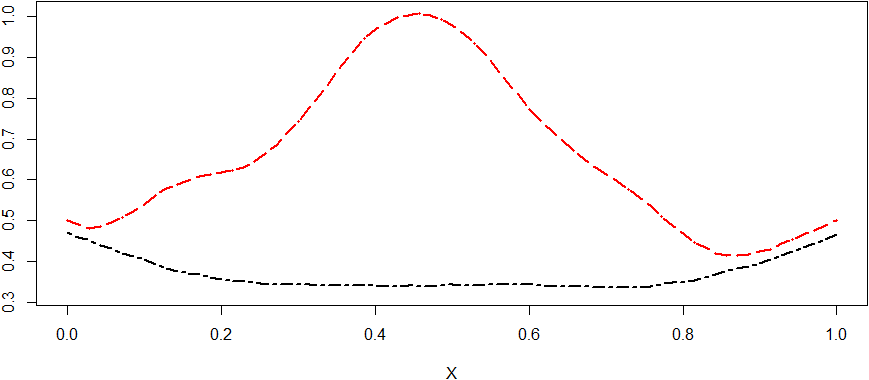}
\\
\hspace{-1.9cm}
  \includegraphics[width=0.523\textwidth]{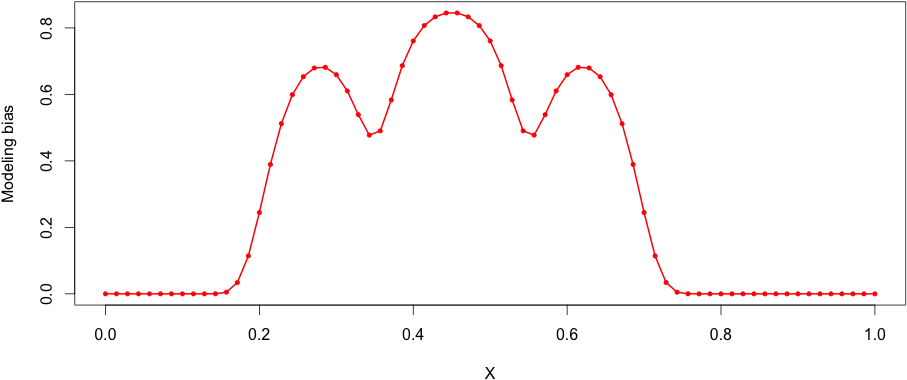}
  &
  \hspace{-1.4cm}
   \includegraphics[width=0.5\textwidth]{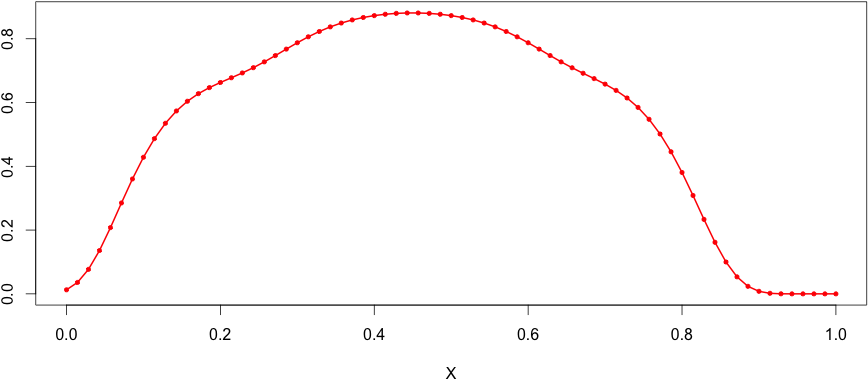}
  \\
    \multicolumn{2}{c}{
   {\Scale[0.85]{\text{Legend for the top graphs:}}}}\\
    \multicolumn{1}{l}{
\includegraphics[width=0.08\textwidth]{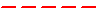}
 {{\Scale[0.85]{ \(90\%\) \text{ bootstrap simultaneous confidence band}}}}
  }
 &
 \multicolumn{1}{l}{
~~ \includegraphics[width=0.08\textwidth]{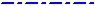}
  \,{\(\Scale[0.85]{\text{the true function } f(x)}\)}
  }
 \\
 \multicolumn{1}{l}{
\includegraphics[width=0.08\textwidth]{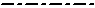}
 {\Scale[0.85]{\(90\%\)\text{ MC simultaneous confidence band}}}}
   &
 \multicolumn{1}{l}{
~~ \includegraphics[width=0.08\textwidth]{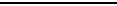}
  ~{\Scale[0.85]{\text{local constant MLE}}}
  }
  \\
 \multicolumn{2}{l}{
  \includegraphics[width=0.08\textwidth]{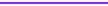}
  ~{\Scale[0.85]{\text{smoothed target function}}}}\\
  \multicolumn{2}{c}{
   {\Scale[0.85]{\text{Legend for the middle and the bottom graphs:}}}}\\
 \multicolumn{2}{l}{
\includegraphics[width=0.08\textwidth]{legend_boot}
 {{\Scale[0.85]{\text{width of the } \(90\%\) \text{ bootstrap confidence bands from the upper graphs}}}}
  }\\
  \multicolumn{2}{l}{
\includegraphics[width=0.08\textwidth]{legend_MC}
 {{\Scale[0.85]{ \text{width of the }\(90\%\)\text{ MC confidence bands from the upper graphs}}}}
  }\\
   \multicolumn{2}{l}{
\includegraphics[width=0.08\textwidth]{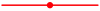}
 {{\Scale[0.85]{ \text{modeling bias from the expression \eqref{eq:numer_smb}} }}}
  }
\\
\multicolumn{2}{l}{\( \)}
\\
\multicolumn{2}{l}{\( \)}
   \end{tabular}
\end{center}
\end{figure}
\begin{figure}[!htbp]
\caption{\textbf{Local quadratic regression:} \newline \hspace*{4.9em}
 Confidence bands,
their widths, and the modeling bias}
\label{fig:localMLE_quadr}
\begin{center}
  \begin{tabular}{ c c }
\hspace{-1.9cm}   {\Scale[0.85]{\text{bandwidth} \(=\,0.12\)}}
&
\hspace{-1.4cm}   {\Scale[0.85]{\text{bandwidth} \(=\,0.3\)}}\\
\hspace{-1.9cm}  \includegraphics[width=0.522\textwidth]{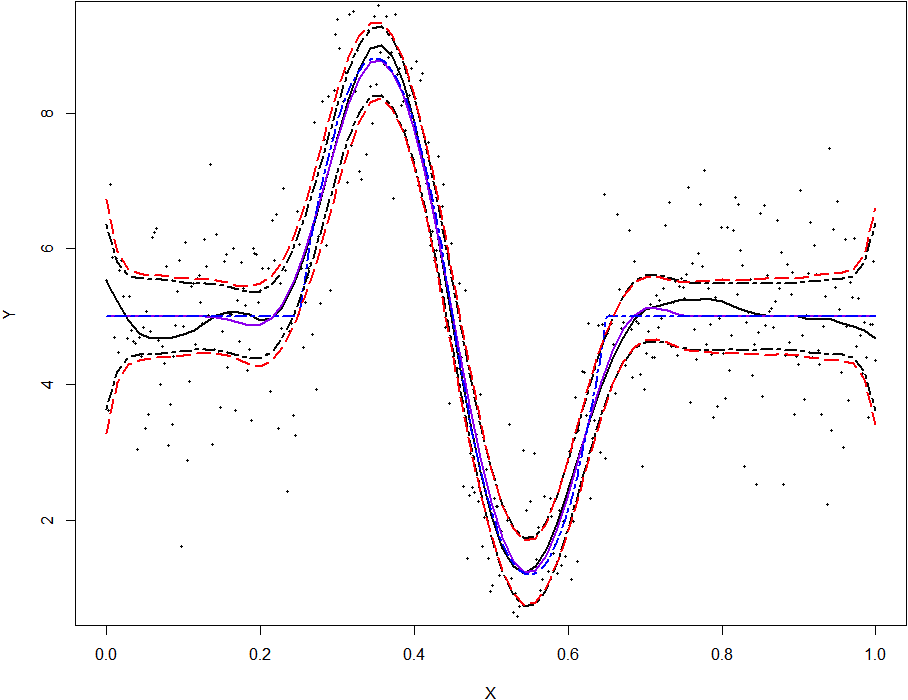}%
&
\hspace{-1.4cm}
\includegraphics[width=0.5\textwidth]{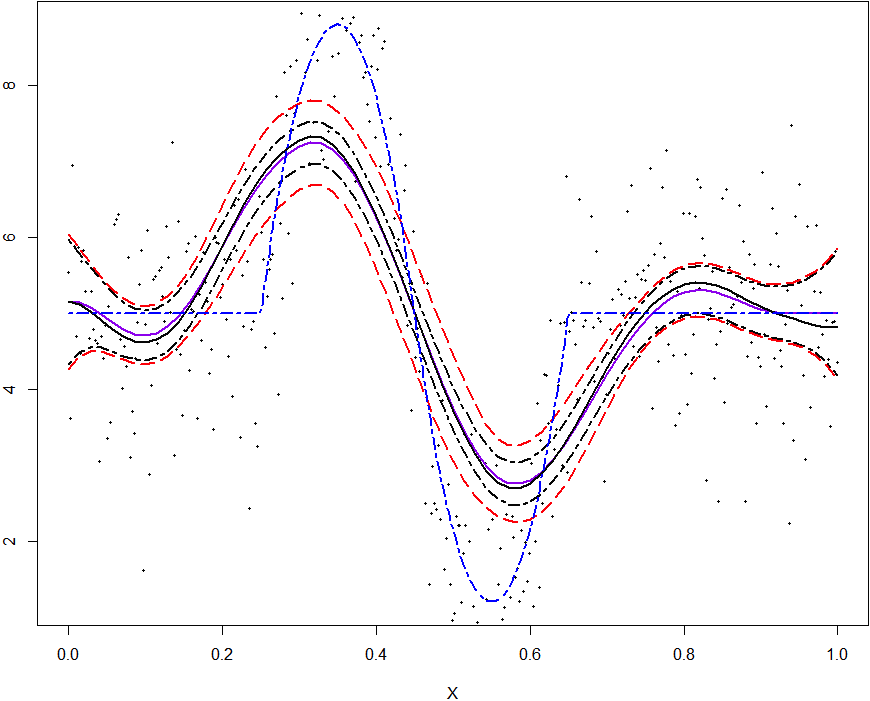}
  \\
\hspace{-1.9cm}
  \includegraphics[width=0.523\textwidth]{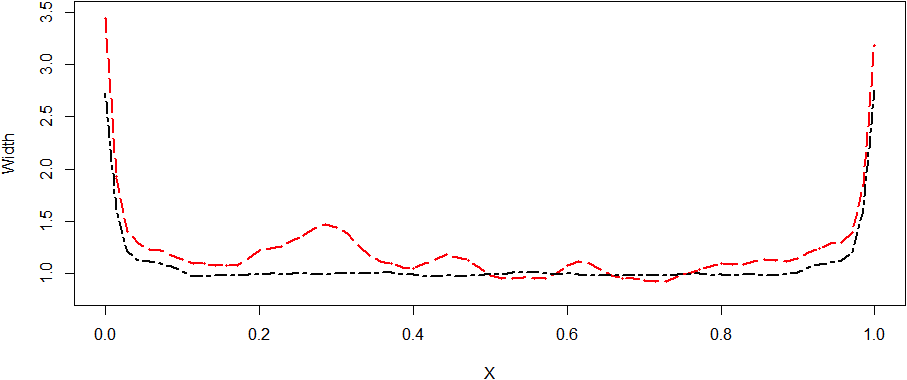}
&
\hspace{-1.4cm}
\includegraphics[width=0.5\textwidth]{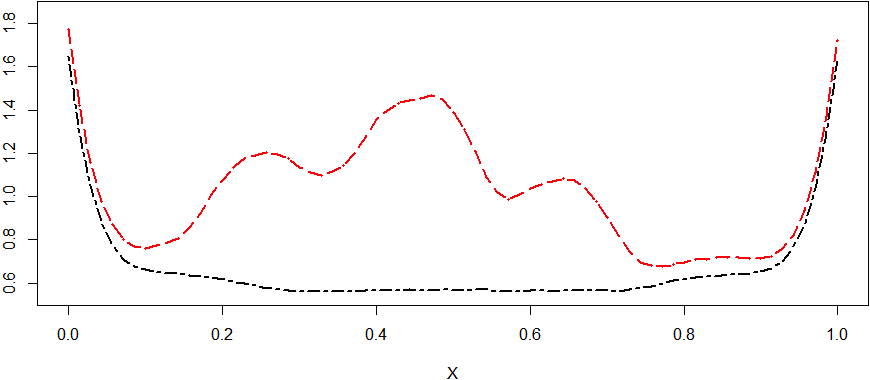}
\\
\hspace{-1.9cm}
  \includegraphics[width=0.523\textwidth]{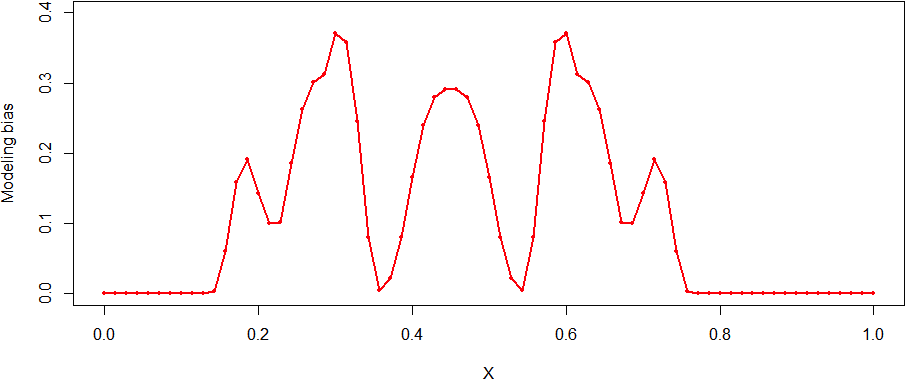}
  &
  \hspace{-1.4cm}
   \includegraphics[width=0.5\textwidth]{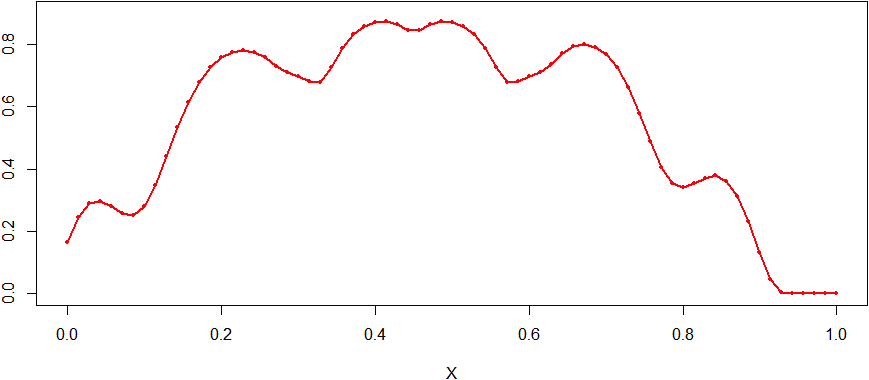}
  \\
    \multicolumn{2}{c}{
   {\Scale[0.85]{\text{Legend for the top graphs:}}}}\\
    \multicolumn{1}{l}{
\includegraphics[width=0.08\textwidth]{legend_boot}
 {{\Scale[0.85]{ \(90\%\) \text{ bootstrap simultaneous confidence band}}}}
  }
 &
 \multicolumn{1}{l}{
~~ \includegraphics[width=0.08\textwidth]{legend_true}
  \,{\(\Scale[0.85]{\text{the true function } f(x)}\)}
  }
 \\
 \multicolumn{1}{l}{
\includegraphics[width=0.08\textwidth]{legend_MC}
 {\Scale[0.85]{\(90\%\)\text{ MC simultaneous confidence band}}}}
   &
 \multicolumn{1}{l}{
~~ \includegraphics[width=0.08\textwidth]{legend_localMLE}
  ~{\Scale[0.85]{\text{local constant MLE}}}
  }
  \\
 \multicolumn{2}{l}{
  \includegraphics[width=0.08\textwidth]{legend_smoothed}
  ~{\Scale[0.85]{\text{smoothed target function}}}}\\
  \multicolumn{2}{c}{
   {\Scale[0.85]{\text{Legend for the middle and the bottom graphs:}}}}\\
 \multicolumn{2}{l}{
\includegraphics[width=0.08\textwidth]{legend_boot}
 {{\Scale[0.85]{\text{width of the } \(90\%\) \text{ bootstrap confidence bands from the upper graphs}}}}
  }\\
  \multicolumn{2}{l}{
\includegraphics[width=0.08\textwidth]{legend_MC}
 {{\Scale[0.85]{ \text{width of the }\(90\%\)\text{ MC confidence bands from the upper graphs}}}}
  }\\
   \multicolumn{2}{l}{
\includegraphics[width=0.08\textwidth]{legend_bias}
 {{\Scale[0.85]{ \text{modeling bias from the expression \eqref{eq:numer_smb_quadr} }}}}
  }
\\
\multicolumn{2}{l}{\( \)}
\\
\multicolumn{2}{l}{\( \)}
   \end{tabular}
\end{center}
\end{figure}

\subsection{Effective coverage probability (local constant estimate)}
\label{sect:numercoverage}
In this part of the experiment we check the bootstrap validity by computing the effective coverage probability values. This requires to perform many independent experiments: for each of independent \(5000\) \(\{Y_{i}\}\sim \eqref{def:numer_Y}\) samples we took \(10^{4}\) independent bootstrap samples \(\{u_{i}\}\sim \mathcal{N}(1,1)\), and constructed simultaneous bootstrap confidence sets for a range of confidence levels. The second row of  Table \ref{fig:localMLE} contains this range  \((1-\alpha)=0.95, 0.9,\dots, 0.5\). 
The third and the fourth rows of Table \ref{fig:localMLE} show the frequencies of the event
$$\max_{1\leq k \leq \numK}\left\{L_{k}(\thetavt_{k})-L_{k}(\thetavs_{k})-\zzb_{k}(\qqqb(\alpha))\right\}\leq 0$$
among \(5000\) data samples, for the bandwidths \(h=0.12, 0.3\), and for the range of \((1-\alpha)\). The results show that the bootstrap procedure is rather conservative for  both \(h=0.12\) and \(h=0.3\), however, the larger bandwidth yields bigger coverage probabilities.
\begin{table}[!h]
\caption{Effective coverage probabilities for the local constant regression}
\label{tab:cover_eff}
\begin{center}
\scalebox{0.9}{
  \begin{tabular}{| c | c|c|c|c|c |c|c|c|c|c| }  \cline{2-11}
  \multicolumn{1}{c|}{}&\multicolumn{10}{c|}{{{Confidence levels}}}\\\hline
\(h\)&$\scale[1]{\mathbf{0.95}}$ &\(\scale[1]{\mathbf{0.90}}\) &$\scale[1]{\mathbf{0.85}}$& \(\scale[1]{\mathbf{0.80}}\)& \(\scale[1]{\mathbf{0.75}}\)& \(\scale[1]{\mathbf{0.70}}\)&\(\scale[1]{\mathbf{0.65}}\)&\(\scale[1]{\mathbf{0.60}}\) & \(\scale[1]{\mathbf{0.55}}\)&\(\scale[1]{\mathbf{0.50}}\)\\
\hline\hline
\(\scale[1]{0.12}\)&
 \(0.971\) & \(0.947\)& \(0.917\) & \(0.888\) &\( 0.863\) &\(0.830\)&\(0.800\)&\( 0.769\)&\(0.738\)&\(0.702\)
  \\ \hline
\(\scale[1]{0.3}\)&\(\scale[1]{0.982}\)&\(\scale[1]{0.963}\)&\(\scale[1]{0.942}\)&\(\scale[1]{0.918}\)&\(\scale[1]{0.895}\)&\(\scale[1]{0.868}\)&\(\scale[1]{0.842}\)&
\(\scale[1]{0.815}\)&\(\scale[1]{0.784}\)& \(\scale[1]{0.750}\)\\
\hline
\end{tabular}
}
\end{center}
\end{table}

\subsection{Correction for multiplicity}
\label{sect:numercorrection}
Here we compare the \(\Ym\) and the bootstrap corrections for multiplicity, i.e. the values \(\qqq(\alpha)\) and
\(\qqqb(\alpha)\) defined in \eqref{bound:simult} and \eqref{def:mcorrqb}.
The numerical results in Tables \ref{tab:corr}, \ref{tab:corr_quadr} are based on  \(10^{4}\) \(\{Y_{i}\}\sim \eqref{def:numer_Y}\) independent samples and \(10^{4}\) independent bootstrap samples \(\{u_{i}\}\sim \mathcal{N}(1,1)\). The second line in Tables \ref{tab:corr}, \ref{tab:corr_quadr} contains the range of the nominal confidence levels  \((1-\alpha)=0.95, 0.9,\dots, 0.5\) (similarly to the Table \ref{tab:cover_eff}). The first column contains the values of the bandwidth \(h=0.12, 0.3\), and the second column -- the resampling scheme: Monte Carlo (MC) or bootstrap (B). The Monte Carlo experiment yields the corrected confidence levels \(1-\qqq(\alpha)\), and the bootstrap yields \(1-\qqqb(\alpha)\).
The lines 3--6 contain the average values of \(1-\qqq(\alpha)\) and \(1-\qqqb(\alpha)\) over all the experiments. The results show that for the smaller bandwidth both the MC and bootstrap corrections are bigger than the ones for the larger bandwidth.
In the case of a smaller bandwidth the local models have less intersections with each other, and hence, the corrections for multiplicity are closer to the Bonferroni's bound.
\begin{table}[!h]
\caption{\textbf{Local constant regression:}\newline \hspace*{3.6em} MC vs Bootstrap confidence levels corrected for multiplicity}
\label{tab:corr}
\begin{center}
\scalebox{0.87}{
  \begin{tabular}{| c | c | c|c|c|c|c |c|c|c|c|c| }  \cline{3-12}
  \multicolumn{2}{c|}{}&\multicolumn{10}{c|}{{{Confidence levels}}}\\\hline
\(h\)&\(\scale[1]{\text{r.m.}}\)&$\scale[1]{\mathbf{0.95}}$ &\(\scale[1]{\mathbf{0.90}}\) &$\scale[1]{\mathbf{0.85}}$& \(\scale[1]{\mathbf{0.80}}\)& \(\scale[1]{\mathbf{0.75}}\)& \(\scale[1]{\mathbf{0.70}}\)&\(\scale[1]{\mathbf{0.65}}\)&\(\scale[1]{\mathbf{0.60}}\) & \(\scale[1]{\mathbf{0.55}}\)&\(\scale[1]{\mathbf{0.50}}\)\\
\hline\hline
\multirow{2}[2]*{\(0.12\)}&MC &
 \(0.997\) & \(0.994\)& \(0.989\) & \(0.985\) &\(0.980\) &\(0.975\)&\(0.969\)&\(0.963\)&\(0.956\)&\(0.949\)
   \\ \cline{2-12}
&B &
 \(0.998\) & \(0.995\)& \(0.991\) & \(0.988\) &\(0.984\) &\(0.979\)&\(0.975\)&\(0.969\)&\(0.963\)&\(0.957\)
  \\ \hline\hline
\multirow{2}[2]*{\(0.3\)}&MC&\(\scale[1]{0.993}\)&\(\scale[1]{0.983}\)&\(\scale[1]{0.973}\)&\(\scale[1]{0.962}\)&\(\scale[1]{0.949}\)&\(\scale[1]{0.936 }\)&\(\scale[1]{0.922}\)&
\(\scale[1]{0.906}\)&\(\scale[1]{0.891}\)& \(\scale[1]{0.873}\)
\\ \cline{2-12}
&B&\(\scale[1]{0.994}\)&\(\scale[1]{0.986}\)&\(\scale[1]{0.977}\)&\(\scale[1]{0.968}\)&\(\scale[1]{0.958}\)&\(\scale[1]{ 0.947}\)&\(\scale[1]{0.935}\)&
\(\scale[1]{0.922}\)&\(\scale[1]{0.908}\)& \(\scale[1]{0.893}\)\\
\hline
\end{tabular}
}
\end{center}
\end{table}

\begin{table}[!h]
\caption{\textbf{Local quadratic regression:}\newline \hspace*{3.6em} MC vs Bootstrap confidence levels corrected for multiplicity}
\label{tab:corr_quadr}
\begin{center}
\scalebox{0.87}{
  \begin{tabular}{| c | c | c|c|c|c|c |c|c|c|c|c| }  \cline{3-12}
  \multicolumn{2}{c|}{}&\multicolumn{10}{c|}{{{Confidence levels}}}\\\hline
\(h\)&\(\scale[1]{\text{r.m.}}\)&$\scale[1]{\mathbf{0.95}}$ &\(\scale[1]{\mathbf{0.90}}\) &$\scale[1]{\mathbf{0.85}}$& \(\scale[1]{\mathbf{0.80}}\)& \(\scale[1]{\mathbf{0.75}}\)& \(\scale[1]{\mathbf{0.70}}\)&\(\scale[1]{\mathbf{0.65}}\)&\(\scale[1]{\mathbf{0.60}}\) & \(\scale[1]{\mathbf{0.55}}\)&\(\scale[1]{\mathbf{0.50}}\)\\
\hline\hline
\multirow{2}[2]*{\(0.12\)}&MC &
 \(0.997\) & \(0.993\)& \(0.989\) & \(0.985\) &\(0.979\) &\(0.974\)&\(0.968\)&\(0.961\)&\(0.954\)&\(0.946\)
   \\ \cline{2-12}
&B &
 \(0.998\) & \(0.995\)& \(0.991\) & \(0.988\) &\(0.984\) &\(0.979\)&\(0.974\)&\(0.969\)&\(0.963\)&\(0.956\)
  \\ \hline\hline
\multirow{2}[2]*{\(0.3\)}&MC&\(\scale[1]{0.993}\)&\(\scale[1]{0.983}\)&\(\scale[1]{0.973}\)&\(\scale[1]{0.961}\)&\(\scale[1]{0.949}\)&\(\scale[1]{0.936}\)&\(\scale[1]{0.921}\)&
\(\scale[1]{0.904}\)&\(\scale[1]{0.887}\)& \(\scale[1]{0.868}\)
\\ \cline{2-12}
&B&\(\scale[1]{0.996}\)&\(\scale[1]{0.991}\)&\(\scale[1]{0.985}\)&\(\scale[1]{0.978}\)&\(\scale[1]{0.971}\)&\(\scale[1]{0.963}\)&\(\scale[1]{0.954}\)&
\(\scale[1]{0.944}\)&\(\scale[1]{0.934}\)& \(\scale[1]{0.923}\)\\
\hline
\end{tabular}
}
\end{center}
\end{table}
\begin{newremark}
The theoretical results of this paper can be extended to the case when a set of considered local models has cardinality of the continuum, and the confidence bands are uniform w.r.t. the local parameter. This extension would require  some uniform statements such as locally uniform square-root Wilks approximation (see e.g. \cite{SpZh2013uniform}). 
\end{newremark} 
\begin{newremark}
The use of the bootstrap procedure in the problem of choosing an optimal bandwidth is considered in \cite{SW2015}.
\end{newremark}

\def\ED{E\!D}
\def\SD{S\!D}
\def\SmB{\operatorname{SmB}}
\section{Conditions}
\label{sect:conditions}
Here we show necessary conditions for the main results. The conditions in Section  \ref{sect:ConditGeneral} come from the general finite sample theory by \cite{Sp2012Pa}, they are required for the results of Sections \ref{sect:FS} and \ref{sect:FStheory}. 
 The conditions in Section \ref{sect:ConditAddBoot} are necessary to prove the statements on multiplier bootstrap validity.

\subsection{Basic conditions}
\label{sect:ConditGeneral}
Introduce the stochastic part of the \(k\)-th likelihood process:
\(\zeta_{k}(\thetav)\eqdef L_{k}(\thetav)- \E L_{k}(\thetav)\), and its marginal summand: \(\zeta_{i,k}(\thetav)\eqdef\ell_{i,k}(\thetav)-\E\ell_{i,k}(\thetav)\) for \(\ell_{i,k}(\thetav)\) defined in \eqref{def:ell_ik}.
\begin{description}
\labitem{\( \bb{(\ED_{0})}\)}{itm:ED0k}
\emph{For each \(k=1,\dots,\numK\) there exist a positive-definite \(\dimp_{k}\times\dimp_{k}\) symmetric matrix \(V_{k}^{2}\) and constants \(\gm_{k}>0, \nu_{k}\geq1\) such that \(\Var\left\{\nabla_{\thetav}\zeta_{k}(\thetavs_{k})\right\}\leq V_{k}^{2} \) and}
\begin{EQA}[c]
\sup_{\gammav\in \R^{\dimp_{k}}}\ \log \E \exp \left\{
\lambda\frac{\gammav^{\T} \nabla_{\thetav}\zeta_{k}(\thetavs_{k})}{\|V_{k}\gammav\|}
\right\}\leq \nu_{k}^{2}\lambda^{2}/2, \quad\quad |\lambda|\leq\gm_{k}.
\end{EQA}
\end{description}
\begin{description}
\labitem{\( \bb{ (\ED_{2})}\)}{itm:ED2k}
\emph{For each \(k=1,\dots,\numK\) there exist a constant \(\omega_{k}> 0\) and for each \(\rr>0\) a constant \(\gm_{2,k}(\rr)\) such that it holds for all \(\thetav\in \Theta_{0,k}(\rr)\) and for \(j=1,2\)}
\begin{EQA}[c]
\sup_{\substack{\gammav_{j}\in \R^{\dimp_{k}}\\
\|\gammav_{j}\|\leq1}
}\ \log \E \exp \left\{
\frac{\lambda}{\omega_{k}}\gammav_{1}^{\T}D_{k}^{-1}\nabla_{\thetav}^{2}\zeta_{k}(\thetav)D_{k}^{-1}\gammav_{2}
\right\}
\leq \nu_{k}^{2}\lambda^{2}/2, \quad\quad |\lambda|\leq\gm_{2,k}(\rr).
\end{EQA}
\end{description}
\begin{description}
\labitem{\( \bb{(\LL_{0})}\)}{itm:L0k}
\emph{For each \(k=1,\dots,\numK\) and for each \(\rr>0\) there exists a constant \(\rddelta_{k}(\rr)\geq0\) such that for \(\rr\leq\rr_{0,k}\) (\(\rr_{0,k}\) come from condition \eqref{condit:concentr} of Theorem \ref{thm:concentr} in Section \ref{sect:FS}) \(\rddelta(\rr)\leq1/2\), and for all \(\thetav \in \Theta_{0,k}(\rr)\) it holds}
\begin{EQA}[c]
\|D_{k}^{-1}\check{D}_{k}^{2}(\thetav)D_{k}^{-1}-\Id_{\dimp_{k}}\|\leq \rddelta_{k}(\rr),
\end{EQA}
\emph {where \(\check{D}_{k}^{2}(\thetav)\eqdef -\nabla_{\thetav}^{2}\E L_{k}(\thetav)\) and \(\Theta_{0,k}(\rr)\eqdef\left\{\thetav\in\Theta_{k}: \|D_{k}(\thetav-\thetavs_{k})\|\leq \rr\right\}\).}
\end{description}
\begin{description}
\labitem{\(\bb{(\mathcal{I})}\)}{itm:Ik}
\emph{There exist constants \(\gmu_{k}>0\) for all \(k=1,\dots,\numK\) s.t.}
\begin{EQA}[c]
\gmu_{k}^{2} D_{k}^{2}\geq  V_{k}^{2}.
\end{EQA}
\emph{Denote \(\bbgmu^{2}\eqdef \max_{1\leq k\leq \numK}\gmu_{k}^{2}\).}
\end{description}

\begin{description}
\labitem{\(\bb{(\LL\rr)}\)}{itm:Lrk}
\emph{For each \(k=1,\dots,\numK\) and \(\rr\geq\rr_{0,k}\) there exists a value \(\gmi_{k}(\rr)>0\) s.t. \\\(\rr\gmi_{k}(\rr)\rightarrow \infty\) for \(\rr \rightarrow \infty\) and}
\(\forall \thetav\in\Theta_{k}: \|D_{k}(\thetav-\thetavs_{k})\|=\rr\) \emph{it holds}
\begin{EQA}[c]
-2\left\{\E L_{k}(\thetav)-\E L_{k}(\thetavs_{k})\right\}\geq\rr^{2}\gmi_{k}(\rr).
\end{EQA}
\end{description}
\subsection{Conditions required for the bootstrap validity}
\label{sect:ConditAddBoot}
\begin{description}
\labitem{\( \bb{(\barbar{\SmB})}\)}{itm:SmBHk}
\emph{ There exists a constant \(\DeltaSmB \geq 0\) 
such that it holds for the matrices 
 \(B_{k}^{2}\) and \(H^{2}_{k}\) defined in \eqref{def:Bsmbk}:}
\begin{EQA}[c]
\max_{1\leq k\leq\numK}
\left\|
H^{-1}_{k}B_{k}^{2}
H^{-1}_{k}
\right\| \leq \DeltaSmB^{2},\\
\DeltaSmB^{2}\leq
 \CONST \left(\frac{n}{\dimp_{\max}^{13}}\right)^{1/8}
\log^{-7/8}(K)\log^{-3/8}(n\dimtotal).
\end{EQA}

\end{description}
\begin{description}
\labitem{\( \bb{ (\ED_{2m})}\)}{itm:ED2mk}
\emph{For each \(k=1,\dots,\numK\), \(\rr>0\), \(i=1,\dots,n\), \(j=1,2\) and for all \(\thetav \in \Theta_{0,k}(\rr)\) it holds for the values \(\omega_{k}\geq 0\) and \(\gm_{2,k}(\rr)\) from the condition \ref{itm:ED2k}: }
\begin{EQA}[c]
\sup_{\substack{\gammav_{j}\in \R^{\dimp_{k}}\\
\|\gammav_{j}\|\leq1}
}\ \log \E \exp \left\{
\frac{\lambda}{\omega_{k}}\gammav_{1}^{\T}D_{k}^{-1}
\nabla_{\thetav}^{2}\zeta_{i,k}(\thetav)
D_{k}^{-1}\gammav_{2}
\right\}
\leq \frac{\nunu^{2}\lambda^{2}}{2n}, \quad\quad |\lambda|\leq\gm_{2,k}(\rr),
\end{EQA}
\end{description}
\begin{description}
\labitem{\( \bb{(\LL_{0m})}\)}{itm:L0mk}
\emph{For each \(k=1,\dots,\numK\), \(\rr>0\), \(i=1,\dots,n\) and for all \(\thetav \in \Theta_{0,k}(\rr)\) there exists a value \(\CONST_{m,k}(\rr)\geq 0\) such that}
\begin{EQA}[c]
\|D_{k}^{-1}\nabla_{\thetav}^{2}\E\ell_{i,k}(\thetav)D_{k}^{-1}\|\leq {\CONST_{m,k}(\rr)}{n}^{-1}.
\end{EQA}
\end{description}
%
%
\begin{description}
\labitem{\( \bb{(\mathcal{I}_{B})}\)}{itm:IBk}
\emph{For each \(k=1,\dots,\numK\) there exists a constant \(\gmu_{B,k}^{2}> 0\) s.t. }
\begin{EQA}[c]
\gmu_{B,k}^{2} D_{k}^{2}\geq B_{k}^{2}.
\end{EQA}
\emph{Denote \(\bbgmu_{B}^{2}\eqdef \max_{1\leq k\leq \numK}\gmu_{B,k}^{2}\).}
\end{description}
%
%
\begin{description}
\labitem{\( \bb{(\barbar{\SD_{1}})}\)}{itm:SD0k}
\emph{There exists a constant \(0\leq\delta_{v^{\ast}}^{2}\leq \CONST\dimtotal/n\) such that it holds for all \(i=1,\dots,n\) with exponentially high probability}
\begin{EQA}[c]
\left\|\bbH^{-1}\left\{\bbg_{i}\bbg_{i}^{\T}- \E\left[\bbg_{i}\bbg_{i}^{\T}\right]\right\}\bbH^{-1}\right\|
\leq \delta_{v^{\ast}}^{2},
\end{EQA}
\emph{where}
\begin{EQA}
\bbg_{i}&\eqdef& \left(\nabla_{\thetav}\ell_{i,1}(\thetavs_{1})^{\T},\dots,\nabla_{\thetav}\ell_{i,\numK}(\thetavs_{\numK})^{\T}\right)^{\T}\in \R^{\dimtotal},\\
\bbH^{2}&\eqdef&\sum\nolimits_{i=1}^{n}\E \left\{\bbg_{i}\bbg_{i}^{\T}\right\},\\
\dimtotal&\eqdef&\dimp_{1}+\dots+\dimp_{\numK}.
\end{EQA}
\end{description}
\begin{description}
\labitem{\(\bb{(Eb)}\)}{itm:Eb}
\emph{The i.i.d. bootstrap weights \(u_{i}\) are independent of \(\Ym\), and for all \(i=1,\dots,n\) it holds for some constants \(\gm_{k}>0, \nu_{k}\geq1\)}
\begin{EQA}[c]
\E u_{i} =1 ,\quad \Var u_{i}=1,\\
\log \E \exp \left\{
\lambda (u_{i}-1)
\right\}\leq \nunu^{2}\lambda^{2}/2, \quad\quad |\lambda|\leq\gm.
\end{EQA}
\end{description}

\subsection{Dependence of the involved terms on the sample size and cardinality of the parameters' set }
\label{typical_local}
Here we consider the case of the i.i.d. observations \(Y_{1},\dots, Y_{n}\) and \(\yy= \CONST\log{n}\) in order to specify the dependence of the non-asymptotic bounds on \(n\) and \(\dimp\).
In the paper by \cite{SpZh2014PMB} (the version of 2015) this is done in detail for the i.i.d. case, generalized linear model and quantile regression.

Example 5.1 in \cite{Sp2012Pa} demonstrates that in this situation \(\gm_{k}=\CONST\sqrt{n}\) and \(\omega_{k}=\CONST /\sqrt{n}\).
then \(\ZZ_{k}(\yy)= \CONST\sqrt{\dimp_{k}+\yy}\) for some constant \(\CONST\geq 1.85\), for  the function \(\ZZ_{k}(\yy)\) given in \eqref{ZZ_k} in Section \ref{sect:FS}.
 Similarly it can be checked that \(\gm_{2,k}(\rr)\) from condition \ref{itm:ED2k}
is proportional to \(\sqrt{n}\): due to independence of the observations
\begin{EQA}
&&\nquad
\log \E \exp \left\{
\frac{\lambda}{\omega_{k}}\gammav_{1}^{\T}D_{k}^{-1}\nabla_{\thetav}^{2}\zeta_{k}(\thetav)D_{k}^{-1}\gammav_{2}
\right\}
\\&=&\sum\nolimits_{i=1}^{n}\log \E \exp \left\{
\frac{\lambda}{\sqrt{n}}\frac{1}{\omega_{k}\sqrt{n}}
\gammav_{1}^{\T}d_{k}^{-1}\nabla_{\thetav}^{2}\zeta_{i,k}(\thetav)d_{k}^{-1}\gammav_{2}
\right\}\\
&\leq& n\frac{\lambda^{2}}{n}\CONST\quad \text{for}~|\lambda|\leq \bar{\gm}_{2,k}(\rr)\sqrt{n},
\end{EQA}
where \(\zeta_{i,k}(\thetav)\eqdef\ell_{i,k}(\thetav)-\E\ell_{i,k}(\thetav)\), \(d_{k}^{2}\eqdef -\nabla^{2}_{\thetav}\E\ell_{i,k}(\thetavs_{k})\) and \(D_{k}^{2}=n d_{k}^{2}\) in the i.i.d. case. Function \(\bar{\gm}_{2,k}(\rr)\) denotes the marginal analog of \(\gm_{2,k}(\rr)\).

Let us show, that for the value \(\delta_{k}(\rr)\) from the condition \ref{itm:L0k} it holds \(\delta_{k}(\rr)=\CONST \rr/\sqrt{n}\). Suppose for all \(\thetav \in \Theta_{0,k}(\rr)\) and \(\gammav\in\R^{\dimp_{k}}:\|\gammav\|=1\) \(\|D_{k}^{-1}\gammav^{\T}\nabla_{\thetav}^{3}\E L_{k}(\thetav)D_{k}^{-1}\|\leq\CONST\), then it holds for some \(\bar{\thetav}\in \Theta_{0,k}(\rr)\):
\begin{EQA}
&&\nquad
\|D_{k}^{-1}D^{2}(\thetav)D_{k}^{-1}-\Id_{\dimp_{k}}\|=
\|D_{k}^{-1}(\thetavs_{k}-\thetav)^{\T}\nabla_{\thetav}^{3}\E L_{k}(\bar{\thetav})D_{k}^{-1}\|\\
&=&\|D_{k}^{-1}(\thetavs_{k}-\thetav)^{\T}D_{k}D_{k}^{-1}\nabla_{\thetav}^{3}\E L_{k}(\bar{\thetav})D_{k}^{-1}\|
\\&\leq& \rr\|D_{k}^{-1}\|\|D_{k}^{-1}\gammav^{\T}\nabla_{\thetav}^{3}\E L_{k}(\bar{\thetav})D_{k}^{-1}\|\leq\CONST\rr/\sqrt{n}.
\end{EQA}
Similarly \({\CONST_{m,k}(\rr)}\leq \CONST\rr/\sqrt{n}+\CONST\) in condition \ref{itm:L0mk}.

The next remark helps to check the global identifiability condition \ref{itm:Lrk} in many situations.
Suppose that the parameter domain \(\Theta_{k}\) is compact and \(n\) is sufficiently large, then the value \(\gmi_{k}(\rr)\) from condition \ref{itm:Lrk} can be taken as \(\CONST\{1-\rr/\sqrt{n}\}\approx \CONST \). Indeed, for \(\thetav: \|D_{k}(\thetav-\thetavs_{k})\|=\rr\)
\begin{EQA}
-2\left\{\E L_{k}(\thetav)-\E L_{k}(\thetavs_{k})\right\}
&\geq& \rr^{2}\left\{1-\rr\|D_{k}^{-1}\|\|D_{k}^{-1}\gammav^{\T}\nabla_{\thetav}^{3}\E L_{k}(\bar{\thetav})D_{k}^{-1}\|\right\}
\\&\geq&
\rr^{2}(1- \CONST \rr/\sqrt{n}).
\end{EQA}
 Due to the obtained orders, the conditions \eqref{condit:concentr} and \eqref{condit:large_dev} of Theorems \ref{thm:concentr} and \ref{thm:large_dev} on concentration of the MLEs \(\thetavt_{k},\, \thetavbt_{k}\) require \(\rr_{0,k} \geq \CONST\sqrt{\dimp_{k} +\yy}\).

\appendix

\section{Approximation of the joint distributions of \(\ell_{2}\)-norms}
\label{sect:maxgar}

Let us previously introduce some notations:\\
\(\Idv_{\numK}\eqdef (1,\dots,1)^{\T}\in \R^{\numK}\);\\
\(\|\cdot\|\) is the Euclidean norm for a vector and spectral norm for a matrix;\\
\(\|\cdot\|_{\max}\) is the maximum of absolute values of elements of a vector or of a matrix; \\
\(\|\cdot\|_{1}\) is the sum of absolute values of elements of a vector or of a matrix.

Consider \(\numK\) random centered vectors \(\phiv_{k}\in \R^{\dimp_{k}}\) for \(k=1,\dots,\numK\). Each vector equals to a sum of \(n\) centered independent vectors:
\begin{align}
\begin{split}
\label{def:phivk}
\phiv_{k}=\phiv_{k,1}&+\dots+\phiv_{k,n},\\
\E\phiv_{k}=\E\phiv_{k,i}&=0\quad \forall\,1\leq i\leq n.
\end{split}
\end{align}
Introduce similarly the vectors \(\psiv_{k}\in \R^{\dimp_{k}}\) for \(k=1,\dots,\numK\):
\begin{align}
\begin{split}
\label{def:psivk}
\psiv_{k}=\psiv_{k,1}&+\dots+\psiv_{k,n},\\
\E\psiv_{k}=\E\psiv_{k,i}&=0\quad \forall\,1\leq i\leq n,
\end{split}
\end{align}
with the same independence properties as \(\phiv_{k,i}\), and also independent of all \(\phiv_{k,i}\).

The goal of this section is to compare the joint distributions of the \(\ell_{2}\)-norms of the sets of vectors \(\phiv_{k}\) and \(\psiv_{k}\), \(k=1,\dots, \numK\) (i.e. the probability laws \(\LL\left(\|\phiv_{1}\|,\dots, \|\phiv_{\numK}\|\right)\) and \(\LL\left(\|\psiv_{1}\|,\dots, \|\psiv_{\numK}\|\right)\)), assuming that their correlation structures are close to each other.

Denote
\begin{EQA}
\label{def:dimtotal}
\dimp_{\max}&\eqdef& \max_{1\leq k\leq\numK}\dimp_{k},\quad
\dimtotal\eqdef\dimp_{1}+\dots+\dimp_{\numK},
\\
\lambda_{\phi,\max}^{2}&\eqdef& \max_{1\leq k\leq\numK}\|\Var(\phiv_{j})\|,
\quad
\lambda_{\psi,\max}^{2}\eqdef \max_{1\leq k\leq\numK}\|\Var(\psiv_{j})\|,
\\
z_{\max}&\eqdef& \max_{1\leq k\leq\numK}z_{k},\quad
z_{\min}\eqdef \min_{1\leq k\leq\numK}z_{k},
\label{def:z_max}
\\
\delta_{z,\max}&\eqdef&\max_{1\leq k\leq\numK}\delta_{z_{k}},\quad
\delta_{z,\min}\eqdef\min_{1\leq k\leq\numK}\delta_{z_{k}},
\label{def:delta_z_max}
\end{EQA}
let also
\begin{EQA}
\label{def:Deltaeps}
\Delta_{\varepsilon}&\eqdef& 
\left(\frac{\dimp_{\max}^{3}}{n}\right)^{1/8}\log^{9/16}(K)\log^{3/8}(n\dimtotal)z_{\min}^{1/8}\\
&& \times\max\left\{\lambda_{\phi,\max},\lambda_{\psi,\max}\right\}^{3/4}\log^{-1/8}(5n^{1/2}).
\end{EQA}
The following conditions are necessary for the Proposition \ref{thm:Cum_norms}
\begin{itemize}
\labitem{(C1)}{itm:A1}
\emph{For some \({\gm}_{k},\nu_{k},\cf_{\phi},\cf_{\psi}>0\)and for all \(i=1,\dots,n\), \(k=1,\dots,\numK\)}
\begin{EQA}
\sup_{\substack{\gammav_{k}\in\R^{\dimp_{k}},\\ \|\gammav_{k}\|=1}}\log\E\exp\left\{\lambda \sqrt{n}\gammav_{k}^{\T}\phiv_{k,i}/\cf_{\phi}\right\} &\leq&  \lambda^{2}\nu_{k}^{2}/2,\quad  |\lambda|<{\gm}_{k},\\
\sup_{\substack{\gammav_{k}\in\R^{\dimp_{k}},\\ \|\gammav_{k}\|=1}}\log\E\exp\left\{\lambda \sqrt{n}\gammav_{k}^{\T}\psiv_{k,i}/\cf_{\psi}\right\} &\leq&  \lambda^{2}\nu_{k}^{2}/2,\quad  |\lambda|<{\gm}_{k},
\end{EQA}
\emph{where \(\cf_{\phi}\geq \CONST\lambda_{\phi,\max}\) and  \(\cf_{\psi}\geq \CONST \lambda_{\phi,\max}\).}
\labitem{(C2)}{itm:A2}
\emph{For some \(\delta_{\Sigma}^{2}\geq 0\)}
\begin{EQA}[c]
\label{cond:deltaSigma}
 \max_{1\leq k_{1},\,k_{2}\leq\numK}\
  \left\|
  \Cov(\phiv_{k_{1}},\phiv_{k_{2}}) -\Cov(\psiv_{k_{1}},\psiv_{k_{2}})
  \right\|_{\max}
 \leq
\delta_{\Sigma}^{2}.
\end{EQA}
\end{itemize}
\begin{proposition}[Approximation of the joint distributions of \(\ell_{2}\)-norms]
\label{thm:Cum_norms}
Consider the centered random vectors \(\phiv_{1},\dots,\phiv_{\numK}\) and \(\psiv_{1},\dots,\psiv_{\numK}\) given in \eqref{def:phivk}, \eqref{def:psivk}. Let the conditions \ref{itm:A1} and \ref{itm:A2} be fulfilled,  and the values  \(z_{k}\geq\sqrt{\dimp_{k}+\Delta_{\varepsilon}}\) and \(\delta_{z_{k}}\geq 0\) be s.t. \(\CONST\max\{n^{-1/2}, \delta_{z,\max}\}
\leq
\Delta_{\varepsilon}\leq
\CONST z_{\max}^{-1}\), then it holds with dominating probability
\begin{EQA}
\P\left(\bigcup\nolimits_{k=1}^{\numK}\left\{\|\phiv_{k}\|>z_{k}\right\} \right)
-
\P\left(\bigcup\nolimits_{k=1}^{\numK}\left\{\|\psiv_{k}\|>z_{k}-\delta_{z_{k}}\right\} \right)&\geq& -\Deltapp\, ,\\
\P\left(\bigcup\nolimits_{k=1}^{\numK}\left\{\|\phiv_{k}\|>z_{k}\right\} \right)
-
\P\left(\bigcup\nolimits_{k=1}^{\numK}\left\{\|\psiv_{k}\|>z_{k}+\delta_{z_{k}}\right\} \right)&\leq& \Deltapp
\end{EQA}
for the deterministic non-negative value
\begin{EQA}
\Deltapp
&\leq&
12.5\CONST\left(\frac{\dimp_{\max}^{3}}{n}\right)^{1/8}\log^{9/8}(K)\log^{3/8}(n\dimtotal) \max\left\{\lambda_{\phi,\max},\lambda_{\psi,\max}\right\}^{3/4}
\\&&+\,3.2\CONST\delta_{\Sigma}^{2}\left(\frac{\dimp_{\max}^{3}}{n}\right)^{1/4}\dimp_{\max}z_{\min}^{1/2}
\log^{2}(K)\log^{3/4}(n\dimtotal)\max\left\{\lambda_{\phi,\max},\lambda_{\psi,\max}\right\}^{7/2}
\\&\leq&
25\CONST\left(\frac{\dimp_{\max}^{3}}{n}\right)^{1/8}\log^{9/8}(K)\log^{3/8}(n\dimtotal) \max\left\{\lambda_{\phi,\max},\lambda_{\psi,\max}\right\}^{3/4},
\label{ineq:Delta_l2}
\end{EQA}
where the last inequality holds for
\begin{EQA}
\delta_{\Sigma}^{2}
&\leq&4 \CONST \left(\frac{n}{\dimp_{\max}^{13}}\right)^{1/8}
\log^{-7/8}(K)\log^{-3/8}(n\dimtotal)\left(\max\left\{\lambda_{\phi,\max},\lambda_{\psi,\max}\right\}\right)^{-11/4}.
\end{EQA}
\end{proposition}
\begin{newremark}
\label{rem:Deltapp}
The approximating error term \(\Deltapp\) consists of three errors, which correspond to:   the Gaussian approximation result (Lemma \ref{lemma:gauss_appr_phiv}), Gaussian comparison
 (Lemma \ref{lemma:GaussCOmpar_Slepian}), and anti-concentration inequality (Lemma \ref{lemma:anti_conc_nonsmooth}).
The bound on \(\Deltapp\) above implies that the number \(\numK\) of the random vectors \(\phiv_{1},\dots,\phiv_{\numK}\) should satisfy \(\log\numK \ll (n/\dimp_{\max}^{3})^{1/12}\) in order to keep the approximating error term \(\Deltapp\) small.
 This condition can be relaxed by using a sharper Gaussian approximation result.
 For instance, using in Lemma \ref{lemma:gauss_appr_phiv} the Slepian-Stein technique plus induction argument from the recent paper by \cite{Chernozhukov.et.al.(2014a)} instead of the Lindeberg's approach, would lead to the improved bound:
\(\CONST\left(\frac{\dimp_{\max}^{3}}{n}\right)^{1/6}\) multiplied by a logarithmic term.
\end{newremark}

\subsection{Joint Gaussian approximation of \(\ell_{2}\)-norm of sums of independent vectors by Lindeberg's method}
\label{sect:gar}
Introduce the following random vectors from \(\R^{\dimtotal}\):
\begin{align}
\begin{split}
\label{def:Phi}
\Phi&\eqdef\left(\phiv_{1}^{\T},\dots,\phiv_{\numK}^{\T}\right)^{\T},\quad
\Phi_{i}\eqdef\left(\phiv_{1,i}^{\T},\dots,\phiv_{\numK,i}^{\T}\right)^{\T},
~i=1,\dots,n,\\
\Phi&=\sum\nolimits_{i=1}^{n}\Phi_{i},\quad \quad\quad\quad\E\Phi=\E\Phi_{i}=0.
\end{split}
\end{align}
Define their Gaussian analogs as follows:
\begin{align}
\label{def:Phibar1}
\Phibar_{i}&\eqdef\left(\phivbar_{1,i}^{\T},\dots,\phivbar_{\numK,i}^{\T}\right)^{\T},
&
\Phibar&\eqdef\left(\phivbar_{1}^{\T},\dots,\phivbar_{\numK}^{\T}\right)^{\T}=
\sum\nolimits_{i=1}^{n}\Phibar_{i},\\
\label{def:Phibar2}
\Phibar_{i}&\sim \mathcal{N}(0,\Var\Phi_{i}),
&
\Phibar&\sim\mathcal{N}(0,\Var\Phi),\\
\label{def:Phibar3}
\phivbar_{k,i}&\sim\mathcal{N}(0,\Var\phiv_{k,i}),
&
\phivbar_{k}&\eqdef \sum\nolimits_{i=1}^{n}\phivbar_{k,i}\sim\mathcal{N}(0,\Var\phiv_{k}).
\end{align}
\begin{lemma}[Joint GAR with equal covariance matrices]
\label{lemma:gauss_appr_phiv}
Consider the sets of random vectors \(\phiv_{j}\) and \(\phivbr_{j}\), \(j=1,\dots,\numK\) defined in \eqref{def:phivk}, and  \eqref{def:Phi}-- \eqref{def:Phibar3}. If the conditions of Lemmas \ref{lemma:max_phi3} are \ref{lemma:max_cube} are fulfilled, then it holds for all \(\Delta,\beta>0\), \(z_{j}\geq\max\left\{\Delta+\sqrt{\dimp_{j}},2.25\log(\numK)/\beta\right\}\) with dominating probability 
\begin{EQA}
\P\left(
\bigcup\nolimits_{j=1}^{\numK}\left\{\|\phiv_{j}\|> z_{j}\right\}
\right)
&\leq&
\P\left(
\bigcup\nolimits_{j=1}^{\numK}\left\{\|\phivbar_{j}\|> z_{j}-\Delta-\frac{3\log(\numK)}{2\beta}\right\}
\right)+\delta_{3,\phi}(\Delta,\beta)
,\\
\P\left(\bigcup\nolimits_{j=1}^{\numK}\left\{\|\phiv_{j}\|>z_{j}\right\} \right)
&\geq&
\P\left(\bigcup\nolimits_{j=1}^{\numK}\left\{\|\phivbar_{j}\|>z_{j}+\Delta+ \frac{3\log(\numK)}{2\beta} \right\}\right)
-\delta_{3,\phi}(\Delta,\beta)
\end{EQA}
for \(\delta_{3,\phi}(\Delta,\beta)\leq  \CONST \left(\frac{1}{\Delta^{3}}
+
\frac{\beta}{\Delta^{2}}
+
\frac{\beta^{2}}{\Delta}
\right)\left\{\frac{\dimp_{\max}^{3}}{n}
\log(\numK)\log^{3}(n\dimtotal)\right\}^{1/2}\) given in \eqref{ineq:Efb_appr}.
\end{lemma}
\begin{proof}[Proof of Lemma \ref{lemma:gauss_appr_phiv}]
\begin{EQA}
\label{eq:maxind}
\P\left(
\bigcup\nolimits_{j=1}^{\numK}\left\{\|\phiv_{j}\|> z_{j}\right\}
\right)&=&
\E\Ind\bigl(
\max\nolimits_{1\leq j\leq\numK}\left\{
\|\phiv_{j}\|^{2}- z_{j}^{2}
\right\}>0
\bigr).
\end{EQA}
Let us approximate the \(\max_{1\leq j\leq\numK}\) function using the smooth maximum:
\begin{EQA}[c]
\label{def:h_beta}
h_{\beta}\left(\{x_{j}\}\right)\eqdef
\beta^{-1}\log\left(\sum\nolimits_{j=1}^{\numK} \ex^{\beta x_{j}}\right)~\text{for }\beta>0,\, x_{j}\in \R,\\
\label{ineq:h_beta}
h_{\beta}\left(\{x_{j}\}\right)-\beta^{-1}\log(\numK)\leq \max_{1\leq j\leq\numK}\{x_{j}\}
\leq h_{\beta}\left(\{x_{j}\}\right).\quad\quad
\end{EQA}
The indicator function \(\Ind\{x > 0\}\) is approximated with the three times  differentiable function \(g(x)\) growing monotonously from \(0\) to \(1\):
\begin{EQA}
\label{def:g}
g(x)&\eqdef&
\begin{cases}
0, & x \leq 0,\\
16x^{3}/3, & x\in [0,1/4],\\
0.5+2(x-0.5)-16(x-0.5)^{3}/3, & x\in [1/4,3/4],\\
1+16(x-1)^{3}/3, & x\in [3/4,1],\\
1, & x\geq 1.
\end{cases}
\end{EQA}
It holds for all \(x\in \R\) and \(\Delta>0\)
\begin{EQA}[c]
\Ind \left\{ x > \Delta\right\} \leq g(x/\Delta)\leq  \Ind \left\{ x/\Delta > 0\right\}
.
\end{EQA}
Therefore
\begin{EQA}
&&\nquad
\P\left(
\max\limits_{1\leq j\leq\numK}\left\{
\|\phiv_{j}\|-z_{j}
\right\}
>\Delta
\right)
\\&\leq&
\E\Ind\left(
\max\limits_{1\leq j\leq\numK}\left\{
\frac{\|\phiv_{j}\|^{2}- z_{j}^{2}}{2z_{j}}
\right\}>\Delta
\right)
\\&\leq&
\E g\left(\max\limits_{1\leq j\leq\numK}\left\{
\frac{\|\phiv_{j}\|^{2}- z_{j}^{2}}{2z_{j}\Delta}
\right\}
\right)
\\
\label{eq:Eg1}
&\leq&
\E g\left(
\frac{1}{\Delta\beta}\log\left\{
\sum\nolimits_{j=1}^{\numK}
\exp\left[\beta\frac{\|\phiv_{j}\|^{2}- z_{j}^{2}}{2z_{j}}\right]
\right\}
\right)
\\&\leq&
\E g\left(\max\limits_{1\leq j\leq\numK}\left\{
\frac{\|\phiv_{j}\|^{2}- z_{j}^{2}}{2z_{j}\Delta}
\right\}+\frac{\log(\numK)}{\beta\Delta}
\right)
\\&\leq&
\E\Ind\left(
\max\limits_{1\leq j\leq\numK}\left\{
\frac{\|\phiv_{j}\|^{2}- z_{j}^{2}}{2z_{j}}
\right\}>-\frac{\log(\numK)}{\beta}
\right)
\\&\leq&
\P\left(
\max\limits_{1\leq j\leq\numK}\left\{
\|\phiv_{j}\|-z_{j}
\right\}
>
-1.5\frac{\log(\numK)}{\beta}
\right),
\label{eq:Eg2}
\end{EQA}
where the last inequality holds for \(z_{j}\geq 2.25\log(\numK)/\beta\).
Denote
\begin{EQA}[c]
\label{def:zv}
\zv\eqdef\left(z_{1},\dots,z_{\numK}\right)^{\T}\in \R^{\numK}, ~ z_{j}>0.
\end{EQA}
Introduce the function \(F_{\Delta,\,\beta}(\Phi,\zv):\R^{\dimtotal}\times\R^{\numK}\mapsto \R\):
\begin{EQA}
\label{def:F_Delta_beta}
F_{\Delta,\,\beta}(\Phi,\zv)
 &\eqdef&
g\left(
\frac{1}{\Delta\beta}\log\left\{
\sum\nolimits_{j=1}^{\numK}
\exp\left[\beta\frac{\|\phiv_{j}\|^{2}- z_{j}^{2}}{2z_{j}}\right]
\right\}
\right)
\end{EQA}
Then by \eqref{eq:Eg1} and \eqref{eq:Eg2}
\begin{EQA}[c]
\hspace{-7cm}
\P\left(
\max\limits_{1\leq j\leq\numK}\left\{
\|\phiv_{j}\|-z_{j}
\right\}
>\Delta
\right)
\\
\label{ineq:F_Db_appr1}
\hspace{-2cm}
\leq
\E F_{\Delta,\,\beta}(\Phi,\zv)
\\
\hspace{5cm}
\leq
\P\left(
\max\limits_{1\leq j\leq\numK}\left\{
\|\phiv_{j}\|-z_{j}
\right\}
>
-\frac{3\log(\numK)}{2\beta}
\right).
\label{ineq:F_Db_appr2}
\end{EQA}
Lemma \ref{lemma:prop_F_beta_Delta} checks that \(F_{\Delta,\,\beta}\left(\cdot,\zv\right)\) admits applying the Lindeberg's telescopic sum device (see \cite{Lindeberg1922neue}) in order to approximate \(\E F_{\Delta,\,\beta}\left(\Phi,\zv\right)\) with  \(\E F_{\Delta,\,\beta}\left(\Phibar,\zv\right)\).
Define for \(q=2,\dots,n-1\) the following \(\R^{\dimtotal}\)-valued random sums:
\begin{EQA}[c]
S_{q}\eqdef \sum\limits_{i=1}^{q-1}\Phibar_{i}+\sum\limits_{i=q+1}^{n}\Phi_{i},
\quad\quad S_{1}\eqdef  \sum\limits_{i=2}^{n}\Phi_{i}, \quad\quad
S_{n}\eqdef \sum\limits_{i=1}^{n-1}\Phibar_{i}.
\end{EQA}
The difference \(F_{\Delta,\,\beta}\left(\Phi,\zv\right)-F_{\Delta,\,\beta}\left(\Phibar,\zv\right)\) can be represented as the telescopic sum:
\begin{EQA}
F_{\Delta,\,\beta}\left(\Phi,\zv\right)-F_{\Delta,\,\beta}\left(\Phibar,\zv\right)
&=&
\sum\nolimits_{i=1}^{n}
\left\{
F_{\Delta,\,\beta}(S_{i}+\Phi_{i},\zv)-F_{\Delta,\,\beta}(S_{i}+\Phibar_{i},\zv)
\right\}.
\end{EQA}
The third order Taylor expansions of \(F_{\Delta,\,\beta}(S_{i}+\Phi_{i},\zv)\) and \(F_{\Delta,\,\beta}(S_{i}+\Phibar_{i},\zv)\) w.r.t. the first argument at \(S_{i}\), and Lemma \ref{lemma:prop_F_beta_Delta} imply for each \(i=1,\dots,n\):
\begin{EQA}
&&
\Bigl|
F_{\Delta,\,\beta}(S_{i}+\Phi_{i},\zv)-F_{\Delta,\,\beta}(S_{i}+\Phibar_{i},\zv)-
\nabla_{\Phi} F_{\Delta,\,\beta}(S_{i},\zv)^{\T}(\Phi_{i}-\Phibar_{i})
\\&&\hspace{0.3cm}
-\,\frac{1}{2}(\Phi_{i}-\Phibar_{i})^{\T}\nabla_{\Phi}^{2} F_{\Delta,\,\beta}(S_{i},\zv) (\Phi_{i}+\Phibar_{i})
\Bigr|
\\&\leq& \frac{\CONST_{3}(\Delta,\beta)}{6}
\left(\max_{1\leq j \leq \numK}\left\{\|S_{j,i}+\phiv_{j,i}\|^{3}\right\}\|\Phi_{i}\|^{3}_{\max}+\max_{1\leq j \leq \numK}\left\{\|S_{j,i}+\phivbar_{j,i}\|^{3}\right\}\|\Phibar_{i}\|^{3}_{\max}\right),
\end{EQA}
where the value \(\CONST_{3}(\Delta,\beta)\) is defined in Lemma \ref{lemma:prop_F_beta_Delta}, and the random vectors \(S_{j,i}\in\R^{\dimp_{j}}\) for \(j=1,\dots,\numK\) are s.t. for all \(i=1,\dots,n\)
\begin{EQA}[c]
S_{i}=\left(S_{1,i}^{\T},S_{2,i}^{\T},\dots,S_{\numK,i}^{\T}\right)^{\T}.
\end{EQA}
By their construction \(S_{i}\) and \(\Phi_{i}-\Phibar_{i}\) are independent, \(\E\Phi_{i}=\E\Phibar_{i}=0\) and \(\Var\Phi_{i}=\Var\Phibar_{i}\), therefore
\begin{EQA}
&&\nquad
\left|\E F_{\Delta,\,\beta}(\Phi,\zv)-\E F_{\Delta,\,\beta}(\Phibar,\zv)\right|
\\&=&
\left|
\sum\nolimits_{i=1}^{n}
\left\{
\E H_{\Delta}(S_{i}+\Phi_{i},\zv)-\E H_{\Delta}(S_{i}+\Phibar_{i},\zv)
\right\}
\right|\\
&\leq&
\frac{\CONST_{3}(\Delta,\beta)}{6}
\sum_{i=1}^{n}
\E
\left(\max_{1\leq j \leq \numK}\left\{\|S_{j,i}+\phiv_{j,i}\|^{3}\right\}\|\Phi_{i}\|^{3}_{\max}+\max_{1\leq j \leq \numK}\left\{\|S_{j,i}+\phivbar_{j,i}\|^{3}\right\}\|\Phibar_{i}\|^{3}_{\max}\right).
\end{EQA}
Lemma \ref{lemma:max_cube} implies for all \(i=1,\dots, n\) with probability \(\geq 1-2\ex^{-\xx}\)
\begin{EQA}[c]
\left(\E\max_{1\leq j \leq \numK}\left\{\|S_{j,i}+\phiv_{j,i}\|^{6}\right\}\right)^{1/2} \leq \CONST\nunu \max_{1\leq j \leq \numK}\|\Var^{1/2}(\phiv_{j})\|^{3}\sqrt{\dimp_{\max}\log(K)}(\dimp_{\max}+6\yy),
\end{EQA}
and the same bound holds for \(\left(\E\max_{1\leq j \leq \numK}\left\{\|S_{j,i}+\phivbar_{j,i}\|^{6}\right\}\right)^{1/2}\).
 Denote
\begin{EQA}
\label{def:deltamax}
\delta_{\max,\phi}&\eqdef&
\frac{1}{2}
\sum_{i=1}^{n}
\left\{\E\left(\|\Phi_{i}\|^{6}_{\max}\right)\right\}^{1/2}
+\left\{\E\left(\|\Phibar_{i}\|^{6}_{\max}\right)\right\}^{1/2}.
\end{EQA}
By Lemma \ref{lemma:max_phi3} it holds for \(t = \left(\yy+\log(\dimtotal)\right)^{3} \left(\sqrt{2}\cf_{\phi}\nunu\right)^{6}n^{-3}\) with probability \(\geq 1-\ex^{-\xx}\)
\begin{EQA}[c]
\|\Phi_{i}\|^{6}_{\max} \leq t, \quad
\|\Phibar_{i}\|^{6}_{\max} \leq t.
\end{EQA}

If \(\yy=\CONST \log n\), then the last bound on \(\left|\E F_{\Delta,\,\beta}(\Phi,\zv)-\E F_{\Delta,\,\beta}(\Phibar,\zv)\right|\) continues  with probability \(\geq 1-6\exp(-\xx)\) as follows
\begin{EQA}
&&\nquad
\left|\E F_{\Delta,\,\beta}(\Phi,\zv)-\E F_{\Delta,\,\beta}(\Phibar,\zv)\right|
\\&\leq&
\CONST\frac{\CONST_{3}(\Delta,\beta)}{3} \sqrt{\dimp_{\max}^{3}\log(K)}\delta_{\max,\phi}\max_{1\leq j \leq \numK}\|\Var^{1/2}(\phiv_{j})\|^{3}
\\&\leq&\frac{\CONST}{3} \left(\frac{1}{\Delta^{3}}
+
\frac{\beta}{\Delta^{2}}
+
\frac{\beta^{2}}{\Delta}
\right)
\frac{\dimp_{\max}^{3/2}}{n^{1/2}}
\log^{1/2}(\numK)
\log^{3/2}\left(n\dimtotal\right)
 \max_{1\leq j \leq \numK}\|\Var^{1/2}(\phiv_{j})\|^{3}
 \left(2\nunu^{2}\cf_{\phi}^{2}\right)^{3/2}
\\&\eqdef& \delta_{3,\phi}(\Delta,\beta) .
\label{ineq:Efb_appr}
\end{EQA}
The derived bounds imply:
\begin{EQA}
&&\nquad\nquad
\P\left(\bigcup\nolimits_{j=1}^{\numK}\left\{\|\phiv_{j}\|>z_{j}\right\} \right)\\
&\overset{\text{by {\small\eqref{ineq:F_Db_appr1}}}}{\leq}&
\E F_{\Delta,\,\beta}\left(\Phi,\zv-\Delta\Idv_{\numK}\right)\\
\label{ineq:smooth_thechn}
&\overset{\text{by {\small\eqref{ineq:Efb_appr}}}}{\leq}&
\E F_{\Delta,\,\beta}\left(\Phibar,\zv-\Delta\Idv_{\numK}\right)+
\delta_{3,\phi}(\Delta,\beta)
\\
&\overset{\text{by } {\small\,\eqref{ineq:F_Db_appr2}}}{\leq}&
\P\left(\bigcup\nolimits_{j=1}^{\numK}\left\{\|\phivbar_{j}\|>z_{j}-\Delta-\frac{3\log(\numK)}{2\beta}\right\}\right)
+
\delta_{3,\phi}(\Delta,\beta)
,
\end{EQA}
and similarly
\begin{EQA}
&&\nquad\P\left(\bigcup\nolimits_{j=1}^{\numK}\left\{\|\phiv_{j}\|>z_{j}\right\} \right)\\
&\geq&
\P\left(\bigcup\nolimits_{j=1}^{\numK}\left\{\|\phivbar_{j}\|>z_{j}+\frac{3\log(\numK)}{2\beta}+\Delta\right\}\right)
-\delta_{3,\phi}(\Delta,\beta).
\end{EQA}
\end{proof}
The next lemma is formulated separately, since it is used for a proof of another result.
\begin{lemma}[Smooth uniform GAR]
\label{lemma:smooth_techn_gauss_appr_phiv}
Under the conditions of Lemma \ref{lemma:gauss_appr_phiv} it holds with dominating probability for the function \(F_{\Delta,\,\beta}\left(\cdot,\zv\right)\) given in  \eqref{def:F_Delta_beta}:
\begin{EQA}
\text{1.1.}&& \P\left(
\bigcup\nolimits_{j=1}^{\numK}\left\{\|\phiv_{j}\|> z_{j}\right\}
\right)
\leq
\E F_{\Delta,\,\beta}\left(\Phibar,\zv-\Delta\Idv_{\numK}\right)+
\delta_{3,\phi}(\Delta,\beta),\\
\text{1.2.}&& \P\left(\bigcup\nolimits_{j=1}^{\numK}\left\{\|\phiv_{j}\|>z_{j}\right\} \right)
\geq
\E H_{\Delta,\,\beta}\left(\Phibar,\zv+\frac{3\log(K)}{2\beta}\Idv_{\numK}\right)
-\delta_{3,\phi}(\Delta,\beta);
\\\text{2.1.} &&
\E F_{\Delta,\,\beta}\left(\Phi,\zv\right) \leq
\P\left(\bigcup\nolimits_{j=1}^{\numK}\left\{\|\phiv_{j}\|>z_{j}-\frac{3\log(K)}{2\beta}\right\} \right),\\
\text{2.2.} &&
\E F_{\Delta,\,\beta}\left(\Phi,\zv\right) \geq
\P\left(\bigcup\nolimits_{j=1}^{\numK}\left\{\|\phiv_{j}\|>z_{j}+\Delta\right\} \right).
\end{EQA}
\end{lemma}
\begin{proof}[Proof of Lemma \ref{lemma:smooth_techn_gauss_appr_phiv}]
The first inequality 1.1 is obtained in \eqref{ineq:smooth_thechn}, the second inequality 1.2 follows similarly from \eqref{ineq:F_Db_appr2} and \eqref{ineq:Efb_appr}. The inequalities 2.1 and 2.2 are given in  \eqref{ineq:F_Db_appr1} and \eqref{ineq:F_Db_appr2}.
\end{proof}
\begin{lemma}
\label{lemma:max_phi3}
Let for some \(\cf_{\phi},\gm_{1},\nunu>0\)and for all \(i=1,\dots,n\), \(j=1,\dots,\dimtotal\)
\begin{EQA}
\log\E\exp\left\{\lambda \sqrt{n}|\phi_{i}^{j}|/\cf_{\phi}\right\} &\leq&  \lambda^{2}\nunu^{2}/2,\quad  |\lambda|<\gm_{1},
\end{EQA}
here \(\phi_{i}^{j}\) denotes the \(j\)-th coordinate of vector \(\phiv_{i}\). Then it holds for all \(i=1,\dots,n\) and \(m,t>0\)
\begin{EQA}
\P\left(
\max_{1\leq j\leq \dimtotal}|\phi_{i}^{j}|^{m} > t
\right)&\leq&
\exp\left\{
-\frac{nt^{2/m}}{2\cf_{\phi}^{2}\nunu^{2}}
+\log(\dimtotal)
\right\}.
\end{EQA}

\end{lemma}
\begin{proof}[Proof of Lemma \ref{lemma:max_phi3}]
Let us bound the \(\max_{j}|\phi_{i}^{j}|\) using the following bound for the maximum:
\begin{EQA}[c]
\max_{1\leq j\leq \dimtotal}|\phi_{i}^{j}| \leq \log\left\{
\sum\nolimits_{j=1}^{\dimtotal} \exp\bigl(|\phi_{i}^{j}| \bigr)
\right\}.
\end{EQA}
By the Lemma's condition
\begin{EQA}
\E\exp\left\{\max_{1\leq j\leq \dimp}
\frac{\lambda \sqrt{n}}{\cf_{\phi}}
|\phi_{i}^{j}|
\right\}
&\leq&
\exp\left(
\lambda^{2}\nunu^{2}/2+\log\dimtotal
\right).
\end{EQA}
Thus, the statement follows from the exponential Chebyshev's inequality.
\end{proof}

\begin{lemma}
\label{lemma:max_cube}
If for the centered random vectors \(\phiv_{j}\in\R^{\dimp_{j}}\) \(j=1,\dots,\numK\)
\begin{EQA}
\sup_{\substack{\gammav\in\R^{\dimp_{j}},\\\|\gammav\|\neq 0}}\log\E\exp
\left\{\lambda
\frac{\gammav^{\T}\phiv_{j}}{\|\Var^{1/2}(\phiv_{j})\gammav\|}
\right\}
&\leq&\nunu^{2}\lambda^{2}/2,\quad\quad|\lambda|\leq \gm
\end{EQA}
for some constants \(\nunu> 0\) and  \(\gm\geq\nunu^{-1}\max_{1\leq j\leq \numK}\sqrt{2\dimp_{j}\log(K)}\), then
\begin{EQA}
\E\max_{1\leq j\leq\numK} \left\{\|\phiv_{j}\|\right\}
&\leq&\CONST\nunu\max_{1\leq j\leq \numK}\|\Var^{1/2}(\phiv_{j})\|
\sqrt{2\dimp_{\max}\log(K)},
\\
\left(\E\max_{1\leq j\leq\numK} \left\{\|\phiv_{j}\|^{6}\right\}\right)^{1/2}
&\leq& \CONST\nunu\max_{1\leq j\leq \numK}\|\Var^{1/2}(\phiv_{j})\|^{3}
\sqrt{2\dimp_{\max}\log(K)}(\dimp_{\max}+6\yy),
\end{EQA}
The second bound holds with probability \(\geq 1-2\ex^{-\xx}\). 
\end{lemma}
\begin{proof}[Proof of Lemma \ref{lemma:max_cube}]
Let us take for each \(j=1,\dots,\numK\) finite \(\varepsilon_{j}\)-grids \(\Grid_{j}(\varepsilon)\subset \R^{\dimp_{j}}\) on the \((\dimp_{j}-1)\)-spheres of radius \(1\) s.t
\begin{EQA}[c]
\forall \gammav\in\R^{\dimp_{j}} \text{ s.t. } \|\gammav\|=1 \ ~\exists \gammav_{0}\in\Grid_{j}(\varepsilon):~\|\gammav-\gammav_{0}\|\leq \varepsilon,~\|\gammav_{0}\|=1.
\end{EQA}
Then
\begin{EQA}
\|\phiv_{j}\|&\leq&
(1-\varepsilon_{j})^{-1}
\max_{\gammav\in \Grid_{j}(\varepsilon_{j})}
\left\{\gammav^{\T}\phiv_{j}\right\}.
\end{EQA}
Hence, by inequality \eqref{ineq:h_beta} and the imposed condition it holds for all\\ \(0<\mu< \gm /\max_{1\leq j\leq \numK}\|\Var^{1/2}(\phiv_{j})\|\):
\begin{EQA}
\E\max_{1\leq j\leq\numK} \left\{\|\phiv_{j}\|\right\}
&\leq&\max_{1\leq j\leq\numK}
\frac{1}{1-\varepsilon_{j}}\E\max_{1\leq j\leq\numK}\max_{\gammav\in \Grid_{j}(\varepsilon_{j})}
\left\{\gammav^{\T}\phiv_{j}\right\}
\\&\leq&
\CONST
\frac{1}{\mu}
\E\log\left\{
\sum_{1\leq j\leq \numK}
\sum_{\gammav\in \Grid_{j}(\varepsilon_{j})}
\exp\left(\mu\gammav^{\T}\phiv_{j}\right)
\right\}
\\&\leq&\CONST
\frac{1}{\mu}
\log\left\{
\sum_{1\leq j\leq \numK}
\sum_{\gammav\in \Grid_{j}(\varepsilon_{j})}
\E\exp\left(\mu\gammav^{\T}\phiv_{j}\right)
\right\}
\\&\leq&
\CONST
\max_{1\leq j\leq \numK}\frac{\log(\numK\card\left\{\Grid_{j}(\varepsilon_{j})\right\})}{\mu}
+
\CONST\frac{\mu\nunu^{2}}{2}\max_{1\leq j\leq \numK}
\|\Var(\phiv_{j})\|
\\&\leq&
\CONST\max_{1\leq j\leq \numK}\{\dimp_{j}\}\frac{\log(K)}{\mu}
+
\CONST\frac{\mu\nunu^{2}}{2}\max_{1\leq j\leq \numK}
\|\Var(\phiv_{j})\|
\\&=&\CONST\nunu\max_{1\leq j\leq \numK}\{\sqrt{\dimp_{j}}\}\max_{1\leq j\leq \numK}\|\Var^{1/2}(\phiv_{j})\|
\sqrt{2\log(K)}
\\&&
\text{for }~\mu=\CONST\nunu^{-1}\max_{1\leq j\leq \numK}\{\sqrt{\dimp_{j}}\}\sqrt{2\log(K)}/\max_{1\leq j\leq \numK}\|\Var^{1/2}(\phiv_{j})\|.
\end{EQA}
For the second part of the statement we combine the first part with the result of Theorem \ref{qf_dev} on deviation of a random quadratic form: it holds with dominating probability for \(V_{\phi_{j}}^{2} \eqdef \Var \phiv_{j}\)
\begin{EQA}
\|\phiv_{j}\|^{2} &\leq& \ZZqf^{2}(\yy, V_{\phi_{j}})
\\&\leq&
 \tr(V_{\phi_{j}}^{2})+6\yy\|V_{\phi_{j}}^{2} \|
 \leq
 \|V_{\phi_{j}}^{2} \|(\dimp_{j}+6\yy)
 .
\end{EQA}

\end{proof}

\begin{lemma}
\label{lemma:prop_F_beta_Delta}
Let \(\Gamma \in\R^{\dimtotal}\), \(\gammav_{j}\in\R^{\dimp_{j}}\) for \(j=1,\dots, \numK\) are s.t. \(\Gamma=\left(\gammav_{1}^{\T},\dots,\gammav_{\numK}^{\T}\right)^{\T}\), and \(\zv\eqdef\left(z_{1},\dots,z_{\numK}\right)^{\T}\) s.t. \(z_{j}\geq \sqrt{\dimp_{j}}\), then it holds for the function \(F_{\Delta,\,\beta}\left(\cdot,\zv\right)\) defined in \eqref{def:F_Delta_beta}:
\begin{EQA}[ccllcl]
\left\|\nabla^{2}_{\Gamma}F_{\Delta,\,\beta}\left(\Gamma,\zv\right)\right\|_{1}
&\leq&
\CONST_{2}(\Delta,\beta)\max_{1\leq j\leq \numK} \left\{\|\gammav_{j}\|^{2}\right\},
\quad &\CONST_{2}(\Delta,\beta)&\eqdef&
\CONST\left(
\frac{1}{\Delta^{2}}+ \frac{\beta}{\Delta}
\right),
\\
\left\|\nabla^{3}_{\Gamma}F_{\Delta,\,\beta}\left(\Gamma,\zv\right)\right\|_{1}
&\leq&
\CONST_{3}(\Delta,\beta)\max_{1\leq j\leq \numK} \left\{\|\gammav_{j}\|^{3}\right\},\quad
&\CONST_{3}(\Delta,\beta)&\eqdef&
\CONST \left(\frac{1}{\Delta^{3}}
+
\frac{\beta}{\Delta^{2}}
+
\frac{\beta^{2}}{\Delta}
\right).
\label{def:C3_Delta_beta}
\end{EQA}
\end{lemma}
\begin{proof}[Proof of Lemma \ref{lemma:prop_F_beta_Delta}]
Denote
\begin{EQA}[c]
\label{def:sGamma}
s(\Gamma)\eqdef
\sum\nolimits_{j=1}^{\numK}
\exp\left(\beta\frac{\|\gammav_{j}\|^{2}- z_{j}^{2}}{2z_{j}}\right),
\quad
h_{\beta}(s(\Gamma))\eqdef \beta^{-1}\log\left\{s(\Gamma)\right\},
\end{EQA}
then \(F_{\beta,\Delta}(\Gamma,\zv)
 = g\left(\Delta^{-1}h_{\beta}\left(s(\Gamma)\right)\right).\)
Let \(\gamma^{q}\) denote the \(q\)-th coordinate of the vector \(\Gamma\in\R^{\dimtotal}\).
It holds for \(q,l,b,r=1,\dots,\dimtotal\):
\begin{EQA}
\frac{d}{d\gamma^{q}}F_{\beta,\Delta}(\Gamma,\zv)&=&
\frac{1}{\Delta}g^{\prime}\left\{\Delta^{-1}h_{\beta}(s(\Gamma))\right\}
\frac{d}{d\gamma^{q}}h_{\beta}(s(\Gamma)),
\\
\frac{d^{2}}{d\gamma^{q}d\gamma^{l}}F_{\beta,\Delta}(\Gamma,\zv)&=&
\frac{1}{\Delta^{2}}g^{\prime\prime}\left\{\Delta^{-1}h_{\beta}(s(\Gamma))\right\}
\frac{d}{d\gamma^{q}}h_{\beta}(s(\Gamma)) \frac{d}{d\gamma^{l}}h_{\beta}(s(\Gamma))
\\&&+\,
\frac{1}{\Delta}g^{\prime}\left\{\Delta^{-1}h_{\beta}(s(\Gamma))\right\}
\frac{d^{2}}{d\gamma^{q}d\gamma^{l}}h_{\beta}(s(\Gamma)),
\\
\frac{d^{3}}{d\gamma^{q}d\gamma^{l}d\gamma^{b}}F_{\beta,\Delta}(\Gamma,\zv)&=&
\frac{1}{\Delta^{3}}g^{\prime\prime\prime}\left\{\Delta^{-1}h_{\beta}(s(\Gamma))\right\}
\frac{d}{d\gamma^{q}}h_{\beta}(s(\Gamma)) \frac{d}{d\gamma^{l}}h_{\beta}(s(\Gamma))
\frac{d}{d\gamma^{b}}h_{\beta}(s(\Gamma))
\\&&+\,
\frac{1}{\Delta^{2}}g^{\prime\prime}\left\{\Delta^{-1}h_{\beta}(s(\Gamma))\right\}
\Biggl\{
\frac{d^{2}}{d\gamma^{q}d\gamma^{b}}h_{\beta}(s(\Gamma)) \frac{d}{d\gamma^{l}}h_{\beta}(s(\Gamma))
\\&&+\,\frac{d}{d\gamma^{q}}h_{\beta}(s(\Gamma)) \frac{d^{2}}{d\gamma^{l}d\gamma^{b}}h_{\beta}(s(\Gamma))
+\frac{d}{d\gamma^{b}}h_{\beta}(s(\Gamma)) \frac{d^{2}}{d\gamma^{q}d\gamma^{l}}h_{\beta}(s(\Gamma))
\Biggr\}
\\&&+\,
\frac{1}{\Delta}g^{\prime}\left\{\Delta^{-1}h_{\beta}(s(\Gamma))\right\}
\frac{d^{3}}{d\gamma^{q}d\gamma^{l}d\gamma^{b}}h_{\beta}(s(\Gamma)).
\end{EQA}
Let for \(1\leq q\leq \dimtotal\) \(j(q)\) denote an index from \(1\) to \(\numK\) s.t. the coordinate \(\gamma^{q}\) of the vector \(\Gamma=\left(\gammav_{1}^{\T},\dots,\gammav_{\numK}^{\T}\right)^{\T}\) belongs to its sub-vector \(\gammav_{j(q)}\).
\begin{EQA}
\label{eq:h1}
\frac{d}{d\gamma^{q}}h_{\beta}(s(\Gamma))&=&
\frac{1}{\beta}\frac{1}{s(\Gamma)}
\frac{d}{d\gamma^{q}}s(\Gamma)
=
\frac{1}{s(\Gamma)}
\frac{\gamma^{q}}{z_{j(q)}}
\exp\left(\beta\frac{\|\gammav_{j(q)}\|^{2}- z_{j(q)}^{2}}{2z_{j(q)}}\right)
,
\end{EQA}
\begin{EQA}
\label{eq:h2}
\frac{d^{2}}{d\gamma^{q}d\gamma^{l}}h_{\beta}(s(\Gamma))&=&
\frac{1}{\beta}\frac{1}{s(\Gamma)}
\frac{d^{2}}{d\gamma^{q}d\gamma^{l}}s(\Gamma)
-\frac{1}{\beta}\frac{1}{s^{2}(\Gamma)}
\frac{d}{d\gamma^{q}}s(\Gamma)
\frac{d}{d\gamma^{l}}s(\Gamma)
\\&&\hspace{-3cm}=
{\displaystyle{
\begin{cases}
{\displaystyle{
\left\{
\frac{1}{z_{j(q)}}
+\beta\left(\frac{\gamma^{q}}{z_{j(q)}}\right)^{2}
\right\}
\frac{1}{s(\Gamma)}
\exp\left(\beta\frac{\|\gammav_{j(q)}\|^{2}- z_{j(q)}^{2}}{2z_{j(q)}}\right)
}}
\vspace{0.2cm}
\\
\vspace{0.4cm}
\hspace{0.9cm}{\displaystyle{
-\frac{\beta}{s^{2}(\Gamma)}
\left\{\frac{\gamma^{q}}{z_{j(q)}}\right\}^{2}
\exp\left(2\beta\frac{\|\gammav_{j(q)}\|^{2}- z_{j(q)}^{2}}{2z_{j(q)}}\right)
}}
,& q=l;\\
{\displaystyle{
\frac{\beta}{s(\Gamma)}
\frac{\gamma^{q} \gamma^{l}}{z_{j(q)}^{2}}\exp\left(\beta\frac{\|\gammav_{j(q)}\|^{2}- z_{j(q)}^{2}}{2z_{j(q)}}\right)
}}
\vspace{0.2cm}
\\
\vspace{0.4cm}
\hspace{0.9cm}{\displaystyle{
-\frac{\beta}{s^{2}(\Gamma)}
\frac{\gamma^{q} \gamma^{l}}{z_{j(q)}^{2}}
\exp\left(2\beta\frac{\|\gammav_{j(q)}\|^{2}- z_{j(q)}^{2}}{2z_{j(q)}}
\right)}}
,& j(q)= j(l), q\neq l;\\
{\displaystyle{
-\frac{\beta}{s^{2}(\Gamma)}
\frac{\gamma^{q} \gamma^{l}}{z_{j(q)}z_{j(l)}}
\exp\left(\beta\frac{\|\gammav_{j(q)}\|^{2}- z_{j(q)}^{2}}{2z_{j(q)}}
+
\beta\frac{\|\gammav_{j(l)}\|^{2}- z_{j(l)}^{2}}{2z_{j(l)}}\right)}}
,& j(q)\neq j(l).
\end{cases}
}}
\end{EQA}
  By definition \eqref{def:sGamma} of \(s(\Gamma)\) it holds for all \(\Gamma\in\R^{\dimtotal}\):
\begin{EQA}[c]
\frac{1}{s(\Gamma)}\exp\left(\beta\frac{\|\gammav_{j}\|^{2}- z_{j}^{2}}{2z_{j}}\right)
\leq1, \quad
\sum_{j=1}^{\numK}\frac{1}{s(\Gamma)}\exp\left(\beta\frac{\|\gammav_{j}\|^{2}- z_{j}^{2}}{2z_{j}}\right)=1.
\end{EQA}
Therefore,
\begin{EQA}
\sum_{q,l=1}^{\dimtotal}\left|
\frac{d}{d\gamma^{q}}h_{\beta}(s(\Gamma)) \frac{d}{d\gamma^{l}}h_{\beta}(s(\Gamma))
\right|
&\leq&
\left\{
\sum_{j=1}^{\numK}
\frac{1}{s(\Gamma)z_{j}}
\exp\left(\beta\frac{\|\gammav_{j}\|^{2}- z_{j}^{2}}{2z_{j}}\right)
\sum_{q=1}^{\dimp_{j}}\gamma^{q}
\right\}^{2}
\\
&\leq&
 \left|\max_{1\leq j\leq \numK} \|\gammav_{j}\| \frac{\sqrt{\dimp_{j}}}{z_{j}}\right|^{2}
 \\&\leq&
 \max_{1\leq j\leq \numK} \|\gammav_{j}\|^{2} \quad
 \text{for } z_{j}\geq \sqrt{\dimp_{j}}
.
\end{EQA}
Similarly
\begin{EQA}
\sum_{q,l=1}^{\dimtotal}\left|
\frac{d^{2}}{d\gamma^{q}d\gamma^{l}}h_{\beta}(s(\Gamma))
\right|
&\leq&\CONST  \beta\max_{1\leq j\leq \numK}\|\gammav_{j}\|^{2},
\\
\sum_{q,l,b=1}^{\dimtotal}\left|
\frac{d^{2}}{d\gamma^{q}d\gamma^{l}}h_{\beta}(s(\Gamma))
\frac{d}{d\gamma^{b}}h_{\beta}(s(\Gamma))
+
\frac{d^{3}}{d\gamma^{q}d\gamma^{l}d\gamma^{b}}h_{\beta}(s(\Gamma))
\right|
&\leq& \CONST \left(\beta+\beta^{2}\right)\max_{1\leq j\leq \numK}\|\gammav_{j}\|^{3}.
\end{EQA}
\end{proof}

\subsection{Gaussian comparison}
\label{sect:Gausscompar}
The following Lemma shows how to compare the expected values of a twice differentiable function evaluated at the independent centered Gaussian vectors. This statement is used for the Gaussian comparison step in the scheme \eqref{rectangle_simult}. The proof of the result is based on the Gaussian interpolation method introduced by \cite{Stein1981MeanGaussEst} and \cite{Slepian1962onesided} (see also \cite{Rollin2013Stein} and \cite{Chernozhukov2013CLT1supp}  and references therein). The proof is given here in order to keep the text self-contained.
\begin{lemma}[Gaussian comparison using Slepian interpolation]
\label{lemma:GaussCOmpar_Slepian}
Let the \(\R^{\dimtotal}\)-dimensional random centered vectors \(\Phibar\) and \(\Psibar\) 
be independent and  normally distributed, \(f(Z):\R^{\dimtotal}\mapsto\R\) is any twice differentiable function s.t. the expected values in the expression below are bounded. Then it holds 
\begin{EQA}
\left|\E f(\Phibar)-\E f(\Psibar)\right| &\leq&
\frac{1}{2}
\left\|\Var\Phibar-\Var\Psibar\right\|_{\max}
\sup_{t\in[0,1]}\left\|\E\nabla^{2}f\left(\Phibar\sqrt{t}+\Psibar\sqrt{1-t}\right)\right\|_{1}.
\end{EQA}
\end{lemma}
\begin{proof}[Proof of Lemma \ref{lemma:GaussCOmpar_Slepian}]
Introduce for \(t\in[0,1]\) the Gaussian vector process \(Z_{t}\) and the deterministic scalar-valued function \(\varkappa(t)\):
\begin{EQA}
Z_{t}&\eqdef& \Phibar\sqrt{t}+\Psibar\sqrt{1-t}\in\R^{\dimtotal},\\
\varkappa(t)&\eqdef& \E f(Z(t)),
\end{EQA}
then \(\E f(\Phibar)=\varkappa(1)\), \(\E f(\Psibar)=\varkappa(0)\) and
\begin{EQA}[c]
\left|\E f(\Phibar)-\E f(\Psibar)\right|=\left|\varkappa(1)-\varkappa(0)\right|
\leq\int_{0}^{1}\left|\varkappa^{\prime}(t)\right| dt.
\end{EQA}
Let us consider \(\varkappa^{\prime}(t)\):
\begin{EQA}
\varkappa^{\prime}(t)&=&\frac{d}{dt}\E f(Z_{t})=
\E\left[\left\{\nabla f(Z_{t})\right\}^{\T} \frac{d}{dt}Z_{t} \right]
\\&=&\frac{1}{2\sqrt{t}}\E\left\{
\Phibar^{\T}\nabla f(Z_{t})\right\}-
\frac{1}{2\sqrt{1-t}}\E\left\{
\Psibar^{\T}\nabla f(Z_{t})\right\}.
\label{eq:kappa_prime}
\end{EQA}
Further we use the Gaussian integration by parts formula (see e.g Section A.6 in \cite{Talagrand2003spin}): if \((x_{1},\dots, x_{\dimtotal})^{\T}\) is a centered Gaussian vector and \(f(x_{1},\dots, x_{\dimtotal})\) is s.t. the integrals below exist, then it holds for all \(j=1,\dots,\dimtotal\):
\begin{EQA}
\label{eq:Gauss_Int}
\E\left\{x_{j}f(x_{1},\dots, x_{\dimtotal})\right\}
&=&
\sum_{k=1}^{\dimtotal}\E(x_{j}x_{k})\E\left\{\frac{d}{dx_{k}}f(x_{1},\dots, x_{\dimtotal})\right\}.
\end{EQA}
Let \(\Phibarj, \Psibarj\) denote the \(j\)-th coordinates of \(\Phibar\) and \(\Psibar\). Let also \(\frac{d}{d_{j}}f(Z_{t})\) denote the partial derivative of the vectors \(f(Z_{t})\) w.r.t. the \(j\)-th coordinate of \(Z_{t}\). Then it holds due to \eqref{eq:Gauss_Int}:
\begin{EQA}
\E\left\{
\Phibar^{\T}\nabla f(Z_{t})\right\}&=&
\sum_{j=1}^{\dimtotal}\E\left\{\Phibarj\frac{d}{d_{j}}f(Z_{t})\right\}
=
\sum_{j,q=1}^{\dimtotal}\E\left(\Phibarj \Phibar^{\,q}\right)
\E\left\{\frac{d}{d\Phibar^{\,q} }\frac{d}{d_{j}}f(Z_{t})\right\}
\\&=&
\sqrt{t}\sum_{j,q=1}^{\dimtotal}\E\left(\Phibarj \Phibar^{\,q}\right)
\E\left\{\frac{d^{2}}{d_{q}d_{j}}f(Z_{t})\right\}.
\end{EQA}
Similarly for the second term in \eqref{eq:kappa_prime}:
\begin{EQA}
\E\left\{
\Psibar^{\T}\nabla f(Z_{t})\right\}&=&
\sqrt{1-t}\sum_{j,q=1}^{\dimtotal}\E\left(\Psibarj \Psibar^{\,q}\right)
\E\left\{\frac{d^{2}}{d_{q}d_{j}}f(Z_{t})\right\},
\end{EQA}
therefore
\begin{EQA}
\varkappa^{\prime}(t)&=&
\frac{1}{2}
\sum_{j=1}^{\dimtotal}\sum_{q=1}^{\dimtotal}
\left\{\E\left(\Phibarj \Phibar^{\,q}\right)-\E\left(\Psibarj \Psibar^{\,q}\right)\right\}
\E\left\{\frac{d^{2}}{d_{q}d_{j}}f(Z_{t})\right\}
\\&\leq&
\frac{1}{2}
\left\|\Var\Phibar-\Var\Psibar\right\|_{\max}
\sup_{t\in[0,1]}\left\|\E\nabla^{2}f(Z_{t})\right\|_{1}.
\end{EQA}
\end{proof}

\subsection{Simultaneous anti-concentration for \(\ell_{2}\)-norms of Gaussian vectors}
\label{sect:ac}
\begin{lemma}[Simultaneous Gaussian anti-concentration]
\label{lemma:anti_conc_nonsmooth}
Let  \(\left(\phivbar_{1}^{\T},\dots,\phivbar_{\numK}^{\T}\right)^{\T}\in\R^{\dimtotal}\) be centered normally distributed random vector, and \(\phivbar_{j}\in \R^{\dimp_{j}}\), \(j=1,\dots,\numK\).
  It holds for all \(z_{j}\geq \sqrt{\dimp_{j}}\) and \(0<\Delta_{j}\leq z_{j}\), 
\(j=1,\dots,\numK\):
\begin{EQA}
\P\left(
\bigcup\nolimits_{j=1}^{\numK}\left\{\|\phivbar_{j}\|> z_{j}\right\}
\right)-
\P\left(
\bigcup\nolimits_{j=1}^{\numK}\left\{\|\phivbar_{j}\|> z_{j}+\Delta_{j}\right\}
\right)
&\leq&\Deltaac\left(\left\{\Delta_{j}\right\}\right),
\end{EQA}
where 
\begin{EQA}
\label{def:Deltaac_st}
\Deltaac\left(\left\{\Delta_{j}\right\}\right)&\leq& \CONST\left\{ \kappa
\sqrt{1 \vee \log(\numK/2)}+\CONST\max_{1\leq j\leq\numK}\{\Delta_{j}\}\sqrt{\max_{1\leq j\leq\numK}\log(2 z_{j}/\Delta_{j})}
\right\},
\end{EQA}
 and \(\kappa\eqdef \max_{1 \leq j\leq\numK}\{\Delta_{j}/z_{j}\}\leq 1\) is a deterministic positive constant.
An explicit definition of \(\Deltaac\left(\left\{\Delta_{j}\right\}\right)\) is given in \eqref{def:Deltaac}.
%
\end{lemma}
\begin{proof}[Proof of Lemma \ref{lemma:anti_conc_nonsmooth}]
\begin{EQA}
\label{eq:ac_firstline}
&&\nquad
\P\left(
\bigcup\nolimits_{j=1}^{\numK}\left\{\|\phivbar_{j}\|> z_{j}\right\}
\right)-
\P\left(
\bigcup\nolimits_{j=1}^{\numK}\left\{\|\phivbar_{j}\|> z_{j}+\Delta_{j}\right\}
\right)
\\&\leq&
\P\left(
\bigcup\nolimits_{j=1}^{\numK}\left\{\|\phivbar_{j}\|z_{j}^{-1}-1> 0\right\}
\right)-
\P\left(
\bigcup\nolimits_{j=1}^{\numK}\left\{\|\phivbar_{j}\|z_{j}^{-1}-1> \kappa\right\}
\right)
\\&=&
\P\left(
\max_{1 \leq j\leq\numK}\left\{\|\phivbar_{j}\|z_{j}^{-1}-1 \right\}> 0
\right)-
\P\left(
\max_{1 \leq j\leq\numK}\left\{\|\phivbar_{j}\|z_{j}^{-1}- 1\right\}> \kappa
\right)
\\&\leq&
\P\left(0\leq
\max_{1 \leq j\leq\numK}\left\{\|\phivbar_{j}\|z_{j}^{-1}-1\right\}\leq \kappa
\right).
\label{ineq:chain_ac}
\end{EQA}
It holds
\begin{EQA}
\|\phivbar_{j}\|&=&\sup_{\substack{\gammav\in\R^{\dimp_{j}},\\\|\gammav\|=1}}\left\{\gammav^{\T}\phivbar_{j}\right\}.
\end{EQA}
Let \(\Grid_{j}(\varepsilon_{j})\subset \R^{\dimp_{j}}\) (for \(1\leq j\leq\numK\)) denote a finite \(\varepsilon_{j}\)-net on \((\dimp_{j}-1)\)-sphere of radius \(1\):
\begin{EQA}[c]
\label{def:grid}
\forall \gammav\in\R^{\dimp_{j}} \text{ s.t. } \|\gammav\|=1 \ ~\exists \gammav_{0}\in\Grid_{j}(\varepsilon_{j}):~\|\gammav-\gammav_{0}\|\leq \varepsilon_{j},~\|\gammav_{0}\|=1.
\end{EQA}
This implies for all \(j=1,\dots,\numK\)
\begin{EQA}[c]
\label{eq:grid}
(1-\varepsilon_{j})\|\phivbar_{j}\|
\leq
\max_{\gammav\in \Grid_{j}(\varepsilon_{j})}
\left\{\gammav^{\T}\phivbar_{j}\right\}
\leq
\|\phivbar_{j}\|.
\end{EQA}
Let us take \(\varepsilon_{1},\dots,\varepsilon_{\numK}>0\) s.t. \(\forall\, j=1,\dots,\numK\)
\begin{EQA}[c]
\label{cond:epsilon_Delta}
\varepsilon_{j}\|\phivbar_{j}\|z_{j}^{-1}
\leq
\kappa,
\end{EQA}
then
\begin{EQA}[c]
0
\leq
\max_{1 \leq j\leq\numK}\left\{\frac{\|\phivbar_{j}\|}{z_{j}}
\right\}
-
\max_{1\leq j\leq\numK}
\max_{\gammav\in \Grid_{j}(\varepsilon_{j})}
\left\{\frac{\gammav^{\T}\phivbar_{j}}{z_{j}}\right\}
\leq
\kappa,
\end{EQA}
and the inequality \eqref{ineq:chain_ac} continues as
\begin{EQA}
&&\nquad
\P\left(0\leq
\max_{1 \leq j\leq\numK}\left\{\|\phivbar_{j}\|z_{j}^{-1}-1\right\}\leq \kappa
\right)
\\&\leq&
\P\left(\left|
\max_{1\leq j\leq\numK}
\sup_{\gammav\in \Grid_{j}(\varepsilon_{j})}
\left\{\frac{\gammav^{\T}\phivbar_{j}}{z_{j}}\right\}
-1
\right|\leq \kappa
\right).
\end{EQA}
The random values \(\gammav^{\T}\phivbar_{j}z_{j}^{-1}\sim \mathcal{N}(0,z_{j}^{-2}\Var\{\gammav^{\T}\phivbar_{j}\})\).
The anti-concentration inequality by \cite{chernozhukov2012comparison} for the maximum of a centered high-dimensional Gaussian vector (see Theorem \ref{theorem:anti_conc_ChChK} below), applied to \(\max_{1\leq j\leq\numK}
\sup_{\gammav\in \Grid_{j}(\varepsilon_{j})}
\left\{\gammav^{\T}\phivbar_{j}z_{j}^{-1}\right\}
\), 
implies
\begin{EQA}
&&\nquad\nquad
\P\left(\left|
\max_{1\leq j\leq\numK}
\sup_{\gammav\in \Grid_{j}(\varepsilon_{j})}
\left\{\frac{\gammav^{\T}\phivbar_{j}}{z_{j}}\right\}
-1
\right|\leq \kappa
\right)
\\&\leq&
\Deltaac
\eqdef
\CONSTac \kappa\sqrt{1 \vee
\log\left(\kappa^{-1}\sum\nolimits_{j=1}^{\numK}\left\{2/\varepsilon_{j}\right\}^{\dimp_{j}}\right)
},
\label{def:Deltaac}
\end{EQA}
where the constant \(\CONSTac\) depends on \(\min\) and \(\max\) of \(\Var\{\gammav^{\T}\phivbar_{j}z_{j}^{-1}\}\leq\E\|\phivbar_{j}\|^{2}z_{j}^{-2}\leq 1\); the sum \(\sum\nolimits_{j=1}^{\numK}\left\{2/\varepsilon_{j}\right\}^{\dimp_{j}}\) is proportional to cardinality of the set \(\{\gammav^{\T}\phivbar_{j}z_{j}^{-1},\, \gammav\in \Grid_{j}(\varepsilon_{j}),\, j=1,\dots,\numK\}\).
If one takes \(\varepsilon_{j}=2\CONST\left\{\Delta_{j}/(2z_{j})\right\}^{\frac{\dimp_{\min}+1}{\dimp_{j}+1}}\), then 
\eqref{cond:epsilon_Delta} holds with exponentially high probability due to Gaussianity of the vectors \(\phivbar_{j}\) and Theorem  1.2 in \cite{SPS2011}, hence
\begin{EQA}
\Deltaac &\leq&
\CONSTac \kappa\sqrt{1 \vee
\CONST\log\left(\frac{1}{2}\sum\nolimits_{j=1}^{\numK}\left\{2/\varepsilon_{j}\right\}^{\dimp_{j}+1}\right)
}
\\&\leq&
\CONSTac\left\{ \kappa
\sqrt{1 \vee \log(\numK/2)}+\CONST\max_{1\leq j\leq\numK}\{\Delta_{j}\}\sqrt{\max_{1\leq j\leq\numK}\log(2 z_{j}/\Delta_{j})}
\right\}.
\label{bound:Deltaac}
\end{EQA}
\end{proof}

\begin{theorem}[Anti-concentration inequality for maxima of a Gaussian random vector, \cite{chernozhukov2012comparison}]
\label{theorem:anti_conc_ChChK}
Let \(\left(X_{1},\dots,X_{p}\right)^\T\) be a centered Gaussian random vector with \(\sigma_{j}^{2}\eqdef \E X_{j}^{2}>0\) for all \(1\leq j\leq p\). Let \(\sigmau\eqdef \min_{1\leq j\leq p}\sigma_{j}\), \(\sigmao\eqdef \max_{1\leq j\leq p}\sigma_{j}\). Then for every \(\epsilon>0\)
\begin{EQA}
\sup_{x\in\R}\P\left(\bigl|\max_{1\leq j\leq p}X_{j}-x\bigr|\leq\epsilon\right) &\leq&
\CONSTac\epsilon\sqrt{1\vee\log(p/\epsilon)},
\end{EQA}
where \(\CONSTac\) depends only on \(\sigmau\) and \(\sigmao\). When the variances are all equal, namely \(\sigmau=\sigmao=\sigma\), \(\log(p/\epsilon)\) on the right side can be replaced by \(\log{p}\).
\end{theorem}

\subsection{Proof of Proposition \ref{thm:Cum_norms}}
\begin{proof}[Proof of Proposition \ref{thm:Cum_norms}]
Let \(\Phi\eqdef \left(\phiv_{1}^{\T},\dots,\phiv_{\numK}^{\T}\right)^{\T}\in \R^{\dimtotal}\) for \(\dimtotal\eqdef\dimp_{1}+\dots+\dimp_{\numK}\) (as in \eqref{def:Phi}), and similarly \(\Psi\eqdef \left(\psiv_{1}^{\T},\dots,\psiv_{\numK}^{\T}\right)^{\T}\in \R^{\dimtotal}\). Let also \(\Phibar\sim \mathcal{N}(0,\Var\Phi)\) and \(\Psibar\sim \mathcal{N}(0,\Var\Psi)\).
Introduce the following value, which comes from Lemma \ref{lemma:GaussCOmpar_Slepian} on Gaussian comparison:
\begin{EQA}
\delta_{2}(\Delta,\beta)&\eqdef& \CONST_{2}(\Delta,\beta)\max_{1\leq j\leq \numK} \sup_{t\in[0,1]}\left\{\E\|\phivbar_{j}\sqrt{t}+\psivbar_{j}\sqrt{1-t}\|^{2}\right\}
\\&\leq&
\CONST_{2}(\Delta,\beta)\max_{1\leq j\leq \numK} \max\left\{ \tr\Var(\phivbar_{j}), \tr\Var(\psivbar_{j})\right\}.
\label{def:delta2}
\end{EQA}

It holds
\begin{EQA}
&&\hspace{-2cm}
\P\left(\bigcup\nolimits_{j=1}^{\numK}\left\{\|\phiv_{j}\|>z_{j}\right\} \right)
\\&\overset{\text{by L.\,\ref{lemma:smooth_techn_gauss_appr_phiv}}}{\geq}&
\E H_{\Delta,\,\beta}\left(\Phibar,\zv+\frac{3\log(K)}{2\beta}\Idv_{\numK}\right)
-\delta_{3,\phi}(\Delta,\beta)
\\&\overset{\text{by L.\,\ref{lemma:GaussCOmpar_Slepian},\,\ref{lemma:prop_F_beta_Delta}}}{\geq}&
\E H_{\Delta,\,\beta}\left(\Psibar,\zv+\frac{3\log(K)}{2\beta}\Idv_{\numK}\right)
-\frac{1}{2}\delta_{\Sigma}^{2}\delta_{2}(\Delta,\beta)
-\delta_{3,\phi}(\Delta,\beta)
\\&\overset{\text{by L.\,\ref{lemma:smooth_techn_gauss_appr_phiv}}}{\geq}&
\P\left(\bigcup\nolimits_{j=1}^{\numK}\left\{\|\psivbar_{j}\|>z_{j}+\Delta+\frac{3\log(K)}{2\beta}\right\} \right)
-\frac{1}{2}\delta_{\Sigma}^{2}\delta_{2}(\Delta,\beta)
-\delta_{3,\phi}(\Delta,\beta)
\\&\overset{\text{by L.\,\ref{lemma:anti_conc_nonsmooth}}}{\geq}&
\P\left(\bigcup\nolimits_{j=1}^{\numK}\left\{\|\psivbar_{j}\|>z_{j}-\delta_{z_{j}}-\Delta\right\} \right)
-\frac{1}{2}\delta_{\Sigma}^{2}\delta_{2}(\Delta,\beta)-\delta_{3,\phi}(\Delta,\beta)
\\&&
-\,2\Deltaac\left(\left\{\delta_{z_{j}}\right\}+2\Delta+\frac{3\log(K)}{\beta}\right)
\label{ineq:ac_use}
\\&\overset{\text{by L.\,\ref{lemma:gauss_appr_phiv}}}{\geq}&
\label{ineq:Delta_minus}
\P\left(\bigcup\nolimits_{j=1}^{\numK}\left\{\|\psiv_{j}\|>z_{j}-\delta_{z_{j}}\right\} \right)
-\frac{1}{2}\delta_{\Sigma}^{2}\delta_{2}(\Delta,\beta)\\&&
-\,\delta_{3,\phi}(\Delta,\beta)-\delta_{3,\psi}(\Delta,\beta)
-2\Deltaac\left(\left\{\delta_{z_{j}}\right\}+2\Delta+\frac{3\log(K)}{\beta}\right),
\end{EQA}
where \(\delta_{3,\psi}(\Delta,\beta)\) is defined similarly to \(\delta_{3,\phi}(\Delta,\beta)\) in \eqref{ineq:Efb_appr}:
\begin{EQA}
\label{def:delta3psi}
\delta_{3,\psi}(\Delta,\beta)&\eqdef&
\frac{\CONST_{3}(\Delta,\beta)}{3}\frac{\dimp_{\max}^{3/2}}{n^{1/2}}
\log^{1/2}(\numK)
\log^{3/2}\left(n\dimtotal\right)
 \left(2\nunu^{2}\cf_{\psi}^{2}\lambda_{\psi,\max}^{2}\right)^{3/2}.
\end{EQA}
By Lemma \ref{lemma:anti_conc_nonsmooth} inequality \eqref{ineq:ac_use} requires the following:
\(\delta_{z_{j}}+2\Delta+\frac{3\log(K)}{\beta}\leq z_{j}\).
The bound in the inverse direction is derived similarly.
Denote the approximating error term obtained in \eqref{ineq:Delta_minus} as
\begin{EQA}
\label{def:Deltapp}
\hspace{-1.2cm}
\Deltapp&\eqdef&
\frac{1}{2}\delta_{\Sigma}^{2}\delta_{2}(\Delta,\beta)+\delta_{3,\phi}(\Delta,\beta)+\delta_{3,\psi}(\Delta,\beta) +2\Deltaac\left(\left\{\delta_{z_{j}}\right\}+2\Delta+\frac{3\log(K)}{\beta}\right).
\end{EQA}
Consider this term in more details, by inequality \eqref{bound:Deltaac}
\begin{EQA}
&&
\nquad
\Deltaac\left(\left\{\delta_{z_{j}}\right\}+2\Delta+\frac{3\log(K)}{\beta}\right)
\leq \max_{1\leq j\leq\numK}
\left(\delta_{z_{j}}+2\Delta+\frac{3\log(K)}{\beta}\right)
\\&&\times\,\left\{\CONST\frac{\log^{1/2}(\numK)}{z_{j}}+\log^{1/2}\left(2 z_{\max}\right)
-\log^{1/2}\left(\delta_{z_{j}}+2\Delta+\frac{3\log(K)}{\beta}\right)
\right\}.
\end{EQA}
Let us take \(\beta=\frac{\log(\numK)}{\Delta}\),
 then
\begin{EQA}
\Deltaac&\leq&
5\CONST\Delta\frac{\log^{1/2}(\numK)}{z_{\min}}
+\CONST\max_{1\leq j\leq\numK}\frac{\delta_{z_{j}}}{z_{j}}\log^{1/2}(\numK)
\\&&+\,
\CONST\left(5\Delta+\delta_{z,\max}\right)\left(\log^{1/2}\left(2z_{\max}\right)
+\sqrt{-\log\left(\delta_{z,\min}+5\Delta\right)}\right)
,
\\
&\leq&
5\CONST\Delta\frac{\log^{1/2}(\numK)}{z_{\min}}
+\CONST\max_{1\leq j\leq\numK}\frac{\delta_{z_{j}}}{z_{j}}\log^{1/2}(\numK)
\\&&+\,
2\CONST\left(5\Delta+\delta_{z,\max}\right)
\sqrt{-\log\left(\delta_{z,\min}+5\Delta\right)}
\\
&\leq&
5\CONST\Delta\frac{\log^{1/2}(\numK)}{z_{\min}}
+\CONST\max_{1\leq j\leq\numK}\frac{\delta_{z_{j}}}{z_{j}}\log^{1/2}(\numK)
+
2\CONST\left(5\Delta+\delta_{z,\max}\right)
\sqrt{-\log\left(5\Delta\right)}
\\
&\leq&
5\CONST\Delta\Bigl\{\frac{\log^{1/2}(\numK)}{z_{\min}}
+2.4\log^{1/2}\bigl(5n^{1/2}\bigr)
\Bigr\}
+\CONST\max_{1\leq j\leq\numK}\frac{\delta_{z_{j}}}{z_{j}}\log^{1/2}(\numK)
\\
&\leq&
6\CONST\Delta\Bigl\{\frac{\log^{1/2}(\numK)}{z_{\min}}
+0.4\log^{1/2}\bigl(5n^{1/2}\bigr)
\Bigr\},
\label{ineq:summ1}
\end{EQA}
where the second inequality holds for \(\delta_{z,\min}+5\Delta\leq 1/(2z_{\max})\), and
the last one holds for \(\delta_{z,\max}\leq \Delta\) and \(\Delta\geq n^{-1/2}\). 
\begin{EQA}
\label{ineq:summ2}
\hspace{-1cm}
\delta_{3,\phi}(\Delta,\beta)+\delta_{3,\psi}(\Delta,\beta) &\overset{{\small\text{by\,\eqref{def:delta3psi}}}}{\leq}&
\CONST\frac{\log^{5/2}(\numK)}{\Delta^{3}}
\frac{\dimp_{\max}^{3/2}}{n^{1/2}}
\log^{3/2}(n\dimtotal)
\left(\lambda_{\phi,\max}^{3}+\lambda_{\psi,\max}^{3}\right)
,
\\
\delta_{\Sigma}\delta_{2}(\Delta,\beta)
&\overset{{\small\text{by\,\eqref{def:delta2}}}}{\leq}&
\CONST\delta_{\Sigma}^{2}\frac{\log(\numK)}{\Delta^{2}}
\max_{1\leq j\leq \numK} \max\left\{ \tr\Var(\phivbar_{j}), \tr\Var(\psivbar_{j})\right\}
\\&\leq& \CONST\delta_{\Sigma}^{2}\frac{\log(\numK)}{\Delta^{2}}\dimp_{\max}\max\left\{\lambda_{\phi,\max}^{2},\lambda_{\psi,\max}^{2}\right\}.
\end{EQA}
After minimizing the sum of the expressions \eqref{ineq:summ1} and  \eqref{ineq:summ2} w.r.t \(\Delta\), we have
\begin{EQA}
\Deltapp
&\leq&
12.5\CONST\left(\frac{\dimp_{\max}^{3}}{n}\right)^{1/8}\log^{9/8}(K)\log^{3/8}(n\dimtotal) \max\left\{\lambda_{\phi,\max},\lambda_{\psi,\max}\right\}^{3/4}
\\&&+\,3.2\CONST\delta_{\Sigma}^{2}\dimp_{\max}z_{\min}^{1/2}\left(\frac{\dimp_{\max}^{3}}{n}\right)^{1/4}
\log^{2}(K)\log^{3/4}(n\dimtotal)\max\left\{\lambda_{\phi,\max},\lambda_{\psi,\max}\right\}^{7/2}
\\&\leq&
25\CONST\left(\frac{\dimp_{\max}^{3}}{n}\right)^{1/8}\log^{9/8}(K)\log^{3/8}(n\dimtotal) \max\left\{\lambda_{\phi,\max},\lambda_{\psi,\max}\right\}^{3/4},
\end{EQA}
where the last inequality holds for
\begin{EQA}
\delta_{\Sigma}^{2}
&\leq&4 \CONST \dimp_{\max}^{-1}z_{\min}^{-1/2}\left(\frac{\dimp_{\max}^{3}}{n}\right)^{-1/8}
\log^{-7/8}(K)\log^{-3/8}(n\dimtotal)\left(\max\left\{\lambda_{\phi,\max},\lambda_{\psi,\max}\right\}\right)^{-11/4}.
\end{EQA}
\end{proof}

\section{Square-root Wilks approximations}
\label{sect:sqrootWilks}
This section's goal is to derive square root Wilks approximations  simultaneously for \(\numK\) parametric models,  for the \(\Ym\) and bootstrap worlds. This is done in Section \ref{sect:unif_Wilks} below.  Both of the results are used in the approximating scheme \eqref{rectangle_simult} for the bootstrap justification. In order to make the text self-contained we recall in Section \ref{sect:FS} some results from the general finite sample theory by \cite{Sp2012Pa,SPS2011,Spokoiny2013Bernstein}. In Section \ref{sect:FStheory} we recall similar finite sample results for the bootstrap world for a single parametric model, obtained in \cite{SpZh2014PMB}.

\subsection{Finite sample theory}
\label{sect:FS}
Let us use the notations given in the introduction: \(L_{k}(\thetav)\), \(k=1,\dots, \numK\) are the log-likelihood processes, which depend on the data \(\Ym\)  and correspond to the regular parametric families of probability distributions \(\{\P_{k}(\thetav),\thetav\in\Theta_{k}\subset\R^{\dimp_{k}}\}\). The general finite sample approach by \cite{Sp2012Pa} does not require that the true distribution \(\P\) of the data \(\Ym\) belongs to any of the parametric families \(\{\P_{k}(\thetav)\}\).
The target parameters \(\thetavs_{k}\) are defined as in \eqref{def:thetavs_k} by projection of the true measure \(\P\) on  \(\{\P_{k}(\thetav)\}\).
Let \(D^{2}_{k}\) denote the full Fisher information \(\dimp_{k}\times\dimp_{k}\) matrices, which are deterministic, symmetric and positive-definite:
\begin{EQA}[c]
D^{2}_{k}\eqdef -\nabla_{\thetav}^{2}\E L_{k}(\thetavs_{k}).
\end{EQA}
Centered \(\dimp_{k}\)-dimensional random vectors \(\xiv_{k}\) denote the normalised scores:
\begin{EQA}[c]
\xiv_{k}\eqdef D_{k}^{-1}\nabla_{\thetav}L_{k}(\thetavs_{k}).
\end{EQA}
Introduce the following elliptic vicinities around the true points \(\thetavs_{k}\):
\begin{EQA}[c]
\label{def:Thetas}
\Theta_{0,k}(\rr)\eqdef\left\{\thetav\in \Theta_{k}: \|D_{k}(\thetav-\thetavs_{k})\|\leq \rr\right\}.
\end{EQA}
Let \(1\leq k\leq \numK\) be fixed. The non-asymptotic Wilks approximating bound by \cite{Sp2012Pa,Spokoiny2013Bernstein} requires that the  maximum likelihood estimate \(\thetavt_{k}\) gets into the local vicinity \(\Theta_{0,k}(\rr_{0,k})\) of some radius \(\rr_{0,k}>0\) with probability \(\geq 1-3\ex^{-\yy}\), \(\yy>0\). This is guaranteed by the following concentration result:
\begin{theorem}[Concentration of the MLE, \cite{Spokoiny2013Bernstein}]
\label{thm:concentr}
Let the conditions \ref{itm:ED0k}, \ref{itm:ED2k}, \ref{itm:L0k}, \ref{itm:Ik} and \ref{itm:Lrk} be fulfilled. If for each \(k=1,\dots,\numK\) for the constants \(\rr_{0,k}>0\) and for the functions \(\gmi_{k}(\rr)\) from \ref{itm:Lrk} holds:
\begin{EQA}[c]
\label{condit:concentr}
\gmi_{k}(\rr)\rr\geq2\left\{\ZZqf(\yy,\BB_{k})+6\omega_{k}\nu_{k}\ZZ_{k}(\yy+\log(2\rr/\rr_{0,k}))
\right\},\quad \rr>\rr_{0,k}
\end{EQA}
where the functions \(\ZZ_{k}(\yy)\) and \(\ZZqf(\yy,\BB_{k})\) are defined in \eqref{ZZ_k} and \eqref{ZZqf_k} respectively,
 then it holds for all \(k=1,\dots,\numK\)
\begin{EQA}[c]
\P\left(
\thetavt_{k}\notin \Theta_{0,k}(\rr_{0,k})
\right)\leq 3\ex^{-\yy}.
\end{EQA}
The constants \(\omega_{k}, \nu_{k}\) and \(\gmu_{k}\) come from the imposed conditions  \ref{itm:ED0k}--\ref{itm:Ik} (from Section \ref{sect:conditions}). In the case  \ref{typical_local} \(\rr_{0,k}\geq \CONST\sqrt{\dimp_{k}+\yy}\).
\end{theorem}
\begin{theorem}[Wilks approximation, \cite{Spokoiny2013Bernstein}]
\label{Wilks}
Under the conditions of Theorem \ref{thm:concentr} for
some \(\rr_{0,k}>0\)
 s.t. \eqref{condit:concentr} is fulfilled, it holds for each \(k=1,\dots,\numK\)
with probability \(\geq 1- 5\ex^{-\yy}\)
\begin{EQA}
\left|
2\Bigl\{L_{k}(\thetavt_{k})-L_{k}(\thetavs_{k})\Bigr\}-\|\xiv_{k}\|^{2}\right|&\leq&
\DeltakWsq(\rr_{0,k},\yy),\\
\Bigl|
\sqrt{ 2\Bigl\{L_{k}(\thetavt_{k})-L_{k}(\thetavs_{k})\Bigr\}}-\|\xiv_{k}\|\Bigr|&\leq&
\DeltakW(\rr_{0,k},\yy)
\end{EQA}
for
\begin{EQA}
\label{def:diamondk}
\DeltakW(\rr,\yy)&\eqdef&
3\rr\left\{
\delta(\rr)+6\nu_{k}\ZZ_{k}(\yy)\omega_{k}
\right\},\\
\label{def:diamondsqk}
\DeltakWsq(\rr,\yy)&\eqdef&\frac{2}{3}\left\{2\rr+\ZZqf(\yy,\BB_{k})\right\}\DeltakW(\rr,\yy),\\
\ZZ_{k}(\yy)&\eqdef& 2\sqrt{\dimp_{k}}+\sqrt{2\yy}+4\dimp_{k}(\yy\gm_{k}^{-2}+1)\gm_{k}^{-1}.
\label{ZZ_k}
\end{EQA}
In the case \ref{typical_local} it holds for \(\rr\leq\rr_{0,k}\):
\begin{EQA}[c]
\DeltakW(\rr,\yy)\leq\CONST\, \frac{\dimp_{k}+\yy}{\sqrt{n}},\quad
\DeltakWsq(\rr,\yy)\leq\CONST\sqrt{\frac{(\dimp_{k}+\yy)^{3}}{{n}}}.
\end{EQA}
The constants \(\gm_{k}\) and \(\delta_{k}(\rr)\) come from the imposed conditions \ref{itm:ED0k}, \ref{itm:L0k} (from Section \ref{sect:conditions}). The function \(\ZZqf(\yy,\BB_{k})\), defined in \eqref{ZZqf_k}, corresponds to the quantile function of deviations of the approximating random value \(\|\xiv_{k}\|\) (see Theorem \ref{qf_dev} below).
\end{theorem}

The following theorem characterizes the tail behaviour of the approximating terms \(\|\xiv_{k}\|^{2}\). It means that with bounded exponential moments  of the vectors \(\xiv_{k}\) (conditions \ref{itm:ED0k}, \ref{itm:Ik}) its squared Euclidean norms \(\|\xiv_{k}\|^{2}\) have three regimes of deviations: sub-Gaussian, Poissonian and large-deviations' zone.
\begin{theorem}[Deviation bound for a random quadratic form, \cite{SPS2011}]
\label{qf_dev}
Let condition \ref{itm:ED0k} be fulfilled, then for \(\gm_{k}\geq \sqrt{2\tr(\BB_{k}^{2})}\) it holds for each \(k=1,\dots,\numK\):
\begin{EQA}[c]
\P\left(
\|\xiv_{k}\|^{2}\geq \ZZqf^{2}(\yy,\BB_{k})
\right) \leq 2\ex^{-\yy}+8.4\ex^{-\yy_{c,k}},
\end{EQA}
where \(\BB_{k}^{2}\eqdef D_{k}^{-1}V_{k}^{2}D_{k}^{-1}\), \(\lambda(\BB_{k})\) is a maximum eigenvalue of \(\BB_{k}^{2}\),
\begin{EQA}
\hspace{-1cm}\ZZqf^{2}(\yy,\BB_{k})&\eqdef&
\begin{cases}
\label{ZZqf_k}
\tr(\BB_{k}^{2})+\sqrt{8\tr(\BB_{k}^{4})\yy},&
\hspace{-1.3cm}
 \yy\leq\sqrt{2\tr(\BB_{k}^{4})}/\{18\lambda(\BB_{k})\},\\
\tr(\BB_{k}^{2})+6\yy\lambda(\BB_{k}),& 
\hspace{-1.3cm}
\sqrt{2\tr(\BB_{k}^{4})}/\{18\lambda(\BB_{k})\}<\yy\leq\yy_{c,k},\\
\left|\mathtt{z}_{c,k}+2(\yy-\yy_{c,k})/\gm_{c,k}\right|^{2}\lambda(\BB_{k})
,& 
\yy>\yy_{c,k},
\end{cases}
\end{EQA}
\begin{EQA}
\label{yyc}
2\yy_{c,k}&\eqdef&2\yy_{c,k}(\BB_{k})\eqdef \mu_{c}\mathtt{z}_{c,k}^{2}
+\log{\det\left(
\Id_{\dimp_{k}}-\mu_{c}\BB_{k}^{2}/\lambda(\BB_{k})
\right)}
,\\
\mathtt{z}_{c,k}^{2}&\eqdef&\left\{\gm_{k}^{2}/\mu_{c}^{2}-\tr{(\BB_{k}^{2})}/\mu_{c}\right\}/\lambda(\BB_{k})
,\\
\gm_{c,k}&\eqdef&\sqrt{\gm_{k}^{2}-\mu_{c}\tr{(\BB_{k}^{2})}}/\sqrt{\lambda(\BB_{k})} ,\\
\mu_{c}&\eqdef&2/3.
\end{EQA}
The matrices \(V_{k}^{2}\) come from condition \ref{itm:ED0k} and can be defined as
\begin{EQA}[c]
\label{def:VPc}
V_{k}^{2}\eqdef \Var\left\{\nabla_{\thetav}L_{k}(\thetavs_{k})\right\}.
\end{EQA}
By condition \ref{itm:Ik} \(\tr(\BB_{k}^{2})\leq \gmu_{k}^{2}\dimp_{k}\), \(\tr(\BB^{4})\leq\gmu_{k}^{4}\dimp_{k} \) and \(\lambda(\BB_{k})\leq\gmu_{k}^{2}\). In the case \ref{typical_local} \(\gm_{k}=\CONST \sqrt{n}\), hence \(\yy_{c,k}= \CONST {n}\), and for \(\yy\leq \yy_{c,k}\) it holds:
\begin{EQA}[c]
\label{ZZqf_short}
\ZZqf^{2}(\yy,\BB_{k})\leq \gmu_{k}^{2}(\dimp_{k} + 6\yy).
\end{EQA}
\end{theorem}

\subsection{Finite sample theory for the bootstrap world}
\label{sect:FStheory}
 Introduce for each \(k=1,\dots,\numK\) the bootstrap score vectors at the point \(\thetav\in\Theta_{k}\):
\begin{EQA}
\label{def:xivb}
\xivb_{k}(\thetav)&\eqdef&D_{k}^{-1}\nabla_{\thetav}\zetavb_{k}(\thetav)\\
&=&\sum_{i=1}^{n}D_{k}^{-1}\nabla_{\thetav}\ell_{i,k}(\thetav)(u_{i}-1).
\end{EQA}
\begin{theorem}[Bootstrap Wilks approximation,
 \cite{SpZh2014PMB}]
\label{Wilks_boots}
Under the conditions of Theorems \ref{thm:concentr} and \ref{thm:large_dev} 
for each \(k=1,\dots,\numK\) and 
some \(\rr_{0,k}^{2}\geq0\) s.t. \eqref{condit:concentr} and \eqref{condit:large_dev} are fulfilled, it holds for each \(k\) with \(\P\)-probability \(\geq 1- 5\ex^{-\yy}\)
\begin{EQA}[c]
\Pb\left(
\Bigl|\sup_{\thetav\in\Theta_{k}}2\left\{\Lb_{k}(\thetav)-
\Lb_{k}(\thetavt_{k})\right\}
-\|\xivb_{k}(\thetavt_{k})\|^{2}
\Bigr|
\leq \DeltabkWsq(\rr_{0,k},\yy)
\right)\geq1-4\ex^{-\yy},\\
\Pb\left(
\Bigl|
\sqrt{\sup_{\thetav\in\Theta_{k}}2\left\{\Lb_{k}(\thetav)-\Lb_{k}(\thetavb_{k})\right\}}-\|\xivb_{k}(\thetavb_{k})\|\Bigr|\leq
\DeltabkW(\rr_{0,k},\yy)
\right)\geq1-4\ex^{-\yy}.
\end{EQA}
where the error terms \(\DeltabkW(\rr,\yy), \DeltabkWsq(\rr,\yy)\) are deterministic and
\begin{align}
\label{def:DeltabkW}
\begin{split}
\DeltabkW(\rr,\yy)&\eqdef2\DeltakW(\rr,\yy)+36\nu_{k}\rr\omegabk(\rr,\yy)\ZZ_{k}(\yy),\\
\DeltabkWsq(\rr,\yy)&\eqdef
\frac{1}{18}\left\{12\rr\DeltabkW(\rr,\yy)+\DeltabkW(\rr,\yy)^{2}\right\}.
\end{split}
\end{align}
\(\DeltakW(\rr,\yy)\) and \(\ZZ_{k}(\yy)\) are defined in \eqref{def:diamondk} and \eqref{ZZ_k} respectively and 
 \begin{EQA}[c]
 \label{def:omegab}
\omegabk(\rr,\yy)=\omegabk\eqdef
\frac{\CONST_{m,k}(\rr)}{\sqrt{n}}+ 2\omega_{k}\nu_{k}\sqrt{2\yy},
\end{EQA}
where \(\CONST_{m,k}(\rr), \omega_{k}, \nu_{k}\) come from the imposed conditions \ref{itm:L0mk}, \ref{itm:ED2k} and \ref{itm:ED0k}. For the case \ref{typical_local} and \(\rr\leq\rr_{0,k}\) it holds:
\begin{EQA}[c]
\DeltabkW(\rr,\yy)\leq\CONST\, \frac{\dimp_{k}+\yy}{\sqrt{n}}\sqrt{\yy},\quad
\DeltabkWsq(\rr,\yy)\leq\CONST\sqrt{\frac{(\dimp_{k}+\yy)^{3}}{{n}}}\sqrt{\yy}.
\end{EQA}
 and \(\omegabk(\rr)\leq \CONST \rr/n + \CONST\sqrt{\yy/n}\).

\end{theorem}

\begin{theorem}[Concentration of the bootstrap MLE, \cite{SpZh2014PMB}]
\label{thm:large_dev}
Let the conditions of Theorems \ref{thm:concentr} and \ref{qfb_dev}, \ref{itm:L0mk} and \ref{itm:ED2mk}
be fulfilled.
If the following holds for each \(k=1,\dots,\numK\), \(\omegabk(\rr,\yy)\) defined in \eqref{def:omegab} and the \(\P\)-random matrices  \(\Bb_{k}^{2}\eqdef D_{k}^{-1}\Varb\left\{\nabla_{\thetav}\Lb_{k}(\thetavs_{k})\right\}D_{k}^{-1}\): 
\begin{EQA}
\label{condit:large_dev}
\gmi_{k}(\rr)\rr&\geq& 2
\left\{\ZZqf(\yy,\BB_{k})+\ZZqf(\yy,\Bb_{k})+6\nu_{k}\ZZb_{k}(\yy)\omegabk(\rr_{0,k})\rr_{0,k}\right\}\\
&&+\,12\nu_{k}(\omega_{k}+\omegabk(\rr,\yy))\ZZ_{k}(\yy+\log(2\rr/\rr_{0,k}))
\quad\text{for }~ \rr>\rr_{0,k},
\end{EQA}
then for each \(k\) it holds with \(\P\)-probability \(\geq 1-3\ex^{-\yy}\)
\begin{EQA}[c]
\Pb\left(
\thetavbt_{k}\notin \Theta_{0,k}(\rr_{0,k})
\right)\leq 3\ex^{-\yy}.
\end{EQA}
\end{theorem}
Lemma \ref{lemma_xivb_unif} below is implied straightforwardly by Lemma B.7 in \cite{SpZh2014PMB}.
\begin{lemma}
\label{lemma_xivb_unif}
Let the conditions of \ref{itm:Eb}, \ref{itm:L0mk} and \ref{itm:ED2mk} be fulfilled, then for each \(k=1,\dots,\numK\) it holds for \(\rr\leq\rr_{0,k}\) with \(\P\)-probability \(\geq 1-\ex^{-\yy}\)
\begin{EQA}[c]
\Pb\left(
\sup_{\thetav\in\Theta_{0,k}(\rr)}
\|\xivb_{k}(\thetav)-\xivb_{k}(\thetavs_{k})\| \leq
\Deltab_{\xi,k}(\rr,\yy)
\right) \geq 1-\ex^{-\yy},
\end{EQA}
where
\begin{EQA}[c]
\label{def:Deltaxik}
\Deltab_{\xi,k}(\rr,\yy) \eqdef
6\nu_{k}\ZZb_{k}(\yy)\omegabk(\rr,\yy)\rr
\end{EQA}
In the case \ref{typical_local}
it holds for the bounding term
\begin{EQA}[c]
\Deltab_{\xi}(\rups,\yy)\leq \CONST\frac{\dimp_{k}+\yy}{\sqrt{n}}\sqrt{\yy}.
\end{EQA}
\end{lemma}
\begin{theorem}[Deviation bound for the bootstrap quadratic form, \cite{SpZh2014PMB}]
\label{qfb_dev}
Let conditions \ref{itm:Eb}, \ref{itm:Ik}, \ref{itm:SD0k}, \ref{itm:IBk} be fulfilled, then for each \(k=1,\dots,\numK\) and \(\gm_{k}\geq \sqrt{2\tr(\Bb_{k}^{2})}\) it holds:
\begin{EQA}[c]
\Pb\left(
\|\xivb_{k}(\thetavs_{k})\|^{2}\leq \ZZqf^{2}(\yy,\Bb_{k})
\right) \geq 1-2\ex^{-\yy}-8.4\ex^{-\yy_{c,k}(\Bb_{k})},
\end{EQA}
where 
\begin{EQA}[c]
\label{Bb}
\Bb_{k}^{2}\eqdef D_{k}^{-1}\Vr^{2}(\thetavs_{k})D_{k}^{-1},\quad
\Vr_{k}^{2}(\thetavs_{k})\eqdef \Varb\nabla_{\thetav}\Lb_{k}(\thetavs_{k}),
\label{def:Vr}
\end{EQA}
\(\ZZqf(\yy,\cdot)\) and \(\yy_{c,k}(\cdot)\) are defined respectively in \eqref{ZZqf_k} and \eqref{yyc}.
Similarly to \eqref{ZZqf_short} it holds for \(\yy\leq \yy_{c,k}(\Bb_{k})\):
\begin{EQA}
\label{ZZqfb_short}
\ZZqf^{2}(\yy,\Bb_{k})&\leq& {\gmub_{k}}^{2}(\dimp_{k} + 6\yy)\\
\text{for }~ {\gmub_{k}}^{2}&\eqdef& (1+\delta_{\Vr,k}^{2}(\xx))(\gmu_{k}^{2}+\gmu_{B,k}^{2})
\label{def:gmub}
\end{EQA}
and \(\delta_{\Vr,k}^{2}(\xx)\) defined in \eqref{def:deltaVr}  (see Section \ref{sect:NCBI_k} on Bernstein matrix inequalities).
\end{theorem}

\subsection{Simultaneous square-root Wilks approximations}
\label{sect:unif_Wilks}
The statements below follow from the results from Sections \ref{sect:FS} and \ref{sect:FStheory} by probability union bound.
\begin{lemma}[Simultaneous concentration bounds]
\label{lemma:simultconc_gen}
\mbox{}
\begin{itemize}
\item[1.]
Let conditions of Theorem \ref{thm:concentr} be fulfilled and \eqref{condit:concentr} hold for each \(k=1,\dots,\numK\) with \(\yy=\yy_{1}+\log(\numK)\) for some \(\yy_{1}>0\), then
\begin{EQA}[c]
\P\left(\bigcup\nolimits_{k=1}^{\numK}\left\{
\thetavt_{k}\notin \Theta_{0,k}(\rr_{0,k})
\right\}\right)\leq 3\ex^{-\yy_{1}}.
\end{EQA}
\item[2.] 
Let conditions of Theorem \ref{thm:large_dev} be fulfilled and \eqref{condit:large_dev} hold for each \(k=1,\dots,\numK\) with \(\yy=\yy_{1}+\log(\numK)\) for some \(\yy_{1}>0\), then it holds with \(\P\)-probability \(\geq 1-3\ex^{-\yy_{1}}\)
\begin{EQA}[c]
\Pb\left(\bigcup\nolimits_{k=1}^{\numK}\left\{
\thetavbt_{k}\notin \Theta_{0,k}(\rr_{0,k})
\right\}\right)\leq 3\ex^{-\yy_{1}}.
\end{EQA}
\end{itemize}
\end{lemma}
\begin{lemma}[Simultaneous Wilks approximations] 
\label{lemma:simultWilks_gen}
\mbox{}
\begin{itemize}
\item[1.]
Let the conditions of part 1 of Lemma \ref{lemma:simultconc_gen} be fulfilled for some \(\rr_{0,k}>0\) and \(\yy=\yy_{1}+\log(\numK)\), then it holds
\begin{EQA}
\P\left(\bigcap\nolimits_{k=1}^{\numK}\left\{\bigl|
2\Bigl\{L_{k}(\thetavt_{k})-L_{k}(\thetavs_{k})\Bigr\}-\|\xiv_{k}\|^{2}\bigr|\leq
\DeltakWsq(\rr_{0,k},\yy_{1}+\log(\numK))
\right\}\right)&\geq& 1-5\ex^{-\yy_{1}},\\
\P\left(\bigcap\nolimits_{k=1}^{\numK}\left\{
\Bigl|
\sqrt{ 2\Bigl\{L_{k}(\thetavt_{k})-L_{k}(\thetavs_{k})\Bigr\}}-\|\xiv_{k}\|\Bigr|\leq
\DeltakW(\rr_{0,k},\yy_{1}+\log(\numK))\right\}\right)&\geq& 1-5\ex^{-\yy_{1}}.
\end{EQA}
\item[2.]
Let the conditions of  parts 1,2 of Lemma \ref{lemma:simultconc_gen} be fulfilled for some \(\rr_{0,k}>0\) and \(\yy=\yy_{1}+\log(\numK)\), then it holds with \(\P\)-probability \(\geq 1- 5\ex^{-\yy_{1}}\) 
\begin{EQA}
\hspace{-1cm}
\Pb\left(\bigcap\limits_{k=1}^{\numK}\left\{
\Bigl|\sup_{\thetav\in\Theta_{k}}2\Bigl\{\Lb_{k}(\thetav)-
\Lb_{k}(\thetavt_{k})\Bigr\}
-\|\xivb_{k}(\thetavt_{k})\|^{2}
\Bigr|\leq
\DeltabkWsq(\rr_{0,k},\yy_{1}+\log(\numK))
\right\}\right)&\geq& 1-4\ex^{-\yy_{1}},\\
\hspace{-1cm}
\Pb\left(\bigcap\limits_{k=1}^{\numK}\left\{
\Bigl|
\sqrt{\sup_{\thetav\in\Theta_{k}}2\Bigl\{\Lb_{k}(\thetav)-\Lb_{k}(\thetavb_{k})\Bigr\}}-\|\xivb_{k}(\thetavb_{k})\|\Bigr|\leq
\DeltabkW(\rr_{0,k},\yy_{1}+\log(\numK))\right\}\right)&\geq& 1-4\ex^{-\yy_{1}}.
\end{EQA}
\end{itemize}
\end{lemma}
\begin{lemma}
\label{lemma:xivb_simult_gen}
Let the conditions of Lemma \ref{lemma_xivb_unif} be fulfilled, then it holds with \(\P\)-probability \(\geq 1-\ex^{-\yy}\)
\begin{EQA}[c]
\Pb\left(\bigcap\nolimits_{k=1}^{\numK}\left\{
\sup_{\substack{\thetav\in\Theta_{0,k}(\rr),\\ \rr\leq \rr_{0,k}}}
\|\xivb_{k}(\thetav)-\xivb_{k}(\thetavs_{k})\| \leq
\Deltab_{\xi,k}(\rr,\yy+\log(\numK))
\right\}
\right) \geq 1-\ex^{-\yy}.
\end{EQA}
\end{lemma}


\section{Proofs of the main results}
\label{sect:proofs_all}
Before proving the statements from Section \ref{sect:mainres} we formulate below the Bernstein matrix inequality, which is necessary for the further proofs.
\subsection{Bernstein matrix inequality}
\label{sect:NCBI_k}
Here we restate the Theorem 1.4 by \cite{Tropp2012user} for the random \(\dimtotal\times\dimtotal\) matrix \(\bbVb^{2}\eqdef\Varb \bigl(\nabla_{\thetav}\Lb_{1}(\thetavs_{1})^{\T},\dots,\nabla_{\thetav}\Lb_{\numK}(\thetavs_{\numK})^{\T}\bigr)^{\T}\) from the bootstrap world. 
Matrix \(\bbVb^{2}\) equals to the sum of independent matrices \(\Varb \bigl(\nabla_{\thetav}\ell_{i,1}(\thetavs_{1})^{\T}u_{i},\dots,\nabla_{\thetav}\ell_{i,\numK}(\thetavs_{\numK})^{\T}u_{i}\bigr)^{\T}\). Let us denote
\begin{EQA}
\bbg_{i}&\eqdef& \left(\nabla_{\thetav}\ell_{i,1}(\thetavs_{1})^{\T},\dots,\nabla_{\thetav}\ell_{i,\numK}(\thetavs_{\numK})^{\T}\right)^{\T}\in \R^{\dimtotal},\\
\bbH^{2}&\eqdef&\sum\nolimits_{i=1}^{n}\E \left\{\bbg_{i}\bbg_{i}^{\T}\right\},
\\
\bbv_{i}&\eqdef& \bbH^{-1}\left\{\bbg_{i}\bbg_{i}^{\T}- \E\left[\bbg_{i}\bbg_{i}^{\T}\right]\right\}\bbH^{-1},
\end{EQA}
 then 
 \begin{EQA}[c]
 \bbH^{2}=\E\bbVb^{2}, \quad\quad
\sum\nolimits_{i=1}^{n}\bbv_{i}^2=
 \bbH^{-1}\bbVb^{2}\bbH^{-1} - \Id_{\dimtotal}.
 \end{EQA}
Define also the deterministic scalar value
\begin{EQA}[c]
\barbar{\kappa}_{v}^{2}\eqdef\Bigl\|\sum\nolimits_{i=1}^{n}\E \bbv_{i}^4 \Bigr\|.
\end{EQA}
\begin{theorem}[Bernstein inequality for \(\bbVb^{2}\)]
\label{lemma_ncbi}
Let the condition \ref{itm:SD0k} be fulfilled,
then 
it holds with probability \(\geq 1-\ex^{-\yy}\):
\begin{EQA}
\|\bbH^{-1}\bbVb^{2}\bbH^{-1}-\Id_{\dimtotal}\|&\leq& \deltabbVb^{2}(\yy),
\end{EQA}
where the error term is defined as
\begin{EQA}
\deltabbVb^{2}(\yy)&\eqdef& \sqrt{2\barbar{\kappa}_{v}^{2}\left\{\log(\dimtotal)+\yy\right\}} + \frac{2}{3}\delta_{v^{*}}^{2}\left\{\log(\dimtotal)+\yy\right\}
\label{def:deltaVr}
\end{EQA}
and is proportional to \(\sqrt {\{\log(\dimtotal)+\yy\}/n}\) in the case \ref{typical_local}. 
\end{theorem}
 We omit here the proof of Theorem \ref{lemma_ncbi}, since it follows straightforwardly from Theorem 1.4 by \cite{Tropp2012user}, and 
 is already given in \cite{SpZh2014PMB}. 
\subsection{Bootstrap validity for the case of one parametric model}
Here we state the results on bootstrap validity from \cite{SpZh2014PMB}, they will be used for some of the further proofs.

\begin{theorem}
\label{thm:cumulat}
Let the conditions of Section \ref{sect:conditions} be fulfilled, then it holds
 for each \(k=1,\dots,\numK\),
 \(z_{k} \geq \max\{2,\sqrt{\dimp_{k}}\}+\CONST (\dimp_{k}+\yy)/\sqrt{n}\) with probability \(\geq 1-12\ex^{-\yy}\):
\begin{EQA}
	\left|
		\P\left( L_{k}(\thetavt_{k}) - L_{k}(\thetavs_{k}) > z_{k}^{2}/2 \right)
		- \Pb\left( \Lb_{k}(\thetavbt_{k}) - \Lb_{k}(\thetavt_{k}) > z_{k}^{2}/2 \right)
	\right|
	& \leq &
	\Delta_{\full,\,k} \, .
\end{EQA}
%
%
The error term
\(\Delta_{\full,k} \leq \CONST\{(\dimp_{k}+\yy)^{3}/n\}^{1/8}\) in the case 
of i.i.d. model; see Section~\ref{typical_local}.
\end{theorem}
\begin{theorem} [Validity of the bootstrap under a  small modeling bias]
\label{thm:cumulat2}
Assume the conditions of Theorem~\ref{thm:cumulat}. 
Then for \(\alpha \leq 1-8\ex^{-\yy}\), it holds
\begin{EQA}
	\left|
		\P\left( L_{k}(\thetavt_{k}) - L_{k}(\thetavs_{k}) > \left(\zzb_{k}(\alpha)\right)^{2}/2 \right) - \alpha
	\right|
	& \leq &
	\Delta_{\zz,\,\full,\,k} \, .
\label{PLtttsDzzf}
\end{EQA}
The error term \( \Delta_{\zz,\,\full,\,k} \leq \CONST\{(\dimp_{k}+\yy)^{3}/n\}^{1/8}\) in the case 
of i.i.d. model; see Section~\ref{typical_local}.
\end{theorem}

\begin{theorem}
[Performance of the bootstrap for a large modeling bias]
\label{thm:resbias}
Under the conditions of Section \ref{sect:conditions} except for \ref{itm:SmBHk} it holds for \(z_{k}\geq \max\{2,\sqrt{\dimp_{k}}\}+\CONST (\dimp_{k}+\yy)/\sqrt{n}\) with probability \(\geq 1-14\ex^{-\xx}\)
\begin{enumerate}
\item[1.] \quad
         $ {\displaystyle{
         \P\left(L_{k}(\thetavt_{k})-L_{k}(\thetavs_{k}) >z_{k}^{2}/2\right)
\leq
\Pb\left(\Lb_{k}(\thetavbt_{k})-\Lb(\thetavt_{k}) >z_{k}^{2}/2\right)+\DeltabfullIk.
}}$
\vspace{0.1cm}
\item[2.]
$ {\displaystyle{\hspace{0.5cm}
\zzb_{k}(\alpha)\geq
\zz_{k}(\alpha+\DeltabfullIk)
}}$

\vspace{0.12cm}
$ {\displaystyle{\hspace{1cm}
+\,\sqrt{\tr\{D_{k}^{-1}H_{k}^{2}D_{k}^{-1}\}}-\sqrt{\tr\{D_{k}^{-1}(H_{k}^{2}-B_{k}^{2})D_{k}^{-1}\}}
 - \Deltaqqfkk,
}}$\hfill\hfill
\item[]
\vspace{0.3cm}
$ {\displaystyle{ \hspace{0.5cm}
\zzb_{k}(\alpha)\leq
\zz_{k}(\alpha-\DeltabfullIk)
}}$

\vspace{0.2cm}
$ {\displaystyle{\hspace{1cm}
+\,\sqrt{\tr\{D_{k}^{-1}H_{k}^{2}D_{k}^{-1}\}}-\sqrt{\tr\{D_{k}^{-1}(H_{k}^{2}-B_{k}^{2})D_{k}^{-1}\}}
+ \Deltaqlfkk.
}}$\hfill\hfill\hfill
\end{enumerate}
The term \(\DeltabfullIk\leq \CONST\{(\dimp_{k}+\yy)^{3}/n\}^{1/8}\) in the case 
of i.i.d. model; see Section~\ref{typical_local}. 
The positive values \(\Deltaqqfkk, \Deltaqlfkk\) are 
bounded from above with \((\gmu_{k}^{2}+\gmu_{B,k}^{2})(\sqrt{8\yy\dimp_{k}}+6\yy)\) for the constants \(\gmu_{k}^{2}>0,\, \gmu_{B,k}^{2}\geq0\) from  conditions \ref{itm:Ik}, \ref{itm:IBk}.
\end{theorem}

\subsection{Proof of Theorem \ref{thm:mainres1}}
\label{sect:proofsofmainres}
\begin{lemma}[Closeness of \(\LL\left(\|\xiv_{1}\|,\dots,\|\xiv_{\numK}\|\right)\) and \(\LL^{\sbt}\left(\|\xivb_{1}\|,\dots,\|\xivb_{\numK}\|\right)\)]
\label{prop:Cum_norms_xivk}
If the conditions \ref{itm:ED0k}, \ref{itm:Ik}, \ref{itm:SmBHk}, \ref{itm:IBk}, \ref{itm:SD0k} and \ref{itm:Eb} are fulfilled, then it holds with probability \(\geq 1- 6\ex^{-\yy}\) for all  \(\delta_{z_{k}}\geq 0\) and \(z_{k}\geq\sqrt{\dimp_{k}}+\Delta_{\varepsilon}\) s.t. \(\CONST\max\limits_{1\leq k\leq\numK}\{n^{-1/2}, \delta_{z_{k}}\}
\leq
\Delta_{\varepsilon}\leq
\CONST \min\limits_{1\leq k\leq\numK}\{1/z_{k}\}\) (\(\Delta_{\varepsilon}\) is given in \eqref{def:Deltaeps}): 
\begin{EQA}
\P\left(\bigcup\nolimits_{k=1}^{\numK}\left\{\|\xiv_{k}\|>z_{k}\right\} \right)
-
\Pb\left(\bigcup\nolimits_{k=1}^{\numK}\left\{\|\xivb_{k}\|>z_{k}-\delta_{z_{k}}\right\} \right)&\geq& -\Deltapp\,,\\
\P\left(\bigcup\nolimits_{k=1}^{\numK}\left\{\|\xiv_{k}\|>z_{k}\right\} \right)
-
\Pb\left(\bigcup\nolimits_{k=1}^{\numK}\left\{\|\xivb_{k}\|>z_{k}+\delta_{z_{k}}\right\} \right)&\leq& \Deltapp\,.
\end{EQA}
for the deterministic nonnegative value
\begin{EQA}
\Deltapp
&\leq&
25\CONST\left(\frac{\dimp_{\max}^{3}}{n}\right)^{1/8}\log^{9/8}(K)\log^{3/8}(n\dimtotal)
\left\{\left(\bbgmu^{2}+\bbgmu_{B}^{2}\right)
\left(1+\deltabbVb^{2}(\yy)\right)
\right\}^{3/8}.
\end{EQA}
A more explicit bound on \(\Deltapp\) is given in Proposition \ref{thm:Cum_norms}, see also Remark \ref{rem:Deltapp}.
\end{lemma}
\begin{proof}[Proof of Lemma \ref{prop:Cum_norms_xivk}]
The statement follows from Proposition \ref{thm:Cum_norms} and Theorem \ref{lemma_ncbi}.
Let us take \(\phiv_{k}:=\xiv_{k}\) and \(\psiv_{k}:=\xivb_{k}\). Define similarly to \(\Phi\) in \eqref{def:Phi}
\begin{EQA}[c]
\label{def:XiXib}
\Xi\eqdef\left(\xiv_{1}^{\T},\dots,\xiv_{\numK}^{\T}\right)^{\T}
\quad
\Xib\eqdef\left({\xivb_{1}}^{\T},\dots,{\xivb_{\numK}}^{\T}\right)^{\T}.
\end{EQA}
Condition \eqref{cond:deltaSigma} rewrites for \eqref{def:XiXib} as
\begin{EQA}[c]
 \left\|\Var\Xi- \Varb\Xib
 \right\|_{\max}
\leq
\delta_{\Sigma}^{2}
\end{EQA}
for some \(\delta_{\Sigma}^{2}\geq 0\). Denote
\begin{EQA}
\bbD^{2}&\eqdef& \diag\left\{D_{1}^{2},\dots,D_{\numK}^{2}\right\},
\\
\bbV^{2}&\eqdef& \Var \left(\nabla_{\thetav}L_{1}(\thetavs_{1})^{\T},\dots,\nabla_{\thetav}L_{\numK}(\thetavs_{\numK})^{\T}\right)^{\T}.
\end{EQA}
\(\bbD^{2}\) is a block-diagonal matrix and \(\bbV^{2}\) is a block matrix. Both of them are symmetric, positive definite and have the dimension \(\dimtotal\times\dimtotal\).
Let also
\begin{EQA}[c]
\bbVb^{2}\eqdef\Varb \left(\nabla_{\thetav}\Lb_{1}(\thetavs_{1})^{\T},\dots,\nabla_{\thetav}\Lb_{\numK}(\thetavs_{\numK})^{\T}\right)^{\T},
\\
\bbg_{i}\eqdef \left(\nabla_{\thetav}\ell_{i,1}(\thetavs_{1})^{\T},\dots,\nabla_{\thetav}\ell_{i,\numK}(\thetavs_{\numK})^{\T}\right)^{\T}\in \R^{\dimtotal},\\
\bbH^{2}\eqdef
\sum\nolimits_{i=1}^{n}\E \left\{\bbg_{i}\bbg_{i}^{\T}\right\},
\quad \bbB^{2}\eqdef
\sum\nolimits_{i=1}^{n}\E \left\{\bbg_{i}\right\}\E \left\{\bbg_{i}\right\}^{\T}.
\end{EQA}
 It holds
\begin{EQA}[c]
\Var\Xi={\bbD}^{-1}\bbV^{2}\bbD^{-1},\quad
\Varb\Xib=\bbD^{-1}\bbVb^{2}\bbD^{-1},\\
\bbH^{2}=\E\bbVb^{2},\quad \bbV^{2}=\bbH^{2}-\bbB^{2}.
\end{EQA}
Therefore
\begin{EQA}
\nquad
 \left\|\Var\Xi- \Varb\Xib
 \right\|_{\max}&=& \bigl\|\bbD^{-1}\bigl(\bbV^{2}-\bbVb^{2} \bigr)\bbD^{-1}\bigr\|_{\max}
 \\&\leq&
 \bigl\|\bbD^{-1}\bigl(\bbH^{2}-\bbVb^{2} \bigr)\bbD^{-1}\bigr\|_{\max}
 +
 \bigl\|\bbD^{-1}\bbB^{2}\bbD^{-1}\bigr\|_{\max}
 \\
 \label{ineq:usencbi}
 &\leq&
\deltabbVb^{2}(\yy) \bigl\|\bbD^{-1}\bbH^{2}\bbD^{-1}\bigr\|
  +
 \bigl\|\bbD^{-1}\bbB^{2}\bbD^{-1}\bigr\|
\\&\leq&
\bigl\{\deltabbVb^{2}(\yy)
+\DeltaSmB^{2}
\bigr\}(\bbgmu^{2}+\bbgmu_{B}^{2})=:\delta_{\Sigma}^{2}\,.
\label{ineq:SmB_use}
\end{EQA}
Here inequality \eqref{ineq:usencbi} follows from the  matrix Bernstein inequality by \cite{Tropp2012user} (see Section \ref{sect:NCBI_k}).  Inequality \eqref{ineq:SmB_use} is implied by conditions \ref{itm:IBk} and \ref{itm:SmBHk}, and Cauchy-Schwarz inequality.

Condition \ref{itm:A1} of Proposition \ref{thm:Cum_norms} is fulfilled for the vectors \(\xiv_{i,k}\) and \(\xivb_{i,k}\) due to conditions \ref{itm:ED0k}, \ref{itm:Ik} and \ref{itm:SD0k}, \ref{itm:Eb}, \ref{itm:SmBHk}, \ref{itm:IBk} for \(\cf_{\phi}:=\bbgmu\) and \(\cf_{\psi}^{2}:=\bigl(\bbgmu^{2}+\bbgmu_{B}^{2}\bigr)\left\{\delta_{v^{\ast}}^{2}+\max_{1\leq i\leq n}\|\bbH^{-1}\E\left[\bbg_{i}\bbg_{i}^{\T}\right]\bbH^{-1}\|^{2}\right\}\).
\end{proof}

\begin{proof}[Proof of Theorem \ref{thm:mainres1}]
Let us denote \(\yy_{2}\eqdef\yy+\log(\numK)\). It holds with probability \(\geq 1- 12\ex^{-\yy}\)
\begin{EQA}
&&
\Pb\left(\bigcup\nolimits_{k=1}^{\numK}
\Bigl\{
\sqrt{2\Lb_{k}(\thetavbt_{k})-2\Lb_{k}(\thetavt_{k})}>z_{k}\Bigr\}
\right)
\\&\overset{\text{{{L.\,\ref{lemma:simultWilks_gen}}}}}{\geq}&
\Pb\left(\bigcup\nolimits_{k=1}^{\numK}
\Bigl\{
\|\xivb_{k}(\thetavt_{k})\|
\geq z_{k}+\Deltab_{\wind,k}(\rr_{0,k},\yy_{2})
\Bigr\}
\right)
\\&\overset{\text{L.\,\ref{lemma:xivb_simult_gen}}}{\geq}&
\Pb\left(\bigcup\nolimits_{k=1}^{\numK}
\Bigl\{
\|\xivb_{k}(\thetavs_{k})\|
> z_{k}+\Deltab_{\wind,k}(\rr_{0,k},\yy_{2})+\Deltab_{\xi,k}(\rr_{0,k},\yy_{2})
\Bigr\}
\right)
\label{ineq:largesmb1}
\\&\overset{\text{L.\,\ref{prop:Cum_norms_xivk}}}{\geq}&
\P\left(\bigcup\nolimits_{k=1}^{\numK}
\Bigl\{
\|\xiv_{k}\|
> z_{k}-\Delta_{\wind,k}(\rr_{0,k},\yy_{2})
\Bigr\}
\right)-\Delta_{\total}
\\&\overset{\text{{{L.\,\ref{lemma:simultWilks_gen}}}}}{\geq}&
\P\left(\bigcup\nolimits_{k=1}^{\numK}
\Bigl\{
\sqrt{2L_{k}(\thetavt_{k})-2L_{k}(\thetavs_{k})}> z_{k}
\Bigr\}
\right)-\Delta_{\total},
\end{EQA}
for
\begin{EQA}
\label{def:Delta_total_1}
\Delta_{\total}&\eqdef& \Deltapp\,,
\\
\label{def:localzk}
\delta_{z_{k}}&:=&
\Delta_{\wind,k}(\rr_{0,k},\yy+\log(\numK))+
\Deltab_{\wind,k}(\rr_{0,k},\yy+\log(\numK))
\\&&+\,\Deltab_{\xi,k}(\rr_{0,k},\yy+\log(\numK))\\
&\leq&\CONST\frac{\dimp_{k}+\yy+\log(\numK)}{\sqrt{n}}\sqrt{\yy+\log(\numK)}~\text{ in the case \ref{typical_local}}.
\label{ineq:deltazk_proof}
\end{EQA}
Definition of \(\Deltapp\)  is given in Proposition \ref{thm:Cum_norms},  see also Remark \ref{rem:Deltapp}. The bound from Lemma \ref{prop:Cum_norms_xivk} says:
\begin{EQA}
\Deltapp
&\leq&
25\CONST\left(\frac{\dimp_{\max}^{3}}{n}\right)^{1/8}\log^{9/8}(K)\log^{3/8}(n\dimtotal)
\left\{\left(\bbgmu^{2}+\bbgmu_{B}^{2}\right)
\left(1+\deltabbVb^{2}(\yy)\right)
\right\}^{3/8}.
\end{EQA}
For \(\delta_{z_{k}}\) bounded as in \eqref{ineq:deltazk_proof} the conditions \(\CONST\max\limits_{1\leq k\leq\numK}\{n^{-1/2}, \delta_{z_{k}}\}
\leq
\Delta_{\varepsilon}\leq
\CONST \min\limits_{1\leq k\leq\numK}\{1/z_{k}\}\) are fulfilled.
\end{proof}

\subsection{Proof of Theorem \ref{thm:mainres1_2}}
\begin{proof}[Proof of Theorem \ref{thm:mainres1_2}]
For the pointwise quantile functions \(\zz_{k}(\alpha)\) and \(\zzb_{k}(\alpha)\) it holds for each \(k=1,\dots,\numK\) with dominating probability:
\begin{align}
\label{ineq:pointwquant}
\begin{split}
&\zzb_{k}\left(\alpha+\Delta_{\full,\,k}\right)
\leq
\zz_{k}\left(\alpha\right),\\
&\zzb_{k}\left(\alpha\right)
\geq
\zz_{k}\left(\alpha+\Delta_{\full\,k}\right) -\varepsilon_{k}
\end{split}
\end{align}
here \(\Delta_{\full,\,k}\leq \left\{(\dimp_{k}+\yy)^{3}/\sqrt{n}\right\}^{1/8}\), it comes from Theorem \ref{thm:cumulat}, and \(\varepsilon_{k}\leq \CONST(\dimp_{k}+\yy)/\sqrt{n}\),
\begin{EQA}
\varepsilon_{k}
&\eqdef&
\begin{cases}
0,&
\hspace{-5.5cm}
\text{if c.d.f. of }
 L_{k}(\thetavt_{k})-L_{k}(\thetavs_{k}) \text{ is continuous in } \zz_{k}(\alpha+\Delta_{\full,\,k});\\
\CONST(\dimp_{k}+\yy)/\sqrt{n}\ \ \text{s.t. \eqref{eq:zztilde} is fulfilled}
,&\ \ \text{otherwise}.
\end{cases}
\\&&
\hspace{-0.5cm}
\P\Biggl(\sqrt{2\bigl\{ L_{k}(\thetavt_{k})-L_{k}(\thetavs_{k})\bigr\}}>\zz_{k}(\alpha+\Delta_{\full,\,k})-\varepsilon_{k}\Biggr)
\label{eq:zztilde}
\geq
\alpha+\Delta_{\full,\,k}.
\end{EQA}
Indeed, due to Theorem \ref{thm:cumulat} and definition \eqref{def:zzalpha_k}
\begin{EQA}
&&\nquad\nquad
\Pb\left(\sqrt{2\Bigl\{\Lb_{k}(\thetavbt_{k})-\Lb_{k}(\thetavt_{k})\Bigr\}}>\zz_{k}(\alpha)\right)
\\&\leq&
\P\Bigl(\sqrt{2\bigl\{ L_{k}(\thetavt_{k})-L_{k}(\thetavs_{k})\bigr\}}>\zz_{k}(\alpha) \Bigr)+\Delta_{\full,\,k}
\leq\alpha+\Delta_{\full,\,k},
\end{EQA}
therefore, by definition \eqref{def:zzbalpha_k} \(\zzb_{k}(\alpha+\Delta_{\full,\,k}) \leq \zz_{k}(\alpha)\). The lower bound is derived similarly.

If there exist the inverse functions \({\qqq}^{-1}(\cdot)\) and \({\qqqb}^{-1}(\cdot)\), then it holds for \(\beta\in(0,1)\):
\begin{align}
\label{ineq:inv_qqqb}
\begin{split}
\P\left(\bigcup\nolimits_{k=1}^{\numK}
\Bigl\{
\sqrt{2L_{k}(\thetavt_{k})-2L_{k}(\thetavs_{k})}\geq \zz_{k}(\beta)
\Bigr\}
\right)
&\leq\qqq^{-1}(\beta),\\
\Pb\left(\bigcup\nolimits_{k=1}^{\numK}
\Bigl\{
\sqrt{2\Lb_{k}(\thetavbt_{k})-2\Lb_{k}(\thetavt_{k})}\geq \zzb_{k}\left(\beta\right)
\Bigr\}
\right)
&\leq
{\qqqb}^{-1}(\beta).
\end{split}
\end{align}
Therefore, it holds
\begin{EQA}
&&
{\qqqb}^{-1}(\beta+\Delta_{\full,\,\max})
\\&\geq&
\Pb\left(\bigcup\nolimits_{k=1}^{\numK}
\Bigl\{
\sqrt{2\Lb_{k}(\thetavbt_{k})-2\Lb_{k}(\thetavt_{k})}\geq \zzb_{k}\left(\beta+\Delta_{\full,\,k}\right)
\Bigr\}
\right)
\\&\overset{\text{by\,\eqref{ineq:pointwquant}}}{\geq}&
\Pb\left(\bigcup\nolimits_{k=1}^{\numK}
\Bigl\{
\sqrt{2\Lb_{k}(\thetavbt_{k})-2\Lb_{k}(\thetavt_{k})}\geq \zz_{k}\left(\beta\right)
\Bigr\}
\right)
\\&\overset{\text{by Th.\,\ref{thm:mainres1}}}{\geq}&
\P\left(\bigcup\nolimits_{k=1}^{\numK}
\Bigl\{
\sqrt{2L_{k}(\thetavt_{k})-2L_{k}(\thetavs_{k})}\geq \zz_{k}\left(\beta\right)
\Bigr\}
\right)-\Delta_{\total}
\\&\overset{
\substack{
\text{by L.\,\ref{lemma:lik_ac_joint}}\\ \text{and\,\eqref{ineq:inv_qqqb}}
}
}{\geq}&
{\qqq}^{-1}(\beta)-\Delta_{\total}-\DeltaacAT, 
\end{EQA}
here \(\DeltaacAT\leq \Delta_{\total}\) (by Lemma \ref{lemma:lik_ac_joint}) and
\begin{EQA}
\label{def:Delta_full_max}
\Delta_{\full,\,\max}&\eqdef&
\max_{1\leq k\leq \numK}
\Delta_{\full,\,k}
\\&\leq&
\CONST\{(\dimp_{\max}+\yy)^{3}/n\}^{1/8}
\text{ in the case
\ref{typical_local}.}
\end{EQA}
Thus
\begin{EQA}
\label{ineq:inv_qqq_1}
{\qqqb}^{-1}(\beta+\Delta_{\full,\,\max})
&\geq&
{\qqq}^{-1}(\beta)-\Delta_{\total}-\DeltaacAT,\\
 {\qqqb}(\alpha)
 &\leq&
 {\qqq}(\alpha+\Delta_{\total}+\DeltaacAT
 )
 +\Delta_{\full,\,\max}.
\label{ineq:qqq_1}
\end{EQA}
Hence it holds
\begin{EQA}
&&
\P\left(\bigcup\nolimits_{k=1}^{\numK}
\Bigl\{
\sqrt{2L_{k}(\thetavt_{k})-2L_{k}(\thetavs_{k})}\geq \zzb_{k}(\beta)
\Bigr\}
\right)
\\&\overset{\text{by\,\eqref{ineq:pointwquant}}}{\leq}&
\P\left(\bigcup\nolimits_{k=1}^{\numK}
\Bigl\{
\sqrt{2L_{k}(\thetavt_{k})-2L_{k}(\thetavs_{k})}\geq
\zz_{k}\left(\beta+\Delta_{\full,\,k}\right) -\varepsilon_{k}
\right)
\\&\overset{
\substack{
\text{by L.\,\ref{lemma:lik_ac_joint}}\\ \text{and\,\eqref{ineq:inv_qqqb}}
}
}{\leq}&
{\qqq}^{-1}(\beta+\Delta_{\full,\,\max})+\DeltaacAT.
\end{EQA}
Therefore, if \(\qqq(\alpha)\geq\Delta_{\full,\,\max}\), then
\begin{EQA}
\P\left(\bigcup\nolimits_{k=1}^{\numK}
\Bigl\{
\sqrt{2L_{k}(\thetavt_{k})-2L_{k}(\thetavs_{k})}\geq \zzb_{k}(\qqq(\alpha)-\Delta_{\full,\,\max})
\Bigr\}
\right)
&\leq&
\alpha +\DeltaacAT. 
\end{EQA}
And by \eqref{ineq:qqq_1} for \(\qqqb(\alpha)\geq2\Delta_{\full,\,\max}\)  it holds
\begin{EQA}
&&
\P\left(\bigcup\nolimits_{k=1}^{\numK}
\Bigl\{
\sqrt{2L_{k}(\thetavt_{k})-2L_{k}(\thetavs_{k})}\geq \zzb_{k}\left({\qqqb}(\alpha)- 2\Delta_{\full,\,\max}\right)
\Bigr\}
\right)-\alpha
\\
&&\leq
 \Delta_{\total}+2\DeltaacAT.
\end{EQA}

Similarly for the inverse direction:
\begin{EQA}
{\qqqb}^{-1}(\beta)
&\leq&
\Pb\left(\bigcup\nolimits_{k=1}^{\numK}
\Bigl\{
\sqrt{2\Lb_{k}(\thetavbt_{k})-2\Lb_{k}(\thetavt_{k})}\geq \zzb_{k}\left(\beta\right)
\Bigr\}-\varepsilon_{1,k}
\right)
\\&\leq&
\Pb\left(\bigcup\nolimits_{k=1}^{\numK}
\Bigl\{
\sqrt{2\Lb_{k}(\thetavbt_{k})-2\Lb_{k}(\thetavt_{k})}\geq \zz_{k}\left(\beta+\Delta_{\full,\,k}\right) -\varepsilon_{1,k}-\varepsilon_{k}
\Bigr\}
\right)
\\&\leq&
\P\left(\bigcup\nolimits_{k=1}^{\numK}
\Bigl\{
\sqrt{2L_{k}(\thetavt_{k})-2L_{k}(\thetavs_{k})}\geq \zz_{k}\left(\beta+\Delta_{\full,\,k}\right)
\Bigr\}
\right)+\Delta_{\total} + \DeltaacAT
\\&\leq&
{\qqq}^{-1}(\beta+\Delta_{\full,\,\max})+\Delta_{\total}+\DeltaacAT,
\end{EQA}
where \(0\leq\varepsilon_{1,k}\leq \CONST(\dimp_{k}+\yy)/\sqrt{n}\). This implies
\begin{EQA}
{\qqqb}^{-1}(\beta)
&\leq&
{\qqq}^{-1}(\beta+\Delta_{\full,\,\max})+\Delta_{\total}+\DeltaacAT
,\\
{\qqqb}(\alpha)&\geq& {\qqq}\left(\alpha-\Delta_{\total}-\DeltaacAT
\right)
-\Delta_{\full,\,\max}.
\label{ineq:qqq_2}
\end{EQA}
\begin{EQA}
&&
\P\left(\bigcup\nolimits_{k=1}^{\numK}
\Bigl\{
\sqrt{2L_{k}(\thetavt_{k})-2L_{k}(\thetavs_{k})}\geq \zzb_{k}(\beta+\Delta_{\full,\,k})
\Bigr\}
\right)
\\&\overset{\text{by\,\eqref{ineq:pointwquant}}}{\geq}&
\P\left(\bigcup\nolimits_{k=1}^{\numK}
\Bigl\{
\sqrt{2L_{k}(\thetavt_{k})-2L_{k}(\thetavs_{k})}\geq
\zz_{k}\left(\beta\right)
\right)
\\&\geq&
{\qqq}^{-1}(\beta)-\DeltaacAT.
\end{EQA}
\begin{EQA}
\P\left(\bigcup\nolimits_{k=1}^{\numK}
\Bigl\{
\sqrt{2L_{k}(\thetavt_{k})-2L_{k}(\thetavs_{k})}\geq \zzb_{k}(\qqq(\alpha)+\Delta_{\full,\,\max})
\Bigr\}
\right)
&\geq&
\alpha-\DeltaacAT.
\end{EQA}
And by \eqref{ineq:qqq_2}
\begin{EQA}
&&
\P\left(\bigcup\nolimits_{k=1}^{\numK}
\Bigl\{
\sqrt{2L_{k}(\thetavt_{k})-2L_{k}(\thetavs_{k})}\geq \zzb_{k}({\qqqb}(\alpha)+ 2\Delta_{\full,\,\max})
\Bigr\}
\right)-\alpha
\\&&\geq
 -\Delta_{\total}-2\DeltaacAT.
\end{EQA}
for
\begin{EQA}
\label{def:Delta_z_total}
\Delta_{\zz,\,\total}&\eqdef&
\Delta_{\total}+2\DeltaacAT
\leq 3\Delta_{\total}.
\end{EQA}

Conditions of Theorem \ref{thm:mainres1} include \(z_{k}\geq C\sqrt{\dimp_{k}}\), therefore, it has to be checked that \(\zzb_{k}(\alpha)\geq C\sqrt{\dimp_{k}}\).
It holds by Theorem \ref{Wilks_boots}, Proposition \ref{thm:Cum_norms}, Lemmas \ref{lemma_xivb_unif} and \ref{lem:gausslower_B} 
 with probability \(\geq 1-12 \ex^{-\yy}\):
\begin{EQA}
&&\nquad\nquad
\Pb\left(\sqrt{2\bigl\{\Lb_{k}(\thetavbt_{k})-\Lb_{k}(\thetavt_{k})\bigr\}}>
\CONST\sqrt{\dimp_{k}-\sqrt{2\yy\dimp_{k}}} +\CONST (\dimp_{k}+\yy)/\sqrt{n}
\right)
\\&\geq&
1-8\ex^{-\yy}
,
\end{EQA}
Taking \( 1-8\ex^{-\yy} \geq \alpha\), we have
\begin{EQA}[c]
\zzb_{k}(\alpha)\geq \CONST\sqrt{\dimp_{k}-\sqrt{2\yy\dimp_{k}}} +\CONST 2 (\dimp_{k}+\yy)/\sqrt{n}.
\end{EQA}

 Inequalities for \({\qqqb}(\alpha)\) had been already derived in \eqref{ineq:qqq_1} and \eqref{ineq:qqq_2} with
\begin{EQA}
\label{def:Delta_qqq}
\Delta_{\qqq}&\eqdef& \Delta_{\total}+\DeltaacAT.
\end{EQA}
\end{proof}
\begin{lemma}
\label{lemma:lik_ac_joint}
Let the conditions from Section \ref{sect:ConditGeneral} be fulfilled, and the values  \(z_{k}\geq\sqrt{\dimp_{k}}\) and \(\delta_{z_{k}}\geq 0\) be s.t. \(\CONST\max\limits_{1\leq k \leq \numK}\{n^{-1/2}, \delta_{z_{k}}\}
\leq
\Delta_{\varepsilon}\leq
\CONST \min\limits_{1\leq k \leq \numK}\left\{1/z_{k}\right\}\) (\(\Delta_{\varepsilon}\) is given in  \eqref{def:Deltaeps}), then it holds with probability \(\geq 1- 12\ex^{-\yy}\)
\begin{EQA}
&&
\P\left(\bigcup\nolimits_{k=1}^{\numK}
\Bigl\{
\sqrt{2L_{k}(\thetavt_{k})-2L_{k}(\thetavs_{k})}\geq z_{k}
\Bigr\}
\right)
\\&&\quad\quad-
\P\left(\bigcup\nolimits_{k=1}^{\numK}
\Bigl\{
\sqrt{2L_{k}(\thetavt_{k})-2L_{k}(\thetavs_{k})}\geq
z_{k} +\delta_{z_{k}}
\right)\leq \DeltaacAT,
\end{EQA}
where
\begin{EQA}
\DeltappLR&\leq&
12.5\CONST\left(\frac{\dimp_{\max}^{3}}{n}\right)^{1/8}\log^{9/8}(K)\log^{3/8}(n\dimtotal)\bbgmu^{3/4}.
\end{EQA}
\end{lemma}
\begin{proof}[Proof of Lemma \ref{lemma:lik_ac_joint}]
This statement's proof is similar to the one of Theorem \ref{thm:mainres1} (see Section \ref{sect:proofsofmainres}). Here instead of the bootstrap statistics we consider only the values from  the \(\Ym\)-world.
Let us denote \(\yy_{2}\eqdef\yy+\log(\numK)\). It holds with probability \(\geq 1- 12\ex^{-\yy}\)
\begin{EQA}
&&
\P\left(\bigcup\nolimits_{k=1}^{\numK}
\Bigl\{
\sqrt{2L_{k}(\thetavt_{k})-2L_{k}(\thetavs_{k})}> z_{k}
\Bigr\}
\right)
\\&\overset{\text{{{L.\,\ref{lemma:simultWilks_gen}}}}}{\leq}&
\P\left(\bigcup\nolimits_{k=1}^{\numK}
\Bigl\{
\|\xiv_{k}\|
> z_{k}-\Delta_{\wind,k}(\rr_{0,k},\yy_{2})
\Bigr\}
\right)
\\&\overset{\text{{{Pr.\,\ref{thm:Cum_norms}}}}}{\leq}&
\P\left(\bigcup\nolimits_{k=1}^{\numK}
\Bigl\{
\|\xiv_{k}\|
> z_{k}+\delta_{z_{k}}+\Delta_{\wind,k}(\rr_{0,k},\yy_{2})
\Bigr\}
\right)+ \DeltappLR
\\&\leq&
\P\left(\bigcup\nolimits_{k=1}^{\numK}
\Bigl\{
\sqrt{2L_{k}(\thetavt_{k})-2L_{k}(\thetavs_{k})}> z_{k}
+\delta_{z_{k}}
\Bigr\}
\right)
+\DeltappLR\,,
\end{EQA}
where
\begin{EQA}
\DeltappLR&\leq&
12.5\CONST\left({\dimp_{\max}^{3}}/{n}\right)^{1/8}\log^{9/8}(K)\log^{3/8}(n\dimtotal)\bbgmu^{3/4}.
\end{EQA}
Similarly to \eqref{def:Delta_total_1} and \eqref{def:localzk} the term \(\DeltappLR\) is equal to  \(\Deltapp\) from Proposition \ref{thm:Cum_norms} with \(\Delta_{\Sigma}^{2}:=0\), \(\delta_{z_{k}}:=\delta_{z_{k}}+2\Delta_{\wind,k}(\rr_{0,k},\yy+\log(\numK))\).
\end{proof}
\begin{lemma}[Lower bound for deviations of a Gaussian quadratic form]
\label{lem:gausslower_B}
Let \(\phiv\sim \mathcal{N}(0,
\Id_{\dimp})\) and \(\Sigma\) is any symmetric non-negative definite matrix, then
 it holds for any \(\yy>0\)
\begin{EQA}[c]
\P\left(\tr\Sigma-\|\Sigma^{1/2}\phiv\|^{2}\geq
2\sqrt{\yy\tr(\Sigma^{2})}
\right)\leq \exp(-\yy).
\end{EQA}
\end{lemma}
\begin{proof}[Proof of Lemma \ref{lem:gausslower_B}]
 It is sufficient to consider w.l.o.g. only the case of diagonal matrix \(\Sigma\), since it can be represented as \(\Sigma=U^{\T}\diag\{a_{1},\dots,a_{\dimp}\}U\) for an orthogonal matrix \(U\) and the eigenvalues \(a_{1}\geq\dots \geq a_{\dimp}\); \(U\phiv\sim \mathcal{N}(0,\Id_{\dimp})\).

By the exponential Chebyshev inequality it holds for \(\mu>0\), \(\Delta>0\)
\begin{EQA}[c]
\P\left(\tr\Sigma-\|\Sigma^{1/2}\phiv\|^{2}\geq \Delta \right)\leq
\exp(-\mu\Delta/2)\E\exp\left(\mu\left\{\tr\Sigma-\|\Sigma^{1/2}\phiv\|^{2}\right\}/2\right).\\
\log\E\exp\left(\mu\left\{\tr\Sigma-\|\Sigma^{1/2}\phiv\|^{2}\right\}/2\right)\leq
\frac{1}{2}\sum_{j=1}^{\dimp}\left\{\mu a_{j}-\log(1+a_{j}\mu) \right\},
\end{EQA}
therefore
\begin{EQA}
\P\left(\tr\Sigma-\|\Sigma^{1/2}\phiv\|^{2}\geq \Delta \right)
&\leq&
\exp\left(-\frac{1}{2}\left[
\mu\Delta+\sum\nolimits_{j=1}^{\dimp}\left\{\log(1+a_{j}\mu)-\mu a_{j}\right\}
\right]
\right)
\\&\leq&
\exp\left(
-\frac{1}{2}\left[
\mu\Delta-{\mu^{2}}\sum\nolimits_{j=1}^{\dimp}a_{j}^{2}/{2}
\right]
\right)
\\&\leq&
\exp\left(
-\Delta^{2}/\left\{4\sum\nolimits_{j=1}^{\dimp}a_{j}^{2}\right\}
\right).
\end{EQA}
If \(\yy:= \Delta^{2}/\left\{4\sum_{i=1}^{\dimp}a_{j}^{2}\right\}\), then \(\Delta=2\sqrt{\yy\sum_{j=1}^{\dimp}a_{j}^{2}}\).
\end{proof}

\subsection{Proof of Theorem \ref{thm:mainres2_largesmb}}
\begin{proof}[Proof of Theorem \ref{thm:mainres2_largesmb}]
Let us denote \(\yy_{2}\eqdef\yy+\log(\numK)\).
By Lemmas \ref{lemma:simultWilks_gen}, \ref{lemma:xivb_simult_gen} and \ref{prop:Cum_norms_xivk} it holds with probability \(\geq 1- 12\ex^{-\xx}\)
\begin{EQA}
&&
\Pb\left(\bigcup\nolimits_{k=1}^{\numK}
\Bigl\{
\sqrt{2\Lb_{k}(\thetavbt_{k})-2\Lb_{k}(\thetavt_{k})}>z_{k}\Bigr\}
\right)
\\&\geq&
\Pb\left(\bigcup\nolimits_{k=1}^{\numK}
\Bigl\{
\|\xivb_{k}(\thetavs_{k})\|
> z_{k}+\Deltab_{\wind,k}(\rr_{0,k},\yy_{2})+\Deltab_{\xi,k}(\rr_{0,k},\yy_{2})
\Bigr\}
\right)
\label{ineq:lsmb3}
\\&\geq&
\P\left(\bigcup\nolimits_{k=1}^{\numK}
\Bigl\{
\|\tilde{\xiv}_{k}\|
> z_{k}-\Delta_{\wind,k}(\rr_{0,k},\yy_{2})
\Bigr\}
\right)-\Delta_{\operatorname{b},\,\total}
\label{ineq:lsmb1}
\\&\geq&
\label{ineq:lsmb2}
\P\left(\bigcup\nolimits_{k=1}^{\numK}
\Bigl\{
\|\xiv_{k}\|
> z_{k}-\Delta_{\wind,k}(\rr_{0,k},\yy_{2})
\Bigr\}
\right)-\Delta_{\operatorname{b},\,\total}
\\&\geq&
\P\left(\bigcup\nolimits_{k=1}^{\numK}
\Bigl\{
\sqrt{2L_{k}(\thetavt_{k})-2L_{k}(\thetavs_{k})}> z_{k}
\Bigr\}
\right)-\Delta_{\operatorname{b},\,\total},
\end{EQA}
here \(\tilde{\xiv}_{k}\eqdef \bigl(D_{k}^{-1}H_{k}^{2}D_{k}^{-1}\bigr)^{1/2}(\Var\xiv_{k})^{-1/2}\xiv_{k}\), and \(\Delta_{\operatorname{b},\total}\) is given below. Using the same notations as in the proof of Lemma \ref{prop:Cum_norms_xivk}, we have
\begin{EQA}
\tilde{\Xi}&\eqdef&\left(\tilde{\xiv}_{1}^{\T},\dots,\tilde{\xiv}_{\numK}^{\T}\right)^{\T}
\\
&=& \bigl(\bbD^{-1}\bbH^{2}\bbD^{-1}\bigr)^{1/2}(\Var\Xi)^{-1/2}
\Xi,
\end{EQA}
and by Theorem \ref{lemma_ncbi} and by conditions \ref{itm:Ik}, \ref{itm:IBk}, it holds with probability \(\geq 1-\ex^{-\yy}\)
\begin{EQA}
\bigl\|\Var\tilde{\Xi}- \Varb\Xib
 \bigr\|_{\max}&=& \bigl\|\bbD^{-1}\bigl(\bbH^{2}-\bbVb^{2} \bigr)\bbD^{-1}\bigr\|_{\max}
\\ &\leq&
\deltabbVb^{2}(\yy) \bigl\|\bbD^{-1}\bbH^{2}\bbD^{-1}\bigr\|_{\max}
\\ &\leq& \deltabbVb^{2}(\yy) (\bbgmu^{2}+\bbgmu_{B}^{2}).
\end{EQA}
Thus, inequality \eqref{ineq:lsmb1} follows from Proposition \ref{thm:Cum_norms} applied to the sets of vectors \(\xivb_{1}(\thetavs_{1}),\dots,\xivb_{\numK}(\thetavs_{\numK})\) and \(\tilde{\xiv}_{1},\dots,\tilde{\xiv}_{\numK}\). The error term \(\Delta_{\operatorname{b},\total}\) is equal to \(\Delta_{\total}\) from Theorem \ref{sect:proofsofmainres} (see \eqref{def:Delta_total_1}, \eqref{def:localzk}) with \(\DeltaSmB^{2}:=0\), thus
\begin{EQA}
\label{ineq:Delta_b_total}
\Delta_{\operatorname{b},\total}&\leq& 25\CONST\left(\frac{\dimp_{\max}^{3}}{n}\right)^{1/8}\log^{9/8}(K)\log^{3/8}(n\dimtotal)
\left\{\left(\bbgmu^{2}+\bbgmu_{B}^{2}\right)
\left(1+\deltabbVb^{2}(\yy)\right)
\right\}^{3/8}.
\end{EQA}
Inequality \eqref{ineq:lsmb2} is implied by definitions of \(\tilde{\xiv}_{k}\) and matrices \(H_{k}^{2}, V_{k}^{2}\), indeed:
\begin{EQA}
&&
\left\|\bigl(D_{k}^{-1}H_{k}^{2}D_{k}^{-1}\bigr)^{-1/2}\Var\xiv_{k}\bigl(D_{k}^{-1}H_{k}^{2}D_{k}^{-1}\bigr)^{-1/2}\right\|
\\&\leq&
\left\|\bigl(D_{k}^{-1}H_{k}^{2}D_{k}^{-1}\bigr)^{1/2}
\bigl(D_{k}
H_{k}^{-2}
V_{k}^{2}
H_{k}^{-2}
D_{k}\bigr)
\bigl(D_{k}^{-1}H_{k}^{2}D_{k}^{-1}\bigr)^{1/2}\right\|
\\&\leq& 1,
\end{EQA}
therefore, \(\|\tilde{\xiv}_{k}\|\geq \|\xiv_{k}\|\).

The second inequality in the statement is proven similarly to \eqref{ineq:qqq_1}. It implies together with Theorem \ref{thm:resbias} the rest part of the statement having
\begin{EQA}
\label{def:Delta_b_qqq}
\Delta_{\operatorname{b},\qqq}&\eqdef& \Delta_{\operatorname{b},\,\total}+\DeltaacAT.
\end{EQA}
\end{proof}

\phantomsection
\addcontentsline{toc}{section}{References}
\bibliography{exp_ts,listpubm-with-url,references_1}

\end{document}